\documentclass[10pt]{amsart}
\usepackage{amsmath,amsthm,amsfonts,amssymb,amscd,flafter,pinlabel}
\usepackage{mathtools}
\usepackage[mathscr]{eucal}
\usepackage{graphics}
 \usepackage[all]{xy}
\usepackage{caption, subcaption}
\usepackage[usenames,dvipsnames]{color}
\usepackage{graphicx}
\usepackage{enumitem}
\usepackage{showlabels}

\usepackage{pgf}
\usepackage{tikz}
\usetikzlibrary{arrows,automata}
\usepackage{inputenc}

\usepackage[colorlinks=true, urlcolor=NavyBlue, linkcolor=NavyBlue, citecolor=NavyBlue, pdftitle={Bridge trisections in rational surfaces}, pdfauthor={Peter Lambert-Cole, Jeffrey Meier}, pdfsubject={trisections,four-manifolds}, pdfkeywords={trisections,four-manifolds}]{hyperref}

\newtheorem{theorem}{Theorem}[section]
\newtheorem{corollary}[theorem]{Corollary}
\newtheorem{conjecture}[theorem]{Conjecture}
\newtheorem{question}[theorem]{Question}
\newtheorem{proposition}[theorem]{Proposition}

\newtheorem{lemma}[theorem]{Lemma}

\theoremstyle{definition}
\newtheorem{definition}[theorem]{Definition}
\newtheorem {example}[theorem]{Example}

\theoremstyle{definition}
\newtheorem{remark}[theorem]{Remark}

\newtheorem{problem}[theorem]{Problem}

\newcommand{\Z}{\ensuremath{\mathbb{Z}}}

\def\x{\mathbf{x}}

\def\d{\mathbf{d}}

\newcommand\alphas{\mbox{\boldmath$\alpha$}}
\newcommand\betas{\mbox{\boldmath$\beta$}}
\newcommand\gammas{\mbox{\boldmath$\gamma$}}
\newcommand\taus{\mbox{\boldmath$\tau$}}
\newcommand\Deltas{\mbox{\boldmath$\Delta$}}

\newcommand{\CC}{\mathbb{C}}
\newcommand{\RR}{\mathbb{R}}
\newcommand{\del}{\partial}

\newcommand{\ZZ}{\mathbb{Z}}

\newcommand{\cD}{\mathcal{D}}
\newcommand{\cL}{\mathcal{L}}
\newcommand{\cJ}{\mathcal{J}}

\newcommand{\cT}{\mathcal{T}}

\newcommand{\cA}{\mathcal{A}}
\newcommand{\cC}{\mathcal{C}}
\newcommand{\cK}{\mathcal{K}}

\newcommand{\cX}{\mathcal{X}}
\newcommand{\cQ}{\mathcal{Q}}
\newcommand{\cE}{\mathcal{E}}
\newcommand{\cH}{\mathcal{H}}
\newcommand{\cB}{\mathcal{B}}
\newcommand{\cV}{\mathcal{V}}

\newcommand{\NN}{\mathbb{N}}

\newcommand{\CP}{\mathbb{CP}}

\begin{document}

\title[{Bridge trisections in rational surfaces}]{Bridge trisections in rational surfaces}

\author[P. Lambert-Cole]{Peter Lambert-Cole}
\address{School of Mathematics \\ Georgia Institute of Technology}
\email{plc@math.gatech.edu}
\urladdr{\href{http://people.math.gatech.edu/~plambertcole3/}{\url{https://people.math.gatech.edu/~plambertcole3/}}}

\author[J. Meier]{Jeffrey Meier}
\address{Department of Mathematics \\ University of Georgia \\ Athens, GA 30602}
\email{jeffrey.meier@uga.edu}
\urladdr{\href{http://jeffreymeier.org}{jeffreymeier.org}}

% \date{\today}
%\keywords{four-manifold, complex surface, trisection, bridge trisection, branched cove}
%\subjclass[2010]{}
\maketitle

%%%%%%%%%%%%%%%%%%%%%%%%%%%%%%%%%%%%%%%%%%%%%%%%%%%%%%%

%%%%%%%%%%%%%%%%%%%%%%%%%%%%%%%%%%%%%%%%%%%%%%%%%%%%%%%
\begin{abstract}
%%%%%%%%%%%%%%%%%%%%%%%%%%%%%%%%%%%%%%%%%%%%%%%%%%%%%%%

We study smooth isotopy classes of complex curves in complex surfaces from the perspective of the theory of bridge trisections, with a special focus on curves in $\CP^2$ and $\CP^1\times\CP^1$.  We are especially interested in bridge trisections and trisections that are as simple as possible, which we call \emph{efficient}.  We show that any curve in $\CP^2$ or $\CP^1\times \CP^1$ admits an efficient bridge trisection.  Because bridge trisections and trisections are nicely related via branched covering operations, we are able to give many examples of complex surfaces that admit efficient trisections.  Among these are hypersurfaces in $\CP^3$, the elliptic surfaces $E(n)$, the Horikawa surfaces $H(n)$, and complete intersections of hypersurfaces in $\CP^N$.  As a corollary, we observe that, in many cases, manifolds that are homeomorphic but not diffeomorphic have the same trisection genus, which is consistent with the conjecture that trisection genus is additive under connected sum.  We give many trisection diagrams to illustrate our examples.

%%%%%%%%%%%%%%%%%%%%%%%%%%%%%%%%%%%%%%%%%%%%%%%%%%%%%%%	
\end{abstract}
%%%%%%%%%%%%%%%%%%%%%%%%%%%%%%%%%%%%%%%%%%%%%%%%%%%%%%%

% Make Table of Contents
%\setcounter{tocdepth}{2}
%\let\oldtocsubsection=\tocsubsection
%\renewcommand{\tocsubsection}[2]{\hspace{2em}\oldtocsubsection{#1}{#2}}
%\tableofcontents

%%%%%%%%%%%%%%%%%%%%%%%%%%%%%%%%%%%%%%%%%%%%%%%%%%%%%%%
%%%%%%%%%%%%%%%%%%%%%%%%%%%%%%%%%%%%%%%%%%%%%%%%%%%%%%%
\section{Introduction}\label{sec:intro}
%%%%%%%%%%%%%%%%%%%%%%%%%%%%%%%%%%%%%%%%%%%%%%%%%%%%%%%
%%%%%%%%%%%%%%%%%%%%%%%%%%%%%%%%%%%%%%%%%%%%%%%%%%%%%%%

The study of simply-connected, smooth four-manifolds is an area of active research with a long history.  In 1964, Wall proved that simply-connected four-manifolds with isomorphic quadratic forms are $h$--cobordant and become diffeomorphic after connected summing with copies of $S^2\times S^2$~\cite{Wall}.  In his groundbreaking 1982 work, Freedman showed that such four-manifolds are homeomorphic~\cite{Freedman}, leaving open the possibility that such manifolds could be homeomorphic, but not diffeomorphic.  This possibility was soon shown to be a reality when Donaldson revolutionized four-manifold topology with the introduction of his gauge theoretic invariants.  In particular, he showed that the degree $d$ hypersurface $S_d$ in $\CP^3$ cannot have $\CP^2$ as a connected summand for odd $d\geq 5$, though it is homeomorphic to a four-manifold that can, by Freedman~\cite{Donaldson}.  The subtlety of the situation is further exposed by a result of Mandelbaum and Moishezon, which shows that these four-manifolds become diffeomorphic after connected summing with a single copy of $\CP^2$~\cite{MM2}.

One of the main goals of this paper is to explore how the theory of trisections, which was introduced by Gay and Kirby in 2016~\cite{GK}, behaves when applied to simply-connected four-manifolds, particularly complex surfaces.

A \emph{trisection} of a smooth four-manifold $X$ is a decomposition $X = Z_1\cup Z_2\cup Z_3$ such that
\begin{enumerate}
	\item Each $Z_i$ is a four-dimensional 1--handlebody;
	\item Each intersection $H_i = Z_{i-1}\cap Z_{i}$ is a three-dimensional handlebody; and
	\item The common intersection $\Sigma = Z_1\cap Z_2\cap Z_3$ is a closed surface.
\end{enumerate}
The surface $\Sigma$ is called the \emph{core} of the trisection, and its genus is called the \emph{genus} of the trisection.  The \emph{trisection genus} $g(X)$ of a four-manifold $X$ is the minimum value of $g$ such that $X$ admits a trisection of genus $g$.  See Section~\ref{sec:trisections} for more details.

The theory of bridge trisections was introduced as an adaptation of the theory of trisections to the setting of knotted surfaces in four-manifolds.  In the present paper, we will be interested in studying complex curves in complex surfaces using bridge trisections.  Note that throughout we will be studying such objects up to smooth isotopy and/or diffeomorphism; we think of the complex geometry as a natural starting point for a more general study of knotted surfaces.

Given a knotted surface $\cK$ in a four-manifold $X$ and a trisection $\cT$ of $X$,  we say that $\cK$ is in \emph{bridge trisected position} with respect to $\cT$ if
\begin{enumerate}
	\item $\Sigma\cap\cK$ is a collection of points;
	\item $B_i\cap\cK$ is a collection of arcs that can be isotoped rel-$\partial$ to lie in $\partial B_i$; and
	\item $Z_i\cap\cK$ is a collection of disks that can be isotoped rel-$\partial$ to lie in $\partial Z_i$.
\end{enumerate}

The induced decomposition of the pair $(X,\cK)$ is called a \emph{bridge trisection}.  For now, we assume that the number of disk components of $Z_i\cap\cK$ is the same for each $i\in\Z_3$, and we refer to this number as the \emph{patch number} of the bridge trisection. See Section~\ref{sec:trisections} for more details.

%%%%%%%%%%%%%%%%%%%%%%%%%%%%%%%%%%%%%%%%%%%%%%%%%%%%%%%
\subsection{Efficient decompositions of simply-connected four-manifolds}
%%%%%%%%%%%%%%%%%%%%%%%%%%%%%%%%%%%%%%%%%%%%%%%%%%%%%%%

There are natural lower bounds on the trisection genus of $X$ coming from the algebraic topology of $X$.  For example, when $X$ is simply-connected, we have that $g(X) \geq b_2(X)$.  We call a trisection of a simply-connected four-manifold $X$ \emph{efficient} if it has genus equal to $b_2(X)$.  In this case, the pieces $Z_i$ of the trisection are all four-balls.  It follows that $X$ admits a handle-decomposition with neither 1--handles nor 3--handles.  It is an open question (Kirby Problem~4.18~\cite{Kirby}) whether or not every simply-connected four-manifold can be built without 1--handles.  Note that the existence of an efficient trisection for a simply-connected four-manifold is strictly stronger than the existence of a handle-decomposition with neither 1--handles nor 3--handles; it says further  that there is a such a handle-decomposition in which the attaching link of the 2--handles has minimal possible tunnel number~\cite{MSZ}.

In this direction, we have the following theorem, which shows that many examples of simply-connected complex surfaces admit efficient trisections.

\begin{theorem}
\label{thmx:eff_tri}
	The following four-manifolds admit efficient trisections.
	\begin{enumerate}
		\item The Kummer quartic surface, $K3$.
		\item The complex hypersurface $S_d$ of degree $d$ inside $\CP^3$.
		\item The Horikawa surfaces $H(n)$.
		\item The elliptic surfaces $E(n)$.
		\item Every complex surface obtained as a cyclic branched cover of $\CP^2$ or $\CP^1 \times \CP^1$ along a smooth, connected complex curve.
		\item The complete intersection $S_\bold d$ inside $\CP^{n+2}$ corresponding to the multi-index $\bold d = (d_1,\ldots, d_{n})$.
	\end{enumerate}
\end{theorem}

Note that Spreer and Tillmann recently determined that $K3$ admits an efficient genus 22 trisection as well~\cite{Spreer-Tillmann}. Note that the examples in the above theorem were known to admit handle-decompositions built with no 1--handles nor 3--handles~\cite{AK,Fuller,Man2}.  With all this in mind, we offer the following conjecture to motivate further investigation.

\begin{conjecture}
\label{conj:eff_tri}
	Every simply-connected, complex four-manifold admits an efficient trisection.
\end{conjecture}

Note that so far we have only discussed efficiency for simply-connected four-manifolds, though a natural extension of the concept is available.  See Section~\ref{sec:branch} for details.

%%%%%%%%%%%%%%%%%%%%%%%%%%%%%%%%%%%%%%%%%%%%%%%%%%%%%%%
\subsection{Efficient decomposition of surface-links}
%%%%%%%%%%%%%%%%%%%%%%%%%%%%%%%%%%%%%%%%%%%%%%%%%%%%%%%

In contrast to the historical interest in minimizing the complexity of handle-decompositions of well-known simply-connected four-manifolds, a systematic study of the complexity of decompositions of knotted surfaces in four-manifolds seems absent.  Bridge trisections provide a natural way to initiate such a systematization.

Suppose that $\cT$ is an efficient trisection of a simply-connected four-manifold $X$ and $\cK$ is in bridge position with respect to $\cT$.  We call the induced bridge trisection \emph{efficient} if it has patch number one.  Since the patch number of a bridge trisection is bounded below by the meridional rank of the fundamental group of the exterior of the knotted surface, we will mostly restrict attention to the case when this group is cyclic. (See Section~\ref{sec:trisections} for more general formulations.)

The foundational result of this paper is that many familiar examples of knotted surfaces coming from complex topology admit efficient bridge trisections.

\begin{theorem}
\label{thmx:eff_bridge}
	Each of the following complex curves admits an efficient bridge trisection.
	\begin{enumerate}
		\item The curve $\cC_d$ of degree $d$ in $\CP^2$, for all $d\in\Z$.
		\item The curve $\cC_{(a,b)}$ of bidegree $(a,b)$ in $\CP^1\times\CP^1$, for all $(a,b)\in\Z\oplus\Z$.
		\item The curve $\cH_{d_0}$ representing $d_0$ times a hyperplane section inside the complete intersection $S_\bold d$ corresponding to the multi-index $\bold d = (d_1,\ldots, d_n)$.
		\item The generic fiber $\cE$ of the elliptic fibration $E(n)$.
	\end{enumerate}
\end{theorem}

A key ingredient throughout the paper is a detailed understanding of how bridge trisections change under three common operations: branched covering, resolution of singular knotted surfaces, and blowing up.  We find that efficiency can be preserved in each case.

\begin{theorem}
\label{thmx:eff_branch}
	Suppose that a surface-link $(X,\cK)$ admits an efficient bridge trisection.  Then, the following related surface-links all admit efficient bridge trisections.
	\begin{enumerate}
		\item The lift $\widetilde\cK$ of $\cK$ in the the $n$--fold cyclic cover $\widetilde X$ of $X$, branched along $\cK$.
		\item The smooth resolution $(X,\cJ)$ of the singular surface-link $(X,m\cK)$ obtained by taking $m$ parallel copies of $\cK$.
		\item The proper transform $(X\#\overline\CP^2,\cK\#\overline\CP^1)$.
	\end{enumerate}
\end{theorem}

These results about efficient bridge trisections enable us to obtain the results about efficient (four-manifold) trisection described in the preceding subsection, as well as the results described in the ensuing subsection.

\begin{conjecture}
\label{conj:eff_bridge}
	Suppose $\cK$ is a complex curve in a simply-connected, complex surface $X$.  Suppose that the fundamental group of $X\setminus\cK$ is cyclic.  Then, $(X,\cK)$ admits an efficient bridge trisection.
\end{conjecture}

As we noted above, there is no evidence that restrictions to simply-connected four-manifolds, complex four-manifolds, or knotted surfaces with cyclic group are required.  With the proper generalization of the notion of efficiency, this conjecture could be massively strengthened.  We have formulated it in the present setting for the sake of simplicity.

%%%%%%%%%%%%%%%%%%%%%%%%%%%%%%%%%%%%%%%%%%%%%%%%%%%%%%%%%%%%%%%%%%%%%%%
\subsection{Exotic four-manifolds and additivity of trisection genus}
\label{subsec:exotic}
%%%%%%%%%%%%%%%%%%%%%%%%%%%%%%%%%%%%%%%%%%%%%%%%%%%%%%%%%%%%%%%%%%%%%%%

It is a straight-forward exercise to verify that a four-manifold $X = X_1\# X_2$ inherits a natural trisection $\cT$ from trisections $\cT_1$ and $\cT_2$ on $X_1$ and $X_2$ such that $g(\cT) = g(\cT_1)+g(\cT_2)$.  It follows that $g(X_1\#X_2)\leq g(X_1)+g(X_2)$.  An important motivating conjecture in the theory of trisections is that the converse holds.

\begin{conjecture}[Additivity Conjecture]
\label{conj:add}
	Trisection genus is additive under connected sum: For any two four-manifolds $X_1$ and $X_2$ we have
	$$g(X_1\#X_2) = g(X_1) + g(X_2).$$
\end{conjecture}

In dimension three, the analogous result holds: Heegaard genus of three-manifolds is additive under connected sum~\cite{Haken}.  In addition to being a foundational question within the theory of trisections, Conjecture~\ref{conj:add} would have serious implications in four-manifold topology, should it be shown to be true. Recall that, given a four-manifold $X$, an \emph{exotic $X$} is a four-manifold $X'$ that is homeomorphic to $X$, but not diffeomorphic to $X$.  In this case, $X$ and $X'$ are called an \emph{exotic pair}.

\begin{proposition}
\label{prop:add}
	If Conjecture~\ref{conj:add} is true, then trisection genus is a homeomorphism invariant.  Consequently, there is no exotic $S^4$, $\CP^2$, $S^1\times S^3$, $S^2\times S^2$, $\CP^2\#\CP^2$, nor $\CP^2\#\overline\CP^2$.
\end{proposition}

\begin{proof}
	Suppose $X$ and $X'$ are an exotic pair.  By a theorems of Wall, in the simply-connected case, and Gompf, in the general case, $X\#S$ and $X'\#S$ are diffeomorphic, where $S\cong\#^n(S^2\times S^2)$~\cite{Gompf,Wall}.  If trisection genus is additive, then since $X\#S$ and $X'\#S$ are diffeomorphic, we have that $g(X) = g(X')$.
	
	The manifolds comprise an exact list of those manifold admitting trisections of genus at most two~\cite{GK,MZ-Genus-2}, so no exotic version can exist.
\end{proof}

While this proposition illuminates the promise of Conjecture~\ref{conj:add}, it also describes the most feasible way to disprove this conjecture.  Namely, one should investigate when exotic pairs admit trisections of the same genus.  Since complex surfaces turn out to give many examples of exotic copies of standard manifolds, this article provides the first step in this program.  For example,  when $d\geq 5$, the complex hypersurface $S_d$ of degree $d$ in $\CP^3$ is homeomorphic (but not diffeomorphic) to either a connected sum of copies of $\CP^2$ and $\overline\CP^2$ (if $d$ is odd) or a connected sum of copies of $K3$ and $S^2\times S^2$ if $d$ is even). The existence of these homeomorphisms follows from the topological data of these surfaces (see Proposition~\ref{prop:hypersurface-data} below), together with Freedman's work~\cite{Freedman}. For the non-existence of corresponding diffeomorphisms, see~\cite{Donaldson,taubes}.

By Theorem~\ref{thmx:eff_tri}(1), we know that $K3$ admits an efficient trisection, so it follows that all of these standard connected sums admit efficient trisections.  (This is Theorem~2 of~\cite{Spreer-Tillmann}.)  However, by Theorem~\ref{thmx:eff_tri}(2), all of the $S_d$ admit efficient trisections.  Thus, we have the following corollary.

\begin{corollary}
\label{corox:exotic}
	There are infinitely many exotic pairs $X$ and $X'$ such that $g(X) = g(X')$.
\end{corollary}

This corollary provides supporting evidence for Conjecture~\ref{conj:add}.  It is interesting to observe that there is a serendipitous convergence of results in the literature.  On one hand, four-manifolds with trisection genus at most two have been classified; so the first place to look for infinitely many simply-connected four-manifolds with the same trisection genus would be at genus three.  In particular, this would be the first place where exotic pairs with the same trisection genus could show up.  On the other hand, as techniques have been refined over the years, experts have been able to give exotic pairs of simply-connected four-manifolds with smaller and smaller second Betti number.  The limit of present technology is $b_2 = 3$~\cite{AkhPark,FinStern}.  Thus, techniques from the theory of trisections and those of exotic manifold theory seem to have converged.

If experts can push past this $b_2=3$ limit, it will disprove Conjecture~\ref{conj:add} (in addition to being a stunning result in its own right, of course).  On the other hand, if Conjecture~\ref{conj:add} is to be true, then there should be some very interesting four-manifolds that admit genus three trisections.  We consider the following problem to be central to the development of the theory of trisections.  (Compare with~\cite{M_spun} in the non-simply connected case.)

\begin{problem}
\label{prob:three}
	Classify those simply-connected, four-manifolds admitting trisections of genus three.
\end{problem}

%%%%%%%%%%%%%%%%%%%%%%%%%%%%%%%%%%%%%%%%%%%%%%%%%%
\subsection{Complete reducibility}
\label{subsec:dissolves}
%%%%%%%%%%%%%%%%%%%%%%%%%%%%%%%%%%%%%%%%%%%%%%%%%%
 
A smooth four-manifold $X$ \emph{smoothly} (resp., \emph{topologically}) \emph{dissolves} (or is \emph{completely decomposable}) if $X$ is diffeomorphic (resp., homeomorphic) to $p \CP^2 \# q \overline{\CP}^2$ for some $p,q\in\mathbb N_0$. A trisection of $p \CP^2 \# q \overline{\CP}^2$ \emph{dissolves} (or is \emph{completely reducible}) if it is a connected sum of genus one trisections.  Little is known about the space of trisections of a given four-manifold.  A good question in this direction is the following, which has been answered in the affirmative for $p+q\leq 2$~\cite{MZ-Genus-2}.

\begin{question}
	Is every $(p+q,0)$ trisection of $p \CP^2 \# q \overline{\CP}^2$ completely reducible?
\end{question}

More generally, a four-manifold $X$ is \emph{almost completely decomposable} if $X\#\CP^2$ dissolves.  Mandelbaum and Moishezon have shown that the class of complex surfaces that is almost completely decomposable includes hypersurfaces in $\CP^3$~\cite{MM1} and, more generally, complete intersections~\cite{MM2}. See Section~\ref{sec:complete} below for a discussion of complete intersections. Note that the blowups $S_d \# k \overline{\CP}^2$ \emph{never} dissolve smoothly for $d \geq 4$.

Many of the surfaces discussed in the present paper are therefore almost completely decomposable.  We say that a trisection is \emph{almost completely reducible} if it becomes completely reducible after connected summing with the genus one trisection for $\CP^2$.

\begin{question}
	Which of the trisections constructed in the present paper are almost completely reducible?
\end{question}

%%%%%%%%%%%%%%%%%%%%%%%%%%%%%%%%%%%%%%%%%%%%%%%%%%%%%%%%
\subsection{Trisections of K3}
\label{subsec:inequivalent}
%%%%%%%%%%%%%%%%%%%%%%%%%%%%%%%%%%%%%%%%%%%%%%%%%%%%%%%%

One consequence of the techniques of this paper is that we are able to construct efficient $(22,0)$--trisections of $K3$ in 13 different ways. 

\begin{theorem}
\label{thmx:K3}
The following constructions give efficient $(22,0)$--trisections of K3:
\begin{enumerate}
\item the 2--fold branched cover of $(S^2 \times S^2, \cC_{4,4})$.
\item the 2--fold branched cover of $(\CP^2 \# n \overline{\CP}^2,\widetilde{\cE}_2)$ for $1 \leq n \leq 8$.
\item the 3--fold branched cover of the $(S^2 \times S^2, \cC_{3,3})$.
\item the 2--fold branched cover of $(\CP^2, \cC_{6})$.
\item the 4--fold branched cover of $(\CP^2, \cC_{4})$.
\item the 2--fold branched cover of $(S(2,2),\cH_2)$.
\end{enumerate}
The following construction gives an inefficient $(25,1)$--trisection of K3
\begin{enumerate}[resume]
\item the 2--fold branched cover of $(\CP^2 \# 9 \overline{\CP}^2, \cE_2)$.
\end{enumerate}
\end{theorem}

It is not immediately clear whether any of these are diffeomorphic as trisections.  We therefore refrain from referring to any of them as `standard'.  However, we have no invariant to distinguish them.  In particular, the $K3$ surface is simply-connected and so we cannot adapt recent work of Islambouli to distinguish them via the Nielsen equivalence classes of the generators of the fundamental group~\cite{Islambouli}.

\begin{question}
Which of these $(22,0)$--trisections of $K3$ are equivalent?  Are any equivalent to the $(22,0)$--trisection discovered by Spreer and Tillmann \cite{Spreer-Tillmann}?
\end{question}

%%%%%%%%%%%%%%%%%%%%%%%%%%%%%%%%%%%%%%%%%%%%%%%%%%%%%%%%
\subsection{Stein trisections}
\label{subsec:stein}
%%%%%%%%%%%%%%%%%%%%%%%%%%%%%%%%%%%%%%%%%%%%%%%%%%%%%%%%

One motivation for the present work is to understand the connection between trisections and complex and symplectic geometry.  Some progress has been made by Gay, who constructed trisections from Lefschetz pencils~\cite{Gay}; it is a well-known result of Donaldson that every symplectic 4--manifold admits a Lefschetz pencil~\cite{Donaldson-Lefschetz}.

A natural approach to to finding such a connection is to try to impose compatible geometric structures on the pieces of a trisection. Recall that a paracompact, complex manifold is {\it Stein} if
\begin{enumerate}
	\item it is holomorphically convex;
	\item global holomorphic functions separate points; and
	\item in a neighborhood of each point, there are global holomorphic functions that form a local coordinate system.
\end{enumerate}

A $(g;k_1,k_2,k_3)$--{\it Stein trisection} of a complex surface $X$ is a collection of three Stein domains $Z_1, Z_2, Z_3$ such that
\begin{enumerate}
\item $Z_1,Z_2,Z_3$ are an open cover of $X$.
\item $Z_i$ is diffeomorphic to $\natural^{k_i} S^1 \times B^3$,
\item $Z_i \cap Z_{i+1}$ is diffeomorphic to $\natural^{g} S^1 \times B^3$,
\item $Z_1 \cap Z_2 \cap Z_3$ is diffeomorphic to $\Sigma_g \times D^2$.
\end{enumerate}

The 1--handlebodies $\natural^{k_i}(S^1\times B^3)$ always admit Stein structures, so the key feature of this definition is that the complex analytic structures on the sectors agree on the double and triple intersections.  In the case of $\CP^2$, we can find a Stein trisection explicitly.

\begin{proposition}
\label{propx:Stein_CP2}
	$\CP^2$ admits a $(1,0)$--Stein trisection.
\end{proposition}

It is a classical result that if $X$ is Stein and if $f\colon Y \rightarrow X$ is a finite holomorphic map, then $Y$ is also Stein.  Furthermore, for every projective surface $S$ we can find a branched covering map $f\colon S \rightarrow \CP^2$.  Thus, we can pull back the Stein trisection of $\CP^2$ to find a covering of $S$ by three Stein manifolds of unknown topology.  While the curves $\cC_d$ can be put into bridge position via a smooth isotopy, for $d \geq 2$ it is not immediately clear whether this is possible {\it as a complex curve}.  Thus, we are led to the following question.

\begin{question}
	Let $X$ be a projective complex surface.  Does $X$ admits a Stein trisection?
\end{question}

For symplectic four-manifolds, we also introduce an analogous notion of a {\it Weinstein trisection}.  Recall that a {\it Weinstein structure} on an open manifold $W$ is a triple $(\omega, \phi, V)$, where $\omega$ is a symplectic form, $\phi$ is an exhausting generalized Morse function and $V$ is a complete vector field that is Liouville for $\omega$ and gradient-like for $\phi$.  A $(g;k_1,k_2,k_3)$--{\it Weinstein trisection} of $(X,\omega)$ consists of
\begin{enumerate}
\item an open cover $X = Z_1 \cup Z_2 \cup Z_3$, where $Z_i$ is diffeomorphic to $\natural^{k_i} S^1 \times D^3$ and the double and triple intersections satisfy the same topological conditions as in a Stein trisection, and
\item a Weinstein structure $(\omega|_{Z_i},\phi_i,V_i)$ on each piece $Z_i$
\end{enumerate}

Again, given a trisection of $(X,\omega)$ we can choose {\it some} Weinstein structures on the sectors, although not necessarily related to the global symplectic form $\omega$.  

Auroux has shown that every closed, symplectic 4--manifold is a branched cover over $\CP^2$ \cite{Auroux}.  Given the extra flexibility to put the branch locus in bridge position, we conjecture that every closed symplectic four-manifold admits such a trisection.

\begin{conjecture}
	Let $(X,\omega)$ be a symplectic four-manifold.  Then $(X,\omega)$ admits a Weinstein trisection.
\end{conjecture}

%%%%%%%%%%%%%%%%%%%%%%%%%%%%%%%%%%%%%%%%%%%%%%%%%%%%%%%%%
\subsection{Geography for $(g,0)$--trisections}
%%%%%%%%%%%%%%%%%%%%%%%%%%%%%%%%%%%%%%%%%%%%%%%%%%%%%%%%%

Trisections provide a restricted domain to explore the geography problem in the sense that one can ask which intersection forms are realized by smooth 4--manifolds that admit $(g,0)$--trisections.  (See~\cite{FKSZ} for a trisection-theoretic discussion of the intersection form.)

\begin{question}
 	If $Q$ is the intersection form of some smooth, oriented 4--manifold $X$, is there a 4--manifold $X'$ that admits a $(g,0)$--trisection and has intersection form $Q$?
\end{question}

The classification of symmetric bilinear forms implies that if $Q$ is unimodular, even and indefinite, then it is isomorphic to $n E_8 \oplus m H$ for some integers $n,m$ (where negative values are interpreted as summing with the opposite orientation).  If $Q$ is the intersection form of a smooth, closed, oriented and simply-connected 4--manifold $X$, then Rokhlin's Theorem implies that $n$ is even.  The well-known 11/8--Conjecture asserts that $b_2(X) \geq \frac{11}{8} \sigma(X)$ for such a four-manifold, or equivalently that  $3n \leq 2m$. 

\begin{conjecture}[Trisected 11/8--Conjecture]
	Suppose that $X$ admits a $(g,0)$--trisection and the intersection form of $X$ is even and indefinite.  Then
\[g \geq \frac{11}{8} \sigma(X).\]
\end{conjecture}

All intersection forms permitted by the conjecture can be realized by connected sums of $S^2 \times S^2$ and $K3$.  Thus, we obtain the following theorem as a corollary of the construction of a $(22,0)$--trisection of $K3$.  The following theorem has been obtained by Spreer and Tillmann, as well~\cite{Spreer-Tillmann}.

\begin{theorem}
\label{thmx:realize}
	Every even and indefinite intersection form consistent with the 11/8--Conjecture is realized by a smooth four-manifold with a $(g,0)$--trisection.
\end{theorem}

%%%%%%%%%%%%%%%%%%%%%%%%%%%%%%%%%%%%%%%%%%%%%%%%%%%%%%%
\subsection*{Organization}
%%%%%%%%%%%%%%%%%%%%%%%%%%%%%%%%%%%%%%%%%%%%%%%%%%%%%%%

In Section~\ref{sec:trisections}, we give a detailed overview of the central objects from the theories of trisections and bridge trisections, introduce the notion of efficiency, and establish some orientation and positivity conventions.  In Section~\ref{sec:branch}, we introduce our main tools regarding branched covers of bridge trisected surfaces.  We also discuss a generalization of the notion of a bridge trisection to the setting of singular surface-links and describe how to bridge trisect resolutions of such surface-links. It is here that parts (1) and (2) of Theorem~\ref{thmx:eff_branch} are proved. Sections~\ref{sec:trisections} and~\ref{sec:branch} are independent of the complex topology that features prominently in later sections.

In Section~\ref{sec:CP2}, we carefully construct the standard trisection and a Stein trisection on $\CP^2$, and give efficient bridge trisections for complex curves in $\CP^2$, proving Theorem~\ref{thmx:eff_bridge}(1). In Section~\ref{sec:cp2_branch}, we use these bridge trisections to give efficient trisections for complex hypersurfaces in $\CP^3$, proving Theorem~\ref{thmx:eff_tri}(1) and~(2), and establishing Corollary~\ref{corox:exotic}.  In Section~\ref{sec:S2xS2}, we give a careful analysis of the genus two trisection of $S^2\times S^2$, and show that complex curves therein admit efficient bridge trisections, proving Theorem~\ref{thmx:eff_bridge}(2).  In Section~\ref{sec:s2xs2_branch}, we use these bridge trisections to produce efficient trisections of branch covers over $S^2\times S^2$, thus proving Theorem~\ref{thmx:eff_tri}(3) and~(4), as well as Theorem~\ref{thmx:eff_bridge}(4).

In Section~\ref{sec:complete}, we discuss complete intersections, proving Theorem~\ref{thmx:eff_tri}(5) and Theorem~\ref{thmx:eff_bridge}(3).  In Section~\ref{sec:K3}, we discuss the proper transform and describe (even more) constructions of $K3$ as a branched cover, proving Theorem~\ref{thmx:eff_branch} and establishing the remaining parts of Theorem~\ref{thmx:K3}.

%%%%%%%%%%%%%%%%%%%%%%%%%%%%%%%%%%%%%%%%%%%%%%%%%%%%%%%
\subsection*{Acknowledgements}
%%%%%%%%%%%%%%%%%%%%%%%%%%%%%%%%%%%%%%%%%%%%%%%%%%%%%%%

We are grateful to John Baldwin, David Gay, Bob Gompf, Paul Melvin, Juanita Pinz\'on-Caicedo, Jonathan Spreer, Andras Stipsicz, Stephan Tillmann, and Alex Zupan for important observation and helpful conversations.
We would like to thank AIM for sponsoring the workshop ``Trisections and low-dimensional topology" in March 2017, which represents the origin of this paper.  In addition, the first author would like to thank Laura Starkston for suggesting he attend this workshop, as the second author did not invite him.

%%%%%%%%%%%%%%%%%%%%%%%%%%%%%%%%%%%%%%%%%%%%%%%%%%%%%%%
%%%%%%%%%%%%%%%%%%%%%%%%%%%%%%%%%%%%%%%%%%%%%%%%%%%%%%%
\section{Trisections and bridge trisections}
\label{sec:trisections}
%%%%%%%%%%%%%%%%%%%%%%%%%%%%%%%%%%%%%%%%%%%%%%%%%%%%%%%
%%%%%%%%%%%%%%%%%%%%%%%%%%%%%%%%%%%%%%%%%%%%%%%%%%%%%%%

In this section, we recall the basic objects central to the theories of trisections and bridge trisections.  We also introduce the notion of efficiency for these objects.  There is no reliance on complex topology in this section.

Given an oriented manifold $M$, we let $\overline M$ denote $M$, equipped with the opposite orientation. Given a submanifold $N\subset M$, we let $\nu(N)$ denote an open regular neighborhood of $N$ in $M$.  Throughout, we consider indices cyclically; i.e., $s_{n+1} = s_1$ in $S=\{s_1,s_2,\ldots, s_n\}$.

Let $X$ be a smooth, closed, oriented 4--manifold.  We use the term {\it 1--handlebody} of genus $g$ to denote the compact 4--manifold $\natural^g (S^1 \times D^3)$ and the term {\it handlebody} of genus $g$ to denote the compact three-manifold $\natural^g (S^1 \times D^2)$.
\begin{definition}
A $(g;\bold k)$--{\it trisection} $\cT$ of $X$ is a decomposition $X = Z_1 \cup Z_2 \cup Z_3$ such that
\begin{enumerate}
\item Each $Z_i$ is a 4--dimensional 1--handlebody of genus $k_i$, where $\bold k = (k_1,k_2,k_3)$;
\item Each $H_i = Z_i \cap Z_{i-1}$ is a three-dimensional 1--handlebody of genus $g$; and
\item $\Sigma = Z_1 \cap Z_2 \cap Z_3$ is a closed, oriented surface of genus $g$.
\end{enumerate}
The union $H_1 \cup H_2 \cup H_3$ is called the \emph{spine} of $\cT$.  We let $Y_1 = \partial Z_1$. As oriented submanifolds of $X$, we have $Y_i = H_i\cup_\Sigma\overline H_{i+1}$ and $\Sigma = \partial H_i$ inside $Y_i$. If $k_1 = k_2 = k_3 = k$, we call $\cT$ \emph{balanced} and a \emph{$(g,k)$--trisection} of $X$. 
\end{definition}

Note that $Y_i = \partial Z_i$ is diffeomorphic to $\#^{k_i}(S^1\times S^2)$, and $\Sigma$ is a genus $g$ Heegaard surface for $Y_i$.  Trisections were introduced by Gay and Kirby in 2012; in particular, every smooth, oriented, connected, closed 4--manifold admits a trisection~\cite{GK}. Importantly, the spine uniquely determines the trisection, by Laudenbach-Poenaru~\cite{LP}.  We refer the reader to\cite{GK} and~\cite{MSZ} for complete details regarding trisections.

\begin{proposition}\label{prop:trisections_handles}
	If $X$ admits a $(g;\bold k)$--trisection, then $X$ admits a handle-decomposition with a single 0--handle, a single 4--handle, $k_i$ 1--handles, $g-k_{i+1}$ 2--handles, and $k_{i+2}$ 3--handles.  In particular,
	\begin{enumerate}
		\item $\chi(X) = 2 + g - k_1 - k_2 - k_3$;
		\item $\text{rank}(\pi_1(X))\leq\min\{k_i\}$; and
		\item $\text{rank}(H_2(X))\leq \min\{g-k_i\}$.
	\end{enumerate} 
\end{proposition} 

\begin{proof}
	For the first claim, see Proposition~20 of~\cite{MZ-GBT} or~\cite{MSZ} for a more detailed account. For the rest, see Proposition~4.1 of~\cite{M_spun}.
\end{proof}

In this paper, we will give many examples of trisections of simply-connected four-manifolds that admit $(g,0)$--trisections.  In this vein, we make the following definition.

\begin{definition}
	A $(g,\bold k)$--trisection of a simply-connected four-manifold is \emph{efficient} if it is balanced and $k=0$.
\end{definition}

Note that we might define a more general notion of efficiency in which we simply ask that the algebraic constraints on $g$ and $k$ given by Proposition~\ref{prop:trisections_handles} be sharp.  In the sequel, however,  we mostly restrict our attention to simply-connected four-manifolds.

A \emph{cut system} for a surface $\Sigma$ is a collection of $g$ disjoint, simple closed curves whose complement in $\Sigma$ is a connected, planar surface.  Cut systems for $\Sigma$ modulo handleslides correspond to handlebodies $H$ with $\partial H = \Sigma$~\cite{Joh}.  A \emph{trisection diagram} for $\cT$ is a triple $(\alphas,\betas,\gammas)$ of cut systems for $\Sigma$ such that the union $H_\alpha\cup H_\beta\cup H_\gamma$ of handlebodies determined by the cut systems is the spine of a trisection.

A core feature of the theory of trisections is that every trisection (hence, every 4--manifold) can be described by a trisection diagram~\cite{GK}.

We will use $H_1$, $H_2$, and $H_3$ interchangeably with $H_\alpha$, $H_\beta$, and $H_\gamma$, respectively.  In diagrams, we will always use red, blue, and green for the $\alpha$--curves, $\beta$--curves, and $\gamma$--curves, respectively.  Furthermore, we will sometimes refer to these curve as \emph{red curves}, \emph{blue curves}, and \emph{green curves}, respectively.  Similarly, for expositional efficiency, we will sometimes write $\alphas_1$, $\alphas_2$, and $\alphas_3$ for $\alphas$, $\betas$, and $\gammas$, respectively.

A collection $\taus$ of $b$ properly embedded arcs in a handlebody $H$ is called a \emph{$b$--tangle}.  Such a collection is {\it trivial} if it can be isotoped rel-$\del$ to lie in $\del H$.  If $\taus=\{\tau^j\}$ is trivial, then there exist a collection of disjoint disks $\Deltas = \{\Delta^j\}$, embedded in $H$, such that $\del \Delta^j = \tau^j \cup \cA^j$ where $\cA^j$ is an arc in $\del \Sigma$.  We call each $\Delta^j$ a {\it bridge disk} and the arc $\cA^j$ the {\it shadow} of $\tau^j$.  A {\it bridge splitting} of a link $L$ is the three-manifold $Y$ is a decomposition $(Y,L) = (H_1, \tau_1) \cup_{\Sigma} \overline{(H_2,\tau_2)}$ where the $H_i$ are handlebodies and the $\taus_i$ are trivial tangles in the $H_i$.  Finally, a collection $\cD = \{\cD^j\}$ of $c$ properly embedded disks in a 1--handlebody $X$ is called a \emph{$c$--disk-tangle}.  A disk-tangle $\cD$ in $X$ is {\it trivial} if it can be isotoped rel-$\del$ to lie in $\del X$.

A \emph{knotted surface} is a pair $(X,\cK)$, where $X$ is a smooth, oriented, closed, connected 4--manifold and $\cK$ is a smoothly embedded, closed surface in $X$.

\begin{definition}
A $(g,\bold k;b,\bold c)$--{\it bridge trisection} of a knotted surface $(X,\cK)$ is a decomposition $(X,\cK) = (Z_1,\cD_1) \cup (Z_2,\cD_2) \cup (Z_3,\cD_3)$ such that
\begin{enumerate}
\item $X = Z_1 \cup Z_2 \cup Z_3$ is a $(g,\bold k)$--trisection of $X$,
\item each $\cD_i$ is a trivial $c_i$--disk-tangle in $Z_i$, where $\bold c = (c_1,c_2,c_3)$; and
\item each $\taus_i = \cD_i \cap \cD_{i-1}$ is a trivial $b$--tangle.
\end{enumerate}
We say that $\cK$ is in {\it bridge trisected position} with respect to a trisection $\cT$ of $X$ if the decomposition $\cT_\cK$ of $(X,\cK)$ induced by $\cT$ is a bridge trisection.  The union $(H_1,\taus_1)\cup(H_2,\taus_2)\cup(H_3,\taus_3)$ is called the \emph{spine} of $\cT_\cK$.  We let $\bold x = \Sigma\cap\cK$, so $(\Sigma,\bold x) = \partial(H_i,\taus_i)$; we call $\bold x$ the \emph{bridge points}.  We let $\cL_i = \partial\cD_i$. If $\cT$ is balanced and $c_1 = c_2 = c_3 = c$, we call $\cT_\cK$ \emph{balanced} and a \emph{$(g,k;b,c)$--bridge trisection} of $(X,\cK)$. 
\end{definition}

Note that $\cL_i$ is a $c_i$--component unlink in $Y_i$, and is in $b$--bridge position with respect to the splitting $Y_i = H_i\cup_\Sigma \overline H_{i+1}$. As with trisections, the spine of a bridge trisection determines $(X,\cK)$ (Corollary~9 of~\cite{MZ-GBT}). The main result of \cite{MZ-GBT} is that, for every knotted surface $(X,\cK)$, $\cK$ can be put in bridge trisected position with respect to any trisection on $X$.

\begin{theorem}[\cite{MZ-GBT}]
Let $\cT$ be a trisection of a 4--manifold $X$.  Every smoothly embedded, closed surface $\cK$ in $X$ can be isotoped to lie in  bridge trisected position with respect to $\cT$.
\end{theorem}

In this paper, we will be particularly interested bridge trisections of minimal possible complexity.

\begin{definition}
	A bridge trisection $\cT$ is \emph{efficient} if it is a $(g,0;b,1)$--bridge trisection.  We call $\cT$ an \emph{efficient $(g,b)$--bridge trisection}.
\end{definition}

Again, we choose, for simplicity, to define efficiency in a more narrow way than is strictly necessary.  In general, we might ask that the bound on $\bold c$ coming from the algebraic topology of the exterior be sharp.  For example, we have the following analogue of Proposition~\ref{prop:trisections_handles}, which is a restatement of Proposition~20 of~\cite{MZ-GBT}.

\begin{proposition}\label{prop:bridge_handles}
	If $\cK\subset X$ is in $(b;\bold c)$--bridge trisected position with respect to a $(g;\bold k)$--trisection $\cT$ of $X$, then 
	\begin{enumerate}
		\item $\cK$ can be built with $c_i$ cups, $b-c_{i+1}$ bands, and $c_{i+2}$ caps inside $X$. In particular,
		$$\chi(\cK) = c_1 + c_2 + c_3 - b.$$
		\item If $k_i=0$, then $\pi_1(X\setminus(\cK)$ has rank at most $c_i$.  In particular, if $(X,\cK)$ admits an efficient trisection, then $X$ is simply-connected and $\pi_1(X\setminus\cK)$ is cyclic.
	\end{enumerate}
\end{proposition}

With this in mind, our notion of efficiency is only achievable for surface-links whose exterior has cyclic fundamental group.

A {\it curve-and-arc system} $(\alphas,\cA)$ for a trivial tangle $(H,\taus)$ consists of a cut system $\alphas$ for $H$ and a collection $\cA$ of shadows for the strands of $\taus$.  A {\it shadow diagram} is a triple $\frak D = ((\alphas,\cA),(\betas,\cB),(\gammas,\cC))$ of cut-and-arc systems such that the union $(H_1,\taus_1)\cup(H_2,\taus_2),(H_3,\taus_3)$ is the spine of a bridge trisection. Again, for expositional efficiency, we will often write $\cA_1$, $\cA_2$, and $\cA_3$ for $\cA$, $\cB$, and $\cC$, respectively.  In keeping with standard color conventions, we will refer to the shadows $\cA$, $\cB$, and $\cC$ as \emph{red arcs}, \emph{blue arcs}, and \emph{green arcs}, respectively. In diagrams, we will also use these color conventions.  To clarify our diagrams, we will often use lighter shades for the arcs than for the curves.  For certain trisection diagrams that are closely associated to shadow diagrams (e.g., via a branched covering), we will sometimes use lighter shades for the curves in the trisection diagram corresponding to arcs in the shadow diagram.

As with trisections, every bridge trisection (hence, every knotted surface) can be encoded by a shadow diagram.

\begin{remark}\label{rmk:last-arc}
	In a shadow diagram for a trivial $b$--tangle, it is only necessary to record $b-1$ of the shadows, because the placement of the last arc is determined by the others.  At times, for simplicity, we will not draw this last shadow (henceforth, a \emph{lost shadow}).  Other times, for completeness or symmetry, we will draw all arcs.
% **not quite!!**** It is a standard convention to only draw $(n-1)$ pairs of arcs in an $n$-bridge presentation of a link, as the last pair of arcs is determined up to isotopy. This convention was followed in \cite{MZ-Bridge} and abandoned in \cite{MZ-GBT}.  We will avoid this convention and draw all shadows in every picture.
\end{remark}

\begin{remark}\label{rmk:oriented_bt}
	If $\cK$ is oriented, we adopt the convention that $\cL_i = \partial\cD_i = \taus_i\cup_\bold x\overline\taus_{i+1}$.  Thus, the orientation on $\cK$ restricts to an orientation on $\cD_i = Z_1 \cap \cK$, which induces an orientation on the unlink $\cL_i$.  We assign orientations to the bridge points so that each strand of $\taus_i$ points from a positive point to a negative point.  Note that each collection of shadow arcs is oriented from positive points to negative points.  For example, the boundary of the bridge disks $\Delta_1$ for $\taus_1$ is $\partial\Delta_1 = \taus_1\cup_\bold x\overline{\cA}$.	
	
	For a bridge point $x\in\bold x$, let $\sigma(x) \in \{\pm 1\}$ denote its sign. In addition, we use the shadow diagram to define a second sign $\epsilon(x)$ on each bridge point $x$.  This determines a partition $\bold x = \bold{x}_+ \cup \bold{x}_-$ into positive and negative bridge points.  Recall that the orientation on $X$ induces an orientation on the central surface $\Sigma$.  If the three incoming arcs of $\cA$, $\cB$, and $\cC$ at $x$ are positively cyclically ordered, we set $\epsilon(x) = 1$; otherwise we set $\epsilon(x) = -1$.  See Figure \ref{fig:signs}(Left).

%%%%%%%%%%%%%%%%%%%%%%%%%%%%%%%%%%%%%%%%%%%%%%%
\begin{figure}[h]
\centering
\includegraphics[width=.65\textwidth]{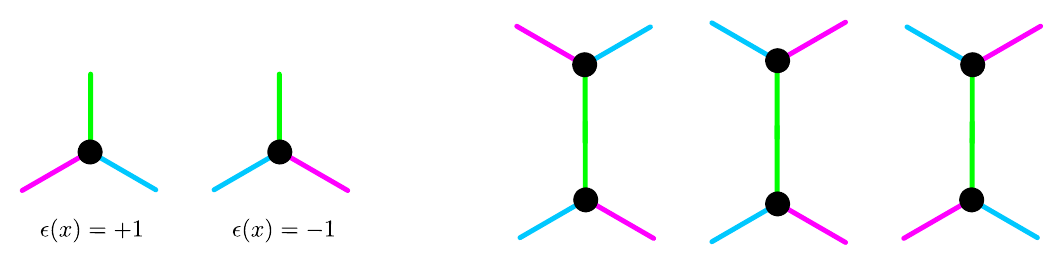}
\caption{{\it (Left)} A positive and negative bridge point in the shadow diagram.  {\it (Right)} A positively twisted, untwisted, and negatively twisted band in the shadow diagram.}
\label{fig:signs}
\end{figure}
%%%%%%%%%%%%%%%%%%%%%%%%%%%%%%%%%%%%%%%%%%%%%%%%%\end{remark}

\end{remark}

%%%%%%%%%%%%%%%%%%%%%%%%%%%%%%%%%%%%%%%%%%%%%%%%%%%%%%%%%%%%%%
%%%%%%%%%%%%%%%%%%%%%%%%%%%%%%%%%%%%%%%%%%%%%%%%%%%%%%%%%%%%%%
\section{Bridge trisecting branched coverings and resolutions of push-offs}\label{sec:branch}
%%%%%%%%%%%%%%%%%%%%%%%%%%%%%%%%%%%%%%%%%%%%%%%%%%%%%%%%%%%%%%
%%%%%%%%%%%%%%%%%%%%%%%%%%%%%%%%%%%%%%%%%%%%%%%%%%%%%%%%%%%%%%

In this section, we discuss how bridge trisections transform under the taking of branched covers, the taking of multiple push-offs, and the resolution of singularities.  For the latter two transformation, we discuss a notion of singular bridge trisections.  We point out that a notion of singular bridge trisections was introduced by Cahn and Kjuchukova for studying branched covers of singular surface-knots in $S^4$~\cite{Cahn-Kj}.  This section is independent of the complex topology that features prominently in later sections.  In this section, we prove parts (1) and (2) of Theorem~\ref{thmx:eff_branch} from the introduction.

%%%%%%%%%%%%%%%%%%%%%%%%%%%%%%%%%%%%%%%%%%%%%%%%%%%%%%%%%%%%%%%%%%%%
\subsection{Diagrams for cyclic branched covers}\label{subsec:diags_branch}
%%%%%%%%%%%%%%%%%%%%%%%%%%%%%%%%%%%%%%%%%%%%%%%%%%%%%%%%%%%%%%%%%%%

Let $(X,\cK)$ be an oriented\footnote{The restriction to oriented knotted surfaces is not strictly necessary; see Remarks~\ref{rmk:2-fold} and~\ref{rmk:non-orientable}.} knotted surface, and let $\widetilde X$ denote the $n$--fold cover of $X$, branched along $\cK$.  Let $\widetilde\cK$ denote the lift of $\cK$ under this covering.  Let $\rho\colon\pi_1(X\setminus\cK)\twoheadrightarrow\Z_n$ denote the surjection from $\pi_1(X\setminus\cK)$ corresponding to this cover, which factors though a surjection $H_1(X\setminus\cK)\twoheadrightarrow\Z_n$ that we also denote by $\rho$. The following proposition illustrates how nicely behaved trisection structures are with respect to cyclic branched covers.

\begin{proposition}[Proposition~13 of~\cite{MZ-GBT}]
\label{prop:eff-branch}
	Suppose that $(X,\cK)$ admits a $(g, \bold k;b, \bold c)$--bridge trisection $\cT$.  Then, $(\widetilde X,\widetilde\cK)$ admits a $(g', \bold{k'};b, \bold c)$--bridge trisection $\widetilde\cT$, where
	\[g' = ng + (n-1)(b-1) \qquad\text{and}\qquad \bold{k'} = n\bold k + (n-1)(\bold c - \bold 1).\]
	Moreover, the pieces of $\widetilde\cT$ are given as the branched covers of the pieces of $\cT$.
\end{proposition}

In particular, efficiency is preserved under cyclic branched coverings.  This gives Theorem~\ref{thmx:eff_branch}(1).

\begin{corollary}\label{coro:eff_branch}
	If $(X,\cK)$ admits an efficient $(g;b)$--bridge trisection, then $(\widetilde X,\widetilde\cK)$ admits an efficient $(g';b)$--trisection, with $g'=ng+(n-1)(b-1)$.
\end{corollary}

The proof of Proposition~\ref{prop:eff-branch} is straight-forward, and given in~\cite{MZ-GBT}.  Our present task is to understand how a shadow diagram $\frak D$ corresponding to $\cT$ gives rise to a shadow diagram $\widetilde{\frak D}$ corresponding to $\widetilde\cT$.  Let $\frak D$ be a shadow diagram corresponding to $\cT$ that satisfies the following specifications.
\begin{enumerate}
	\item For each $i$, the cut system $\alphas_i$ is disjoint from the shadow arcs $\cA_i$.
	\item For each $i$, a shadow arc has been discarded from $\cA_i$; cf. Remark~\ref{rmk:last-arc}.
	\item Each bridge point $x\in\bold x$ has been assigned a sign $\sigma(x)$ as per Remark~\ref{rmk:oriented_bt}.
\end{enumerate}
Note that none of these specifications provides a meaningful restriction on $\frak D$; for example, the curves of $\alphas_i$ can always be isotoped off of the arcs of $\cA_i$.

A \emph{pairing} of $\frak D$ is a collection $\omega$ of $b$ oriented arcs in $\Sigma$ with oriented boundary $\partial\omega = \bold x$ (i.e., each arc of $\omega$ connects points in $\bold x$ with opposite sign).  A pairing is \emph{admissible} if the count of algebraic intersections of each curve of $\alphas\cup\betas\cup\gammas$ with $\omega$ is divisible by $n$.  We will now describe how, given an admissible pairing $\omega$ for $\frak D$, we can construct a shadow diagram $\widetilde{\frak D}$ corresponding to the trisection $\widetilde\cT$.

Let $\Sigma^\omega$ denote the compact surface with $b$ boundary components obtained by slicing $\Sigma$ open along~$\omega$.  The boundary of $\Sigma^\omega$ decomposes as the union $\omega^+\cup_\bold x\overline\omega^-$, where each of $\omega^\pm$ is an oriented copy of~$\omega$.  (We think of $\omega^-$ as being the left side of $\omega$, while $\omega^+$ is thought of as the right side.)  Define
$$\widetilde\Sigma =\bigcup_{j=1}^n\Sigma^\omega_j,$$
where each $\Sigma^\omega_j$ is a copy of $\Sigma^\omega$, with boundary denoted by $\omega^+_j\cup_{\bold x_j}\overline\omega^-_j$, and where the union is taken by identifying $\omega^+_j$ and $\omega^-_{j+1}$.
We think of $\widetilde\Sigma$ as a \emph{parking garage}; someone driving around on the $j^{\text{th}}$ deck ($\Sigma^\omega_j$) transfers to the $(j+1)^{\text{st}}$ deck ($\Sigma^\omega_{j+1}$) by ``entering'' an arc of $\omega^+_j$, the ``exit'' of which is the corresponding arc of $\omega^-_{j+1}$.

The cut system $\alphas_i$ on $\Sigma$ induces a collection of arcs and curves $\alphas_i^\omega$ on $\Sigma^\omega$.  Let
$$\widetilde\alphas_i' = \bigcup_{j=1}^n(\alphas_i^\omega)_j$$
denote the corresponding collection of curves on the surface $\widetilde\Sigma$.  A collection of curves on a surface is called a \emph{sub-cut system} if the result of surgering the surface along the curves is connected.

\begin{lemma}\label{lem:sub-cut}
	For each $i$, $\widetilde\alphas_i'$ is a sub-cut system of $ng$ curves for $\widetilde\Sigma$.  Moreover, $\widetilde\alphas_i'$ bounds a collection $D'_i$ of disjoint disks in $\widetilde H_i$.
\end{lemma}

\begin{proof}
	Each curve of $\alphas_i$ intersects the arcs of admissible pairing algebraically a multiple of $n$ times, so it lifts to $n$ curves on $\widetilde\Sigma$.  Since each such curve bounds a disk in $H_i$ that is disjoint from the branch locus, the disk lifts to $n$ disjoint disks extending the lifts of the curve.  The result of surgering $\widetilde\Sigma$ along the curves of $\widetilde\alphas_i'$ is the $n$--fold cover of the result of surgering $\Sigma$ along the curves of $\alphas_i$, branched along the bridge points $\bold x$.  It follows that the former surgery space is connected, since the latter is.
\end{proof}

Next, let the $b-1$ shadows of $\cA_i$ induce a collection of arcs $\cA_i^\omega$ on $\Sigma^\omega$.  Let $(\cA_i^\omega)^\pm$ denote a pair of push-offs of $\cA_i^\omega$, so each of $(\cA_i^\omega)^\pm$ is properly embedded in $\Sigma^\omega$ and co-bounds with $\cA_i^\omega$ and $\partial\Sigma^\omega$ a collection of rectangles that comprise the trace of the push-off to that side.  Let
$$\widetilde{\cA_i}' = \bigcup_{j=1}^n(\cA_i^\omega)_j^\pm$$
denote the corresponding collection of curves on $\widetilde\Sigma$.

\begin{lemma}\label{lem:arc_lifts}
	For each $i$, $\widetilde{\cA_i}'$ is a collection of $n(b-1)$ curves on $\widetilde\Sigma$ that separate $\widetilde\Sigma$ into $b-1$ $n$--punctured spheres and a connected component.	Moreover each $\widetilde{\cA_i}'$ bounds a collection $D_{\Delta_i}$ of $n(b-1)$ disjoint disks in $\widetilde H_i$.
\end{lemma}

\begin{proof}
	Consider the collection $\Deltas_i$ of bridge disks for $\taus_i$.  Let $\delta\subset\Deltas_i$ be a bridge disk for a strand $t\subset\taus_i$, and let $t_*\subset\Sigma$ denote the corresponding shadow of $t$.  The lift $\widetilde\delta$ is an arrangement of $n$ semi-disks $\{\widetilde\delta^l\}_{l=1}^n$ which meet along the strand $\widetilde t\subset\widetilde\taus_i$ so that $\partial(\widetilde\delta^l) = \widetilde t\cup\widetilde t_*^l$, where $\{\widetilde t_*^l\}$ are the $n$ lifts of the shadow $t_*$.  In other words, the union $\widetilde\delta\cup_{l=1}^n\widetilde\delta^l$ is a bouquet of $n$ disks whose common (and pairwise) intersection is $\widetilde t$, an arc in the boundary of each.  This bouquet is properly embedded in $\widetilde H_i$.

	Let $N_\delta$ denote a regular neighborhood of $\widetilde\delta$ in $\widetilde H_i$, and let $D_\delta = \overline{\partial N\setminus\widetilde\Sigma}$ --- i.e., the portion of $\partial N$ that consists of $n$ properly embedded disks in $\widetilde H_i$.  Since each disk of $D_\delta$ is parallel into the bouquet $\widetilde\delta$, we can assume that $\partial D_\delta$ is contained in push-offs of $\partial\widetilde\delta$ -- i.e., contained in $\widetilde{\cA_i}'$.  Let $D_{\Delta_i}$ denote the union over $\delta\subset\Deltas_i$ of the $D_\delta$.
	
	Thus, $D_{\Delta_i}$ is a collection of $n(b-1)$ disjointly embedded disks in $\widetilde H_i$ with $\partial D_{\Delta_i}\subset\widetilde{\cA_i}'$.  (Note that now, we have discarded the redundant shadow arc.)  Moreover, each of the $b-1$ copies of $D_\delta$ separates off a ball chose boundary intersects $\widetilde\Sigma$ in an $n$--punctured sphere.  The fact that $\partial D_{\Delta_i} = \widetilde{\cA_i}$ follows from the fact that each $N_\delta$ contains the lifts of the bridge disk $\delta$.
\end{proof}

Let $D''_i$ denote a sub-collection of $(n-1)(b-1)$ disks of $D_{\delta_i}$ so that $\widetilde\alphas_i''=\partial D''_i$ is a sub-cut system for $\widetilde H_i$.  (Equivalently, throw away one of the $n$ disks coming from each of the $(b-1)$ neighborhoods $N_\delta$.)  Let $\widetilde\alphas_i = \widetilde\alphas_i'\cup\widetilde\alphas_i''$.  Let $\widetilde{\cA_i}$ denote a single lift of $\cA_i$ to $\widetilde\Sigma$ -- i.e. , $\widetilde{\cA_i}$ consists of $b-1$ arcs in $\widetilde\Sigma$ which connect the lift $\widetilde{\bold x}$ of the bridge points $\bold x$.

Let
$$\widetilde\Deltas_i = \bigsqcup_{\delta\in\Deltas_i}\widetilde\delta^1,$$
and let $\widetilde t_i = \widetilde\Deltas_i\cap\widetilde\Sigma$.  In other words, $\widetilde\Deltas_i$ consists of one petal from each of the bouquets that cover the bridge disks of $\Deltas_i$, and $\widetilde \cA_i$ is the corresponding shadows.

\begin{proposition}\label{prop:diag_branch}
	The tuple $\widetilde{\frak D} = (\widetilde\Sigma;\widetilde\alphas_1,\widetilde\alphas_2,\widetilde\alphas_3;\widetilde{\cA_1},\widetilde{\cA_2},\widetilde{\cA_3})$ is a shadow diagram for the bridge trisection~$\widetilde\cT$ of $(\widetilde X,\widetilde\cK)$.
\end{proposition}

\begin{proof}
	The pieces of the trisection $\widetilde \cT$ are precisely the lifts of the pieces of $\cT$ under the $n$--fold branched covering. We have already seen that $\widetilde\Sigma$ is the cover of $\Sigma$.  It remains to see that the $\widetilde\alphas_i$ are cut systems for the $\widetilde H_i$ and that the $\widetilde{\cA_i}$ are shadows for the $\widetilde \taus_i$.
	
	By Lemmas~\ref{lem:sub-cut} and~\ref{lem:arc_lifts}, each $\alphas_i$ bounds a collection of disks in $\widetilde H_i$, which can be assumed to be disjoint by a standard inner-most disk and outer-most arc argument, since the curves of $\widetilde\alphas_i$ are disjoint.  Moreover, the result of surgering $\widetilde H_i$ along $\alphas_i$ is connected.  Since $\widetilde\alphas_i$ consists of $g' = ng + (n-1)(b-1)$ curves, it is a cut system for $\widetilde H_i$.
	
	By Lemma~\ref{lem:arc_lifts}, the lift of the bridge disks $\Deltas_i$ is a bouquet of semi-disks, the petals of which are bridge disks for the lift $\widetilde\taus_i$ of the bridges $\taus_i$.  Choosing $\widetilde \Deltas_i$ to consists of one petal in each bouquet gives a collection of bridge disks for $\widetilde\taus_i$ with corresponding shadows $\widetilde t_i$.  Note that different choice of petals are related by slides of bridge disks over the compression disks of $D_{\Delta_i}$.
\end{proof}

To illustrate the methods of this subsection, consider the shadow diagram given as the left-most frame of Figure~\ref{fig:branch_diags}. (We will see later that this shadow diagram corresponds to a bridge trisection of the cubic curve $(\CP^2,\cC_3)$.)  The second frame of this figure shows the same shadow diagram without a redundant arc of each color and with an admissible pairing $\omega$ for the bridge points.  Note first that, in this case, we have chosen the original blue arcs to give the pairing; this will not always be the case in this paper, but such a choice will always give a (possibly non-admissible) pairing. Note second that this pairing is admissible because the green curve hits the pairing (with coherent orientation) a total of three times, while the red curve and blue curve miss it entirely. Note lastly that if the red curve, say, were isotoped to pass the center of the square vertically, it would intersect the pairing exactly once; thus, the pairing would cease to be admissible.

\begin{figure}[h!]
\centering
\includegraphics[width=.9\textwidth]{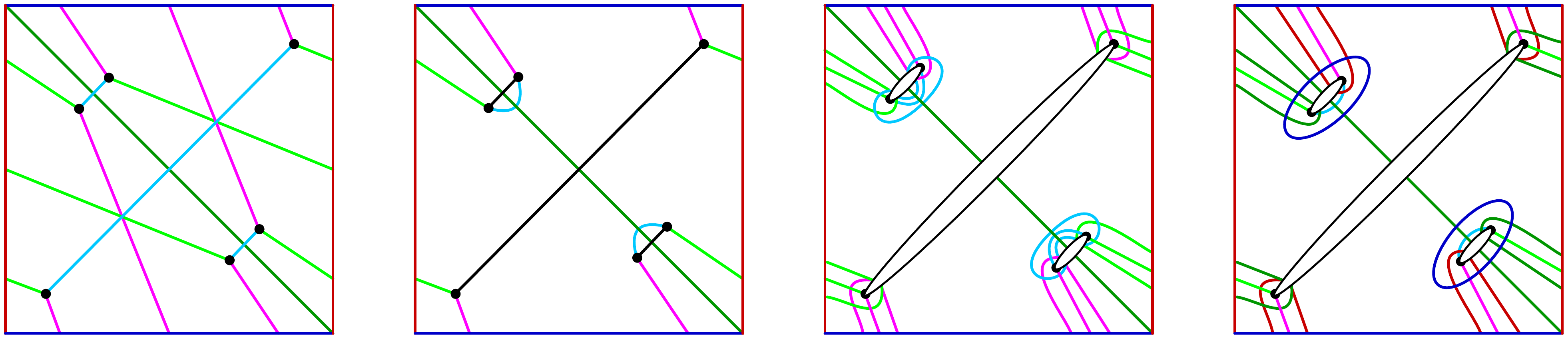}
\caption{(Left to Right) A shadow diagram for a $(1,0;3,1)$--bridge trisection; the same diagram, less an arc of each color, together with an admissible pairing (shown in black and oriented upward); the third frame shows the result of slicing the diagram open along the pairing; the fourth frame is equivalent to the third, when considered inside the union of three copies, as in Figure~\ref{fig:cubic_branch_square}.}
\label{fig:branch_diags}
\end{figure}

The third frame of Figure~\ref{fig:branch_diags} shows a copy of $\Sigma^\omega$, together with the induced properly embedded arcs $\alphas_i^\omega$, $(\cA^\omega_i)^\pm$, as well as a copy of the original shadows.  The fourth frame is identical, except that the colors of both types of induced arcs have been darkened, foreshadowing their role as curves on $\widetilde\Sigma$, and the dark-blue arcs have been isotoped to lie entirely on $\Sigma^\omega$.  (This last isotopy is not strictly allowed on $\Sigma^\omega$, but will be once three copies of $\Sigma^\omega$ are glued together to give $\widetilde\Sigma$.)

Figure~\ref{fig:cubic_branch_square} shows the shadow diagram corresponding to the 3--fold branched cover for this example, which is built with three copies of $\Sigma^\omega$ and the induced curves and arc thereon.  Note that some arcs and curves have been discarded, since including all of them  would be redundant.  Similarly, only one lift of each shadow has been included.  Note that because we started with a shadow diagram of type $(1,0;3,1)$, the corresponding shadow diagram for the 3--fold branched cover has type $(7,0;3,1)$. 

\begin{figure}[h!]
\centering
\includegraphics[width=.9\textwidth]{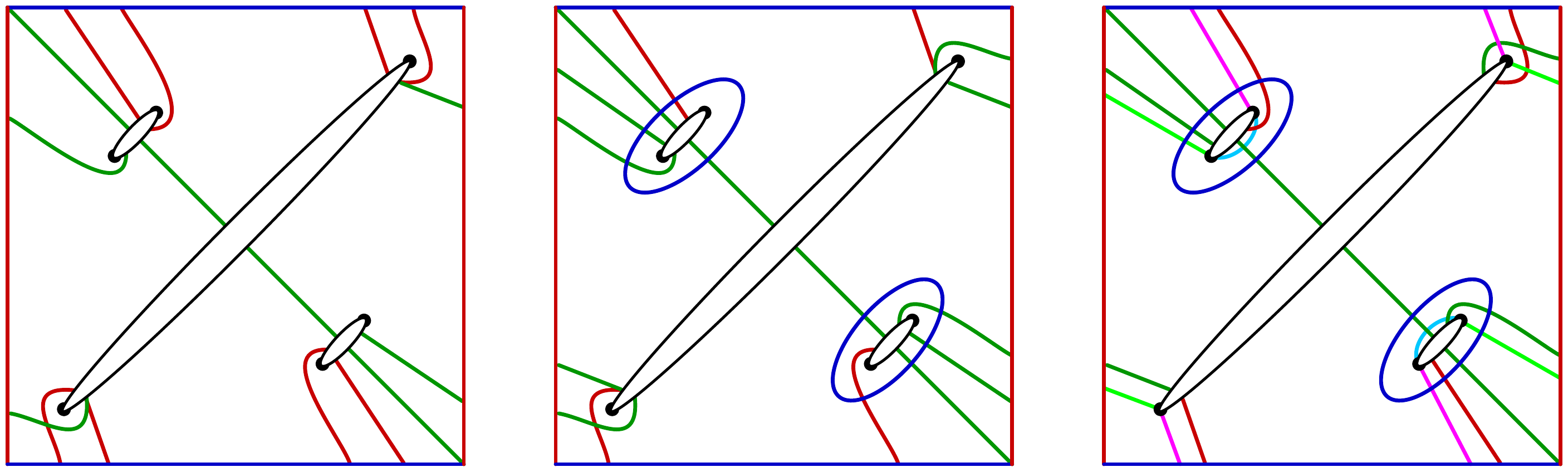}
\caption{A shadow diagram for a $(7,0;3,1)$--bridge trisection arising from the shadow diagram and associated objects of Figure~\ref{fig:branch_diags}. To interpret the underlying surface, note that each square becomes a torus once opposite edges are identified and that the southeastern edge of each black oval in each square is identified with the northwestern edge of the corresponding oval in the square directly to the right (considered cyclically).}
\label{fig:cubic_branch_square}
\end{figure}

\begin{remark}\label{rmk:2-fold}
	The construction outlined above is made significantly easier in the case where $n=2$ -- i.e., when considering a double-branched cover.  The reason this case is simpler is that it is not necessary to take the push-offs $(\cA^\omega)^\pm$ of the induced arcs $\cA_i^\omega$ in $\Sigma^\omega$.  This is because the lift of the collection of bridge disks $\Deltas_i$ is a bouquet with two petals, which is simply a disk.  So, for each shadow $t_*$, the union $\widetilde t_*$ of the two lifts under the double-branched covering is already a compression disk for $\widetilde H_i$.  Additionally, since the parking garage only has two levels, it is not necessary to remember whether one is going up or down when crossing the lifts of the cut arcs, only that one is changing level.
\end{remark}

\begin{remark}\label{rmk:non-orientable}
	In the construction outlined above, we made the simplifying assumption that the knotted surface was oriented. If $\cK$ is non-orientable, then each meridian of $\cK$ has order two in $H_1(\partial(\nu(\cK)))$.  It follows that each meridian of $\cK$ has order two in $H_1(X\setminus\cK)$, and the only cyclic branched covering that $(X,\cK)$ admits is a 2--fold one.  This observation, together with Remark~\ref{rmk:2-fold}, allows us to conclude that the construction described above holds in the case that $\cK$ is non-orientable, provided $n=2$. 
\end{remark}

%%%%%%%%%%%%%%%%%%%%%%%%%%%%%%%%%%%%%%%%%%%%%%%%%%%%%%%%%%%%%%%%%%%%
\subsection{Diagrams for resolutions of push-offs}\label{subsec:push-off_toruss}
%%%%%%%%%%%%%%%%%%%%%%%%%%%%%%%%%%%%%%%%%%%%%%%%%%%%%%%%%%%%%%%%%%%

We will now introduce a generalization of the notion of a bridge trisection to the setting of singular surfaces.  A knotted surface $(X,\cK)$ is \emph{singular} if it smoothly embedded away from finitely many points, near each of which the pair $(X,\cK)$ looks like the cone on a link in $S^3$.  In the present paper, all singularities encountered will be transverse intersections -- i.e., cones on Hopf links.  Singular bridge trisections were first studied by Cahn and Kjuchikova~\cite{Cahn-Kj}.

Let $L$ be a $c$--component link in $\#^k(S^1\times S^2)$ that is contained in a three-ball. A collection $\cD$ of disks in $\natural^k(S^1\times B^3)$ with $\partial\cD=L$ is a \emph{$c$--cone-tangle} for $L$ if each component of $\cD$ is an embedded cone on its boundary.  (If the boundary of a component of $\cD$ is unknotted, then we can consider that component as smoothly embedded by smoothing the cone point.)  Note that two types of singularities can arise in a cone-tangle.  If a component of $L$ is knotted, then the disk bounding that component will be a singular at the cone point, with the link of the singularity given by the boundary knot.  If two components of $L$ are unknotted, but linked, then the disk components they bound will be individually smoothly embedded, but they will intersect transversely.  These singularities look locally like the cone on a Hopf link.  Only this latter type of singularity will be considered in this paper; the components of the boundary of a cone-tangle will always be unknotted, but may be linked.

\begin{definition}
A $(b;\bold c)$--{\it singular bridge trisection} of a singular knotted surface $(X,\cK)$ is a decomposition $(X,\cK) = (X_1,\cD_1) \cup (X_2,\cD_2) \cup (X_3,\cD_3)$ such that
\begin{enumerate}
\item $X = X_1 \cup X_2 \cup X_3$ is a trisection of $X$,
\item each $\cD_i$ is a trivial $c_i$--cone-tangle in $X_i$, where $\bold c = (c_1, c_2, c_3)$ and
\item each $\taus_i = \cD_i \cap \cD_{i+1}$ is a trivial $b$--tangle.
\end{enumerate}
We say that $\cK$ is in {\it singular bridge trisected position} with respect to a trisection $\cT$ of $X$ if the decomposition $\cT_\cK$ of $(X,\cK)$ induced by $\cT$ is a singular bridge trisection.  The union $(H_1,\taus_1)\cup(H_2,\taus_2)\cup(H_3,\taus_3)$ is called the \emph{spine} of $\cT_\cK$.  We let $\bold x = \Sigma\cap\cK$, so $(\Sigma,\bold x) = \partial(H_i,\taus_i)$; we call $\bold x$ the \emph{bridge points}.
\end{definition}

Let $(X,\cK)$ be a knotted surface, and let $N_\cK = \overline{\nu(\cK)}$ be a closed, regular neighborhood of $\cK$ in $X$.  We say that $\cK$ has \emph{self-intersection $e$} if the normal disk-bundle $D^2\hookrightarrow N_\cK\twoheadrightarrow \cK$ has Euler number $e$; the quantity $e$ is also called the \emph{normal Euler number} of $\cK$ and denoted $e(\cK)$. Denote by $m\cK$ the union of $m$ copies of $\cK$ inside $N_\cK$, considered as a (possibly) singular knotted surface in $X$.
%Let $\cT$ be a $(g,\bold k;b,\bold c)$--bridge trisection diagram for $(X,\cK)$.

\begin{proposition}\label{prop:singular}
	Suppose that a knotted surface $\cK$ with self-intersection $e$ can be put in $(b,\bold c)$--bridge trisected position with respect to a trisection $\cT$ of $X$.  Let $m\geq 1$.  Then, the singular knotted surface $m\cK$ can be put in $(mb,m\bold c)$--singular bridge trisected position with respect to $\cT$ such that $L_2$ and $L_3$ are unlinks and $L_1$ is the split union of a $m(c_1-1)$--component unlink and a $(m,me)$--torus link.  In particular, all three cone-tangles are smooth. 
	
	Given a shadow diagram for $\cT_\cK$ that has been locally trivialized so that a component of $\partial \cD_1$ appears as in the first frame of Figure~\ref{fig:push-off_torus}, a shadow diagram for $\cT_{m\cK}$ is obtained by modifying this local picture as indicated by last frame of Figure~\ref{fig:push-off_torus}, and taking parallel push-offs elsewhere.

\end{proposition}

\begin{proof}
	Consider a disk $D$ of the trivial $c$--disk-tangle $\cD_1 = \cK\cap X_1$ in the bridge trisection $\cT_\cK$.  Without loss of generality, we can assume that any non-triviality of the bundle $N_\cK$ is localized near this disk.  In other words, the $e(m-1)$ intersection points of $e\cK$ occur in a neighborhood $N_D$ of $D$.
	
	Let $\cD_i'= e\cK\cap X_i$. It follows that $\cD_2$ and $\cD_3$ are both trivial disk-tangles (of cardinality $mc_2$ and $mc_3$, respectively) and that the disk components of $\cD_1$ other than $D$ together give rise to a trivial $m(c_1-1)$--disk-tangle, while $D$ gives rise to $m$ disks intersecting pairwise in $e$ points.  It follows that the boundary links $L_2$ and $L_3$ are both unlinks, while $L_1$ is the disjoint union of a $m(c_1-1)$--component unlink with a $(m,me)$--torus link $J$.  This last fact follows, as this torus link is the boundary of the union $m$ sections of the disk bundle $D^2\hookrightarrow N_D\twoheadrightarrow D^2$ of Euler number $e$.
	
	Lastly, because the links $L_i$ were in $b$--bridge position inside the $Y_i = \partial X_i$, we have that the $L_i' = \partial \cD_i'$ are in $mb$--bridge position.  Collections of bridge disks for the latter tangles are given by $m$ push-offs of collections of bridge disks for the former tangles.
\end{proof}

\begin{remark}
If the normal Euler number of $\cK$ is 0, then we can assume the singular knotted surface guaranteed by Proposition \ref{prop:singular} is embedded as well.
\end{remark}

The local transformation from the disk $D$ to $m$ intersecting disks bounded by the torus link is shown in Figure~\ref{fig:push-off_torus}.  In this figure, the bridge index of $\partial D$ is $b_1=2$, while $e=2$ and $m=3$. If $\partial D$ had bridge number $b_1$ in $\cT_\cK$, then $J$ is in $mb_1$--bridge position inside $Y_1 = H_1\cup_\Sigma H_2$.

\begin{figure}[h!]
\centering
\includegraphics[width=.9\textwidth]{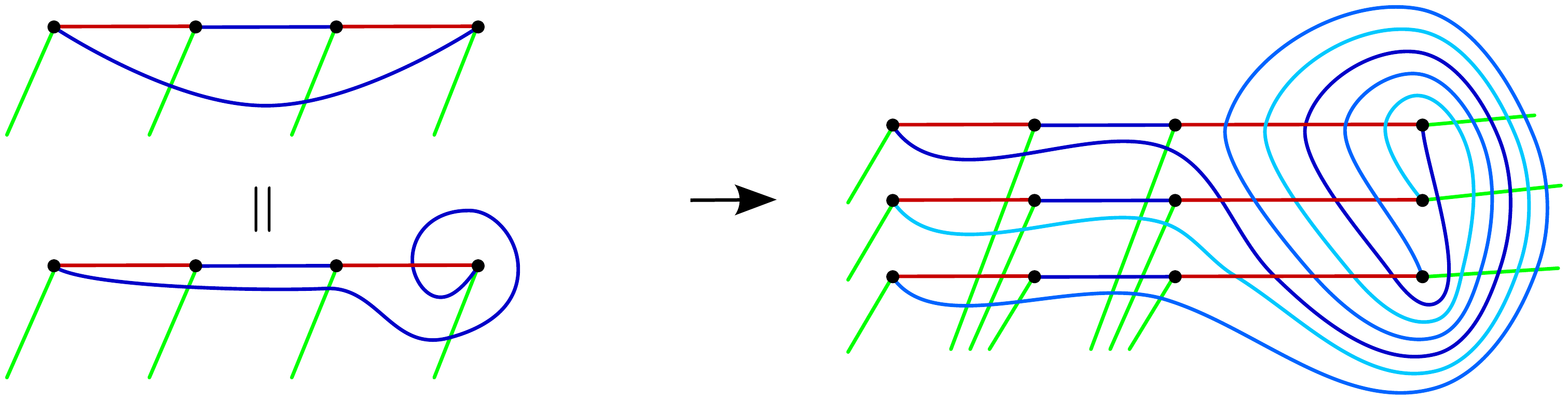}
\caption{(Left) A localized view of shadow diagram for the knotted surface $(X,\cK)$ with $e(\cK)=e$ that has been locally standardized so that a disk $D\subset\cD_1$ is isotopic into $\Sigma$, with $\partial D$ in $b_1$--bridge position. (Right) The singular shadow diagram for the singular union $(X,m\cK)$ of push-offs of $\cK$, displaying the local picture around the $(m,me)$--torus link $J$. Here, $b_1=2$, $m=3$, and $e=2$.}
\label{fig:push-off_torus}
\end{figure}

Note that in bridge trisection $\cT_{m\cK}$ for $(X,m\cK)$, the $mc_1$--cone-tangle $\cD_1'$ contains exactly $em(m-1)/2$ intersection points among $m$ of its disks and that these intersection points will be coherently oriented, according to the sign of $e$, which we assume to be positive for simplicity.

We resolve the singular points of $\cD_1'$ by removing a neighborhood of each and replacing it with an annulus, which can be thought of as the Seifert surface for the Hopf link arising as the link of the singularity.  Let $\cJ$ denote the knotted surface obtained by resolving all of the singularities of $e\cK$.  Note that the effect of this is to change the union of disks containing the singularities in $\cD_1'$ to a smooth surface $F$ bounded by the $(m,me)$--torus link $J$ in $X_1$.  See Figure~\ref{fig:torus_link_bands}. Note that $F$ has Euler characteristic $me-m^2e+m$, which is maximal, given this particular torus link.  Note also that, with respect to the standard radial Morse function on $X_1$, $F$ is ribbon, with $me$ minima and $m(me-1)$ saddles.

\begin{figure}[h!]
\centering
\includegraphics[width=.4\textwidth]{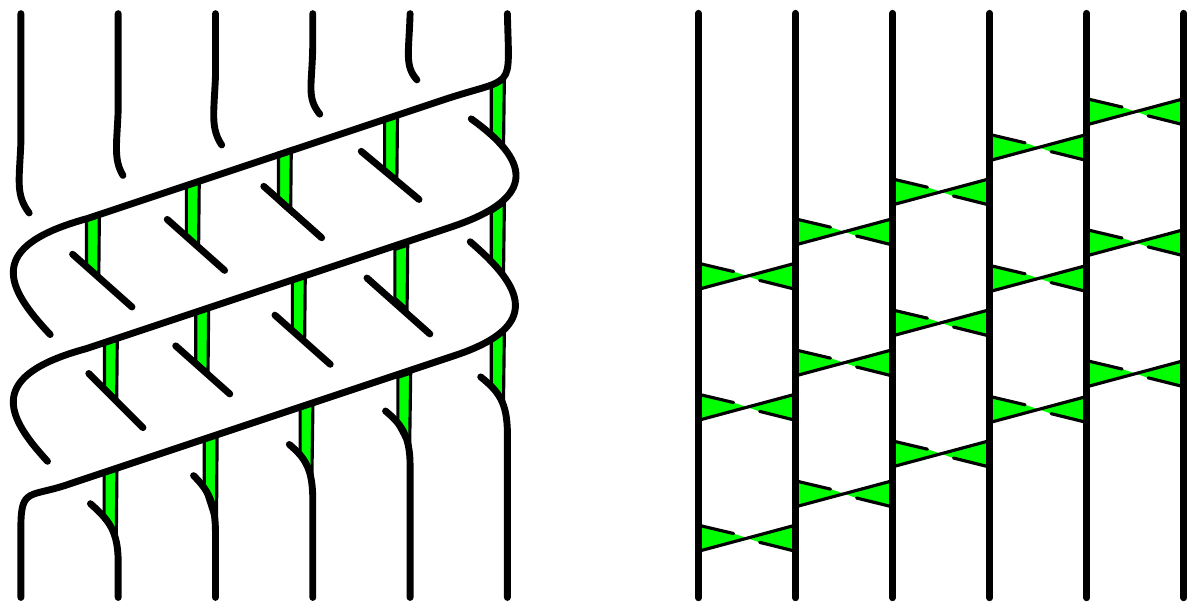}
\caption{(Left) The minimal genus surface $F$ for the torus link $J = T(m,me)$, viewed as a ribbon surface for $J$; (Right) The same surface viewed as the result of attaching $m(me-1)$ bands to the $me$--component unlink.  Here, $m=3$ and $e=2$.}
\label{fig:torus_link_bands}
\end{figure}

\begin{proposition}\label{prop:bridge_push-off}
	Suppose that a knotted surface $\cK$ with self-intersection $e \neq 0$ can be put in $(b,\bold c)$--bridge trisected position with respect to a trisection $\cT$ of $X$.  Let $m\geq 1$.  Then, the smooth resolution $(X,\cJ)$ of the singular knotted surface $m\cK$ can be put in $(b',m\bold c)$--bridge trisected position with respect to $\cT$, where $b' = mb+m(me-1)$.
	
	Given a shadow diagram for $\cT_\cK$ that has been locally trivialized so that a component of $\partial \cD_1$ appears as in the first frame of Figure~\ref{fig:push-off_torus}, a shadow diagram for $\cT_\cJ$ is obtained by modifying this local picture as indicated by first frame of Figure~\ref{fig:push-off_destab}, and taking parallel push-offs elsewhere.
\end{proposition}

\begin{proof}
	The smooth resolution $\cJ$ is almost in bridge trisected position with respect to $\cT$, except that the intersection $F \subset \cJ\cap X_1$ is a connected ribbon surface, not a trivial disk-tangle.  We will isotope the trisection $\cT$ so that $\cJ$ is in bridge trisected position.
	
	First, we isotope $\cT$ so that the bands $\upsilon$ of $F$ lie in $H_2$.  Note that $(Y_1,L_1) = (H_1,\taus_1)\cup(H_2,\taus_2)$ is currently a $m(c_1-1)$--component unlink split union a $(m,me)$--torus link $J$.  The former link will be irrelevant for our consideration, so we focus only on $J$, to which all the bands of $F$ are attached once they lie in $H_2$.  Suppose the torus link is in $m$--bridge position.  A local picture of this set-up is shown in the right frame of Figure~\ref{fig:push-off_torus}.
	
	We can perturb $\cT$ to increase the complexity of the bridge splitting of the torus link from $mb_1$ to $mb_1+m(me-1)$.  Doing so, we can arrange that the bands $\upsilon$ of $F$ are level in $\Sigma$ and dualized by bridge disks for $\tau_2\subset H_2$.  See the left frame of Figure~\ref{fig:push-off_bands}, which shows $J$, unperturbed, but with the bands $\upsilon$.

\begin{figure}[h!]
\centering
\includegraphics[width=.9\textwidth]{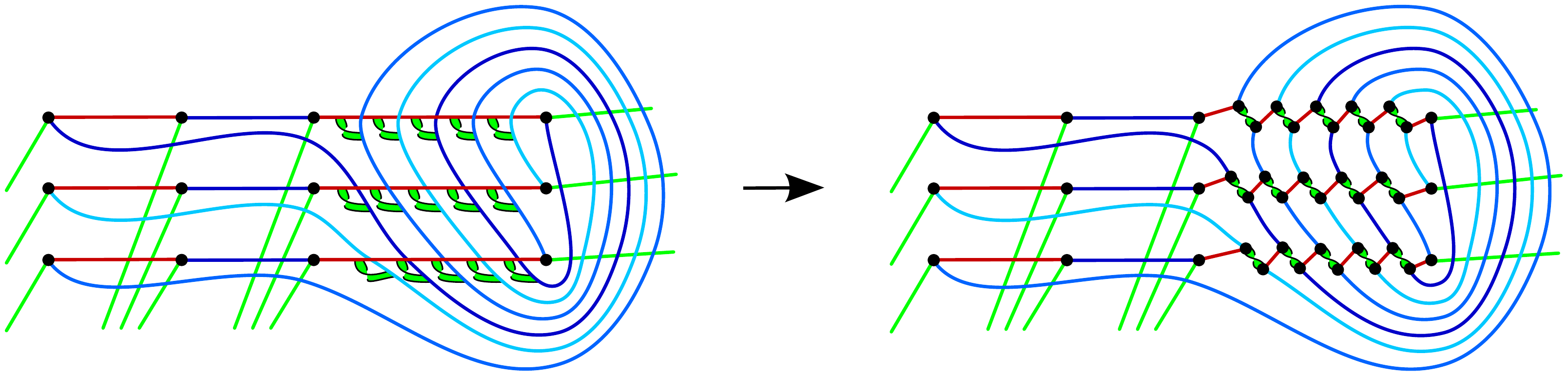}
\caption{(Left) A shadow diagram for the $(m,me)$--torus link $J$, together with the $me(m-1)$ bands $\upsilon$ of its ribbon spanning surface $F$, which have been pushed to lie in $Y_1$. (Right)  The dual picture: A shadow diagram for the unlink $U$ that bounds the minima of $F$, together with the dual bands $\upsilon^*$, giving a local banded bridge splitting.  Here, $m=3$ and $e=2$.}
\label{fig:push-off_bands}
\end{figure}
	
	Before perturbing the bridge splitting of $L_1 = \partial\cD_1$, we have that $\cD_2\subset X_2$ is a trivial $mc_2$--disk-tangle.  The disks of $\cD_2$ can be thought of as the trace of the bridges $\tau_2$ as some collection of bands for $\cJ$ are resolved.  (Here, we refer to bands coming from a handle decomposition of $\cJ\setminus F$.)  See Lemma~3.1 of~\cite{MZ-Bridge} for details.  After perturbing the bridge splitting of $L_1$ to increase the bridge index of the torus link $J$, we have created new bridge arcs in $\tau_2$, each of which contributes a new trivial disk to $\cD_2$, which now must be a trivial $(mc_2+m(me-1))$--disk-tangle.
	
	However, because we have leveled and dualized the bands of $F$ inside $H_2$, we can include them into the band system for $\cJ\setminus F$.  Each band of $F$ is dualized by one of the new arcs of $\tau_2$, so the new trivial disk that was just described as being created in $\cD_2$ will no longer exist.
	
	This process can be thought of as pushing the saddles of $F$ from $X_1$ to $X_2$, through $H_2$. The bands $\upsilon$ are thought of as being attached to $J$, and their resolution gives the unlink $U$ bounding the the $m$ minima of $F$.  We switch perspectives and consider the dual bands $\upsilon^*$, which we think of as being attached to the unlink bounding the minima of $F$.  If we push the saddles of $F$ up through $H_2$, into $X_2$, as described above, the link $L_1$ changes precisely by the resolution of $\upsilon$.  Now we see a unlink of $mc_1-m+me$ components as $L_1$, and $\cD_1$ is a trivial $(mc_1-m+me)$--disk-tangle.  Since the bands $\upsilon$ were dualized by bridge disks in $H_2$, the same is true for $\upsilon*$.  It follows from all that has been said that $\cD_2$ is still a trivial $mc_2$--disk-tangle.  This dual picture is shown in the right frame of Figure~\ref{fig:push-off_bands}.  The transition in this figure from the the left frame to the right frame encapsulate the process of perturbing the bridge splitting of $J$, and switching our perspective from the torus link with bands $(J,\upsilon)$ to the bridge banded link diagram $(U,\upsilon^*)$.
	
	Having locally arranged that $(U,\upsilon^*)$ is a bridge banded link diagram compatible with the original bridge trisection, we can interpret the bands of $\upsilon^*$ as new green arc shadows of the trivial tangle $\tau_3$.  The final result is shown in the first frame of Figure~\ref{fig:push-off_destab}.
\end{proof}

\begin{figure}[h!]
\centering
\includegraphics[width=.9\textwidth]{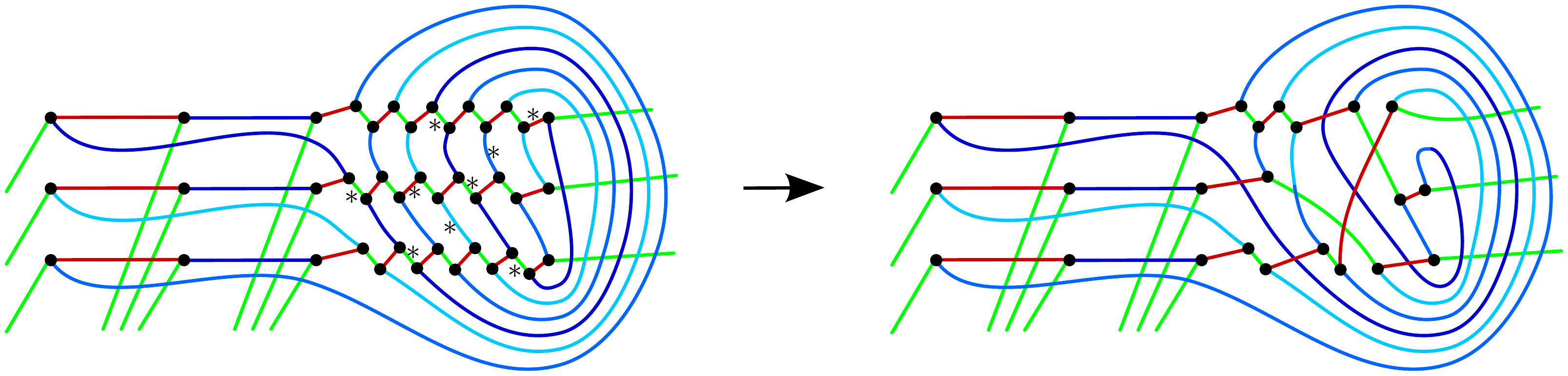}
\caption{The process of destabilizing the bridge trisection of the $(X,\cJ)$.  (Left) The shadow diagram resulting from incorporating the bands of the torus link $J$ into the trisection, as in Proposition~\ref{prop:bridge_push-off} and Figure~\ref{fig:push-off_bands}.  Asterisks mark bridges that can be destabilized to produce the shadow diagram in the right frame.}
\label{fig:push-off_destab}
\end{figure}

The next proposition shows that the increase in the parameters of the bridge trisection resulting from the taking of copies and the ensuing resolution can be undone via deperturbations.

\begin{proposition}\label{prop:push-off_efficient}	
	Suppose that a knotted surface $\cK$ with self-intersection $e \neq 0$ can be put in $(b,\bold c)$--bridge trisected position with respect to a trisection $\cT$ of $X$.  Let $m\geq 1$.  Then, the smooth resolution $(X,\cJ)$ of the singular knotted surface $m\cK$ can be put in $(b'',\bold c)$--bridge trisected position with respect to $\cT$, where $b'' = mb+me(m-1) - 3(m-1)$.
	
	Given a shadow diagram for $\cT_\cK$ that has been locally trivialized so that a component of $\partial \cD_1$ appears as in the first frame of Figure~\ref{fig:push-off_torus}, a shadow diagram for $\cT_\cJ$ is obtained by modifying this local picture as indicated by last frame of Figure~\ref{fig:push-off_destab}, and taking parallel push-offs elsewhere.	
\end{proposition}

\begin{proof}
	Given the bridge trisection $\cT_\cJ$ produced in Proposition~\ref{prop:bridge_push-off}, it suffices to destabilize the disk-tangle $\cD_i$ a total of $3(m-1)$ times.  To do this, one must find $m-1$ arcs of each color in the shadow diagram for $\cT_\cJ$ whose interior doesn't intersect the arcs of the other two colors and whose endpoints span distinct pairs of connected components in the other two colors.
	
	An example of this is shown in the left frame of Figure~\ref{fig:push-off_destab}.  The arcs marked with asterisks have the desired property.  It is easily seen, for example, that the two marked green arcs connect distinct pairs of red-blue curves.  By referring to the combinatorics of the first frame of Figure~\ref{fig:push-off_torus}, it can be verified that the red arcs and the blue arcs span distinct pairs of connected components in the other two colors, as well.
	
	The reader might be concerned that there are arcs or curves complicating this local picture that are not shown.  Any other arcs and curves that are red or blue have been assumed to fall outside the local picture.  Indeed, though, there may be green arcs and green curves crashing through this picture that are not shown.  However, these can be assumed to be supported away from the red/blue intersection of the right frame of Figure~\ref{fig:push-off_torus} and, hence, away from the region supporting the bands in Figure~\ref{fig:push-off_bands} and the destabilization arcs in Figure~\ref{fig:push-off_destab}.
	
	Although we have only provided a recipe for destabilizing in the example shown in which $b_1=2$, $m=3$, and $e=2$, it should be clear that the combinatorics of the torus links are simple enough that similar destabilizations could be performed for any $b_1$--bridge positioning of any $(m,me)$--torus link.  In general, the grid-like arrangement of hexagons that contains he asterisks in the left frame of Figure~\ref{fig:push-off_destab} will be $me-2$ hexagons wide and $m-1$ hexagons tall.  In order to destabilize, one need only find a collection of disjoint arcs (disjoint even at their end points) of the following sort:
	\begin{enumerate}
		\item there are $m-1$ blue arcs, one in each row of hexagons;
		\item there are $me-1$ green arcs, one in each column of hexagons (including the columns to the left and right of the grid); and
		\item there are $m-1$ red arcs, which lie in distinct northwest-southeast-running diagonals when these diagonals are considered modulo $m$.
	\end{enumerate}
	These conditions ensure that each arc connects a pair of distinct components of the disk-tangle corresponding to the arcs of the other two colors and that no chain of arcs of one color connects the same two components as any other chain of arcs in that color.
	For example, in the grid in the left frame of Figure~\ref{fig:push-off_destab}, the red arcs have been chose to lie in diagonals three and seven, if we treat the bottom left hexagon as lying in the second diagonal.  The effect of this is that this pair of red arcs spans the three shades of blue arcs; each shade of blue corresponds to a distinct component of the blue/green unlink $L_2$.
	
	Since the bridge number before destabilizing was $b' = mb+m(me-1)$ and we destabilized $2m+me-3$ times, the final bridge number is $b'' = mb+m^2e-m-2m-me+3 = mb+me(m-1) -3(m-1)$, as desired.
\end{proof}

Finally, we conclude that efficiency can be preserved under the process of taking smooth resolutions of parallel push-offs a given knotted surface.  This proves Theorem~\ref{thmx:eff_branch}(2).

\begin{theorem}
\label{thm:push-off_efficient}
	Suppose that a knotted surface $\cK$ with self-intersection $e \neq 0$ can be put in efficient $b$--bridge trisected position with respect to a trisection $\cT$ of $X$.  Let $m\geq 1$.  Then, the smooth resolution $(X,\cJ)$ of the singular knotted surface $m\cK$ can be put in efficient $b''$--bridge trisected position with respect to $\cT$, where $b'' = mb+me(m-1)-3(m-1)$.	
\end{theorem}

\begin{proof}
	This is restating of Proposition~\ref{prop:push-off_efficient} in the case of $c=1$.
\end{proof}

We will use this result in Section~\ref{sec:complete} to deduce that complex surfaces obtained as complete intersections of collections of hypersurfaces in $\CP^n$ admit efficient bridge trisections.

%%%%%%%%%%%%%%%%%%%%%%%%%%%%%%%%%%%%%%%%%%%%%%%%%%%%%%%%%%%%%%
%%%%%%%%%%%%%%%%%%%%%%%%%%%%%%%%%%%%%%%%%%%%%%%%%%%%%%%%%%%%%%
\section{Complex curves in $\CP^2$}\label{sec:CP2}
%%%%%%%%%%%%%%%%%%%%%%%%%%%%%%%%%%%%%%%%%%%%%%%%%%%%%%%%%%%%%%
%%%%%%%%%%%%%%%%%%%%%%%%%%%%%%%%%%%%%%%%%%%%%%%%%%%%%%%%%%%%%%

In this section, we discuss complex curves in $\CP^2$, showing that they admit complex bridge trisections and, when considered only up to smooth isotopy, efficient bridge trisections. We prove Theorem~\ref{thmx:eff_bridge}(1) and Proposition~\ref{propx:Stein_CP2} from the introduction.

\subsection{Toric model of $\CP^2$}

The toric geometry of $\CP^2$ yields a trisection $\cT$ as follows.  Define the moment map $\mu \colon \CP^2 \rightarrow \RR^2$ by the formula
\[\mu([z_1:z_2:z_3]) = \left( \frac{3 |z_1|}{|z_1| + |z_2| + |z_3|}, \frac{3 |z_2|}{|z_1| + |z_2| + |z_3|} \right)\]
The image of $\mu$ is the convex hull of the points $\{ (0,0),(3,0),(0,3) \}$.  The fiber of $\mu$ over an interior point is a torus (Lagrangian with respect to the Fubini-Study Kahler form); the fiber over an interior point of a face of the polytope is $S^1$; and the fiber over a vertex is a point.  The preimage of an entire face of the polytope is a complex line $$L_i = \{[z_1:z_2:z_3] : z_i = 0\}
$$ for some $i=1,2,3$.  

%%%%%%%%%%%%%%%%%%%%%%%%%%%%%%%%%%%%%%%%%%%%%%%
\begin{figure}
\centering
\includegraphics[width=.2\textwidth]{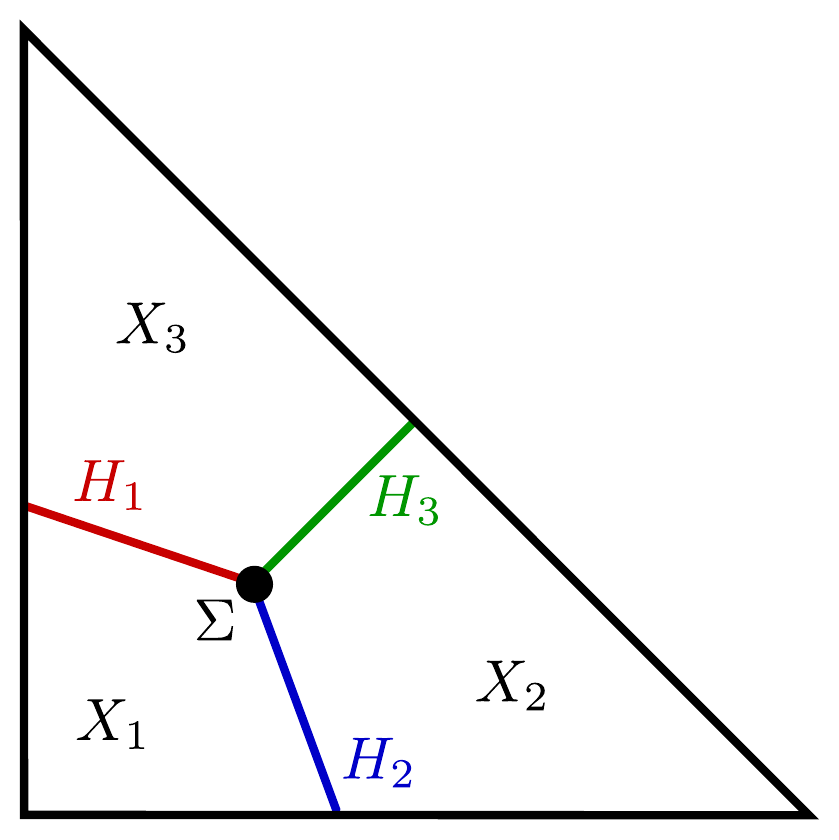}
\caption{The moment polytope of $\CP^2$, with the trisection decomposition described.}
\label{fig:toric}
\end{figure}
%%%%%%%%%%%%%%%%%%%%%%%%%%%%%%%%%%%%%%%%%%%%%%%%%

We can use the barycentric subdivision of the simplex $\mu(\CP^2)$ to construct a trisection $\cA$ of $\CP^2$.  The preimage of the barycenter $(1,1)$ is the torus 
$$ \Sigma = \{ [z_1:z_2:z_3] : |z_1| = |z_2| = |z_3| = 1\}$$
that is the core of $\cA$.  The barycentric subdivision consists of six triangles; grouping these in pairs determines three subsets of $\mu(\CP^2)$ whose preimages are the three pieces of the trisection decomposition.  Define subsets
\begin{align*}
X_i &= \left\{ [z_1:z_2:z_3] : |z_{i}|,|z_{i+1}| \leq |z_{i+2}| \right\} \\
H_i &= \left\{ [z_1:z_2:z_3] : |z_i| \leq |z_{i+1}| = |z_{i+2}|\right\}
\end{align*}
In the coordinate chart $\phi_1(x_1,x_2) = [x_1:x_2:1]$ on $\CP^2$, the handlebody $X_i$ is exactly the polydisk cut out by the inequalities $|x_1| \leq 1$ and $|x_2| \leq 1$.  Its boundary is clearly the union of two solid tori $H_1$ and $H_2$.

\begin{proposition}
\label{prop:tris_CP2}
	The decomposition $\CP^2 = X_1 \cup X_2 \cup X_3$ is a $(1,0)$--trisection.
\end{proposition}

We will refer to $\cT$ as the {\it standard} trisection of $\CP^2$. 
There is trisection diagram $(\alphas,\betas,\gammas)$ for $\cA$ where each cut system consists of a single simple closed curve.  The central surface $\Sigma = X_1 \cap X_2 \cap X_3$ is the torus $\{[e^{i \theta}: e^{i \psi}:1] : \theta,\psi \in [0,2\pi]\}$ and we can choose attaching circles
\[ \alpha = \{ [ e^{i \theta}:1:1]\} \qquad \beta = \{ [1: e^{i \psi}:1] \} \qquad \gamma = \{ [e^{-i \theta}: e^{-i \theta}:1]\}\]
Note that  $\{[\alpha],[\beta]\}$ is an oriented basis for $H_1(T^2)$ and $[\gamma ] = - [\alpha] - [\beta]$ in homology.  

Next, we will show that $\CP^2$ admits a Stein trisection.  Choose some $N \gg 0$ and for $i = 1,2,3$, define functions
\[f_i([z_1:z_2:z_3]) \coloneqq \left( \frac{|z_i|}{|z_{i-1}|} \right)^{2N} + \left( \frac{|z_{i+1}|}{|z_{i-1}|} \right)^{2N},\]
where the index $i$ is defined mod 3.  Choose some $\delta > 0$ and set
\[\widetilde{X}_i \coloneqq f^{-1}_i( (-\infty,2 + \delta)).\]

\begin{proposition}
\label{prop:Stein_CP2}
	The triple $\widetilde{X}_1,\widetilde{X}_2,\widetilde{X}_3$ is a $(1,0)$--Stein trisection of $\CP^2$.
\end{proposition}

\begin{proof}
In the affine chart $z_{i-1} = 1$, we have $f_i = |z_i|^{2N} + |z_{i+1}|^{2N}$ and its Levi form is positive definite along the boundary of $\widetilde{X}_i$.  Thus $\widetilde{X}_i$ is strictly Levi convex and therefore a Stein domain.  Moreover, it follows easily that each $X_i$ is a subset of $\widetilde{X}_i$, so that the union $\widetilde{X}_1 \cup \widetilde{X}_2 \cup \widetilde{X}_3$ is $\CP^2$.

Each function $f_i$ is constant along fibers of the moment map $\mu$, so that each $\widetilde{X}_i$ and each double and triple intersection is a union of fibers of $\mu$.  It is easy to see that $\widetilde{X}_i$ is an open neighborhood of $X_i$ and diffeomorphic to $B^4$; that $\widetilde{X}_i \cap \widetilde{X}_{i+1}$ is a neighborhood of $H_{i+1}$ and diffeomorphic to $S^1 \times B^3$; and that the triple intersection $\widetilde{X}_1 \cap \widetilde{X}_2 \cap \widetilde{X}_3$ is a neighborhood of the core and diffeomorphic to $T^2 \times D^2$.  Thus, we have obtained a Stein trisection of $\CP^2$.
\end{proof}

This proves Proposition~\ref{propx:Stein_CP2} from the introduction. Note that the pieces of the Stein trisection of $\CP^2$ are collar neighborhoods of the pieces of the standard trisection of $\CP^2$.

\subsection{Complex line}

In this subsection and the sequel, we will use $\cV_d$ to denote the algebraic variety obtained as the zero-set of the homogeneous polynomial $z_0^d + z_1^d + z_2^d$ in $\CP^2$, while using $(\CP^2,\cC_d)$ to denote the smooth isotopy class of $\cV_d$.

\begin{proposition}
\label{prop:line-bridge}
The projective line $\cV_1$ is in efficient bridge position.
\end{proposition}

\begin{proof}
In homogeneous coordinates, the intersections of $\cV_1$ with the handlebodies and central surface are
\begin{align*}
\cV_1 \cap H_1 = \tau_1 &= \left\{[-1 - e^{si}: e^{si}:1] : s \in \left[ \frac{2\pi}{3},\frac{4\pi}{3} \right] \right\} \\
\cV_1 \cap H_2 = \tau_2 &= \left\{[1:-1 - e^{si}: e^{si}] : s \in \left[ \frac{2\pi}{3},\frac{4\pi}{3} \right] \right\} \\
\cV_1 \cap H_3 = \tau_3 &= \left\{[e^{si}:1:-1 - e^{si}] : s \in \left[ \frac{2\pi}{3},\frac{4\pi}{3} \right] \right\} \\
\cV_1 \cap X_1 \cap X_2 \cap X_3 &= \left\{  \left[e^{\frac{2\pi i}{3}}: e^{-\frac{2 pi i}{3}}:1 \right], \left[e^{\frac{-2\pi i}{3}}: e^{\frac{2 pi i}{3}}:1 \right] \right\}
\end{align*}
In particular, $\cV_1$ intersects the central surface in two points and intersects each handlebody in a single arc.  Since the construction is triply-symmetric and $\cV_1$ is homeomorphic to $S^2$, we have that
\[ 2 =\chi(\cV_1) =  2 - 3 + \chi(\cD_1) + \chi(\cD_2) + \chi(\cD_3) = 3 \chi(\cD_1) - 1.\]
It follows that each $\cD_i$ is a topological disk.

To prove that $\cV_1$ is in bridge position, we must show that each disk $\cD_i$ and each arc $\taus_i$ is boundary parallel.  Consider the family of projective lines
\[\cV_{1,t} \coloneqq \{ z_0 + t z_1 + z_2 = 0\}\]
for $t \in [0,1]$.  For each $t < 1$, the above analysis continues to hold: $\cV_{1,t}$ intersects the central surface in two points, intersects each handlebody in a single arc, and intersects each 4-dimensional piece in a topological disk.  However, at $t = 1$, the line $\cV_{1,1}$ intersects the central surface along the simple closed curve $\alpha = \{[-1:z_1:1] : |z_1| = 1\}$, intersects the handlebody $H_1$ along the disk $\{[-1: z_1:1]: |z_1| \leq 1\}$, and is disjoint from the interiors of $X_1$ and $X_3$.  Thus, the family $\cD_{1,t} = \cV_{1,t} \cap X_1$ is an isotopy of the disk $\cD_1$ into the boundary of $X_1$ and the family $\tau_{2,t} = \cV_{1,t} \cap H_2$ is an isotopy of the tangle $\tau_2$ into the boundary of $H_2$.

By the 3-fold cyclic symmetry, this implies that each disk $\cD_i$ and each arc $\taus_i$ is boundary parallel.

\end{proof}

\begin{corollary}
	The pair $(\CP^2,\cC_1)$ admits an efficient $(1,1)$--bridge trisection, with shadow diagram as depicted in Figure~\ref{fig:line}.
\end{corollary}

\begin{proof}
Define $\phi(s,t) = t(1 + e^{is} + e^{-is})$ for $(s,t) \in \left[ \frac{2\pi}{3},\frac{4\pi}{3} \right] \times [0,1]$.  The family of arcs
\[\tau'_{1,t} = [-1 - e^{is} + \phi(s,t):e^{is}:1]\]
gives an isotopy of the arc $\tau_1$ to the segment $\tau'_{1,1} = [e^{-is}:e^{is}:1]$ on the central torus.  This is depicted in Figure \ref{fig:line} as the arc $\cA$.  Repeating this cyclically, we obtain arcs $\cB = \tau'_{2,1} = [1: e^{-is}:e^{is}] = [e^{-is}:e^{-2is}:1]$ and $\cC = \tau'_{3,1} = [e^{is}:1: e^{-is}] = [e^{2is}: e^{is}:1]$ on the central surface.
\end{proof}

%%%%%%%%%%%%%%%%%%%%%%%%%%%%%%%%%%%%%%%%%%%%%%%
\begin{figure}[h!]
\centering
\includegraphics[width=.2\textwidth]{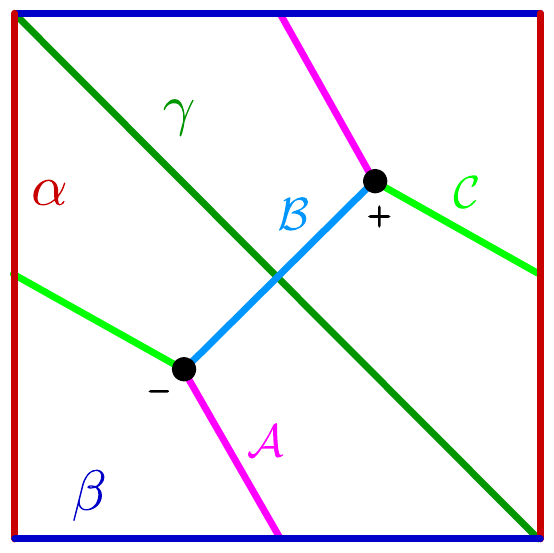}
\caption{The standard trisection diagram of $\CP^2$ together with a shadow diagram for $\cC_1$.}
\label{fig:line}
\end{figure}
%%%%%%%%%%%%%%%%%%%%%%%%%%%%%%%%%%%%%%%%%%%%%%%%%

\subsection{Complex curves}

We now describe bridge trisections of complex curves of arbitrary degree in $\CP^2$.  We first state and prove a well-known result, that the smooth isotopy class of a nonsingular curve in $\CP^2$ is determined by its degree.

\begin{proposition}
\label{prop:curve-isotopy}
Let $\cC,\cC'$ be nonsingular complex curves in $\CP^2$.  If $\cC$ and $\cC'$ are homologous, then they are ambient isotopic in $\CP^2$.  In particular, every nonsingular, degree $d$ complex curve is isotopic to $\cV_d$.
\end{proposition}

\begin{proof}
Since $\cC$ and $\cC'$ are linearly equivalent, we can find a family $\cC_{\lambda}$ of curves, parametrized by $\lambda \in \CP^1$, such that $\cC = \cC_0$ and $\cC' = \cC_{\infty}$.  This pencil of curves gives a `straight-line' isotopy between $\cC$ and $\cC'$.  The set of $\lambda \in \CP^1$ such that $\cC_{\lambda}$ is singular has complex codimension 1, hence real codimension 2, in $\CP^1$.  Consequently, we can choose a path $\lambda(t) \in \CP^1$ for $t \in [0,1]$ such that $\lambda(0) = 0$ and $\lambda(1) = \infty$ and $\cC_{\lambda(t)}$ is a nonsingular complex curve for all $t$.  This family of curves is an isotopy.
\end{proof}

\begin{figure}[h!]
\centering
\includegraphics[width=.5\textwidth]{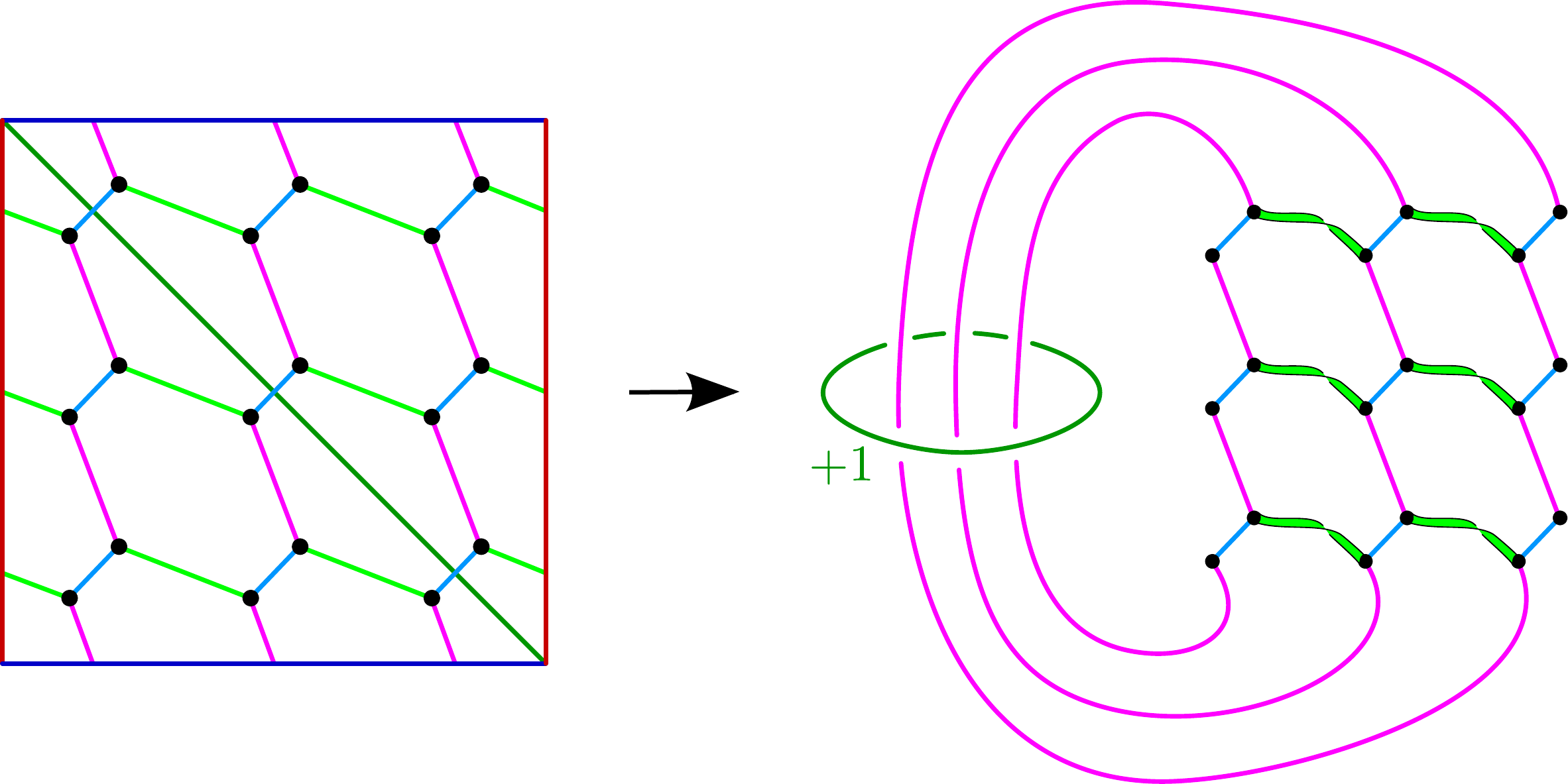}
\caption{An example of the banded link diagram corresponding to the $(1,0;d^2,d)$--bridge splitting of the degree $d$ complex curve in $\CP^2$.  Shown is the instance of $d=3$.}
\label{fig:Banded_Link_Cubic}
\end{figure}

To construct a bridge trisection of $\cC_d$, we can follow the method of Subsection \ref{subsec:push-off_toruss}, taking $d$ parallel copies of the bridge trisection for $\cC_1$ and then resolving crossings. This yields the shadow diagrams in Figure \ref{fig:Curves_Plane_Square}.  Presently, we give a different construction.

\begin{proposition}
	The pair $(\CP^2,\cC_d)$ admits a $(1,0;d^2,d)$--bridge trisection, a shadow diagram for which cuts $\Sigma$ into hexagons, as shown with examples in Figure~\ref{fig:Curves_Plane_Square}.
\end{proposition}

\begin{proof}
To describe complex curves in $\CP^2$, we mimic \cite[Example 6.2.7]{GS}, which is based on \cite{AK}.  To obtain the curve $\cC_d$, take a generic collection of $d$ complex lines in $\CP^2$.  These lines pairwise intersect in a single point.  Resolving the $\frac{1}{2}d(d-1)$ double points yields $\cC_d$.  A banded link presentation for the pair $(\CP^2,\cC_d)$ can be obtained as follows.  Consider the torus link $L = T(d,d)$, obtained as the closure of the $d$--strand braid $(\sigma_1 \cdots \sigma_{d-1})^{d}$ with braid axis $\gamma$. See Figure~\ref{fig:Banded_Link_Cubic} for the case of $d=3$. As the closure of a positive braid, $L$ has a canonical Seifert surface $F$ obtained by taking disks bounded by the $d$ Seifert circles and attaching 1--handles for each crossing coming from the braid word.  This yields a surface $F$ with $d$ boundary components, $\chi(F) = d - (d-1)^2 = 3d - d^2 - d$, and $g(F) = \frac{1}{2}(d-1)(d-2)$.  Attach a 2--handle to $B^4$ along the braid axis with framing $+1$ and cap off with a 4--handle to obtain $\CP^2$.  The link $L$ becomes the $d$--component unlink after the 2--handle attachment and each component can be capped off with a disk in the boundary of the four-dimensional 4--handle. (These disks can be thought of as cores of the four-dimensional 2--handle.)  Capping off $F$ thusly yields $\cC_d$ \cite{AK,GS}.  The right side of Figure~\ref{fig:Banded_Link_Cubic} shows a banded link diagram corresponding to this description of $\cC_d$.

	We will show that shadow diagrams of the form illustrated in Figure\ref{fig:Curves_Plane_Square} correspond to the handle decompositions for $\cC_d$ described above.

	The trisection of $\CP^2$ determines a handle decomposition as follows: the 4--ball $X_1$ is the 0--handle; the $\gamma$ curve is dual to the compressing disk bounded by $\beta$, therefore we push the $\gamma$ curve into the handlebody $H_\beta$ and attach a 2--handle along $\gamma$ with its surface framing, which is +1; finally, the 4--ball $X_3$ is the 4--handle \cite[Lemma 13]{GK}.

\begin{figure}[h!]
\centering
\includegraphics[width=.75\textwidth]{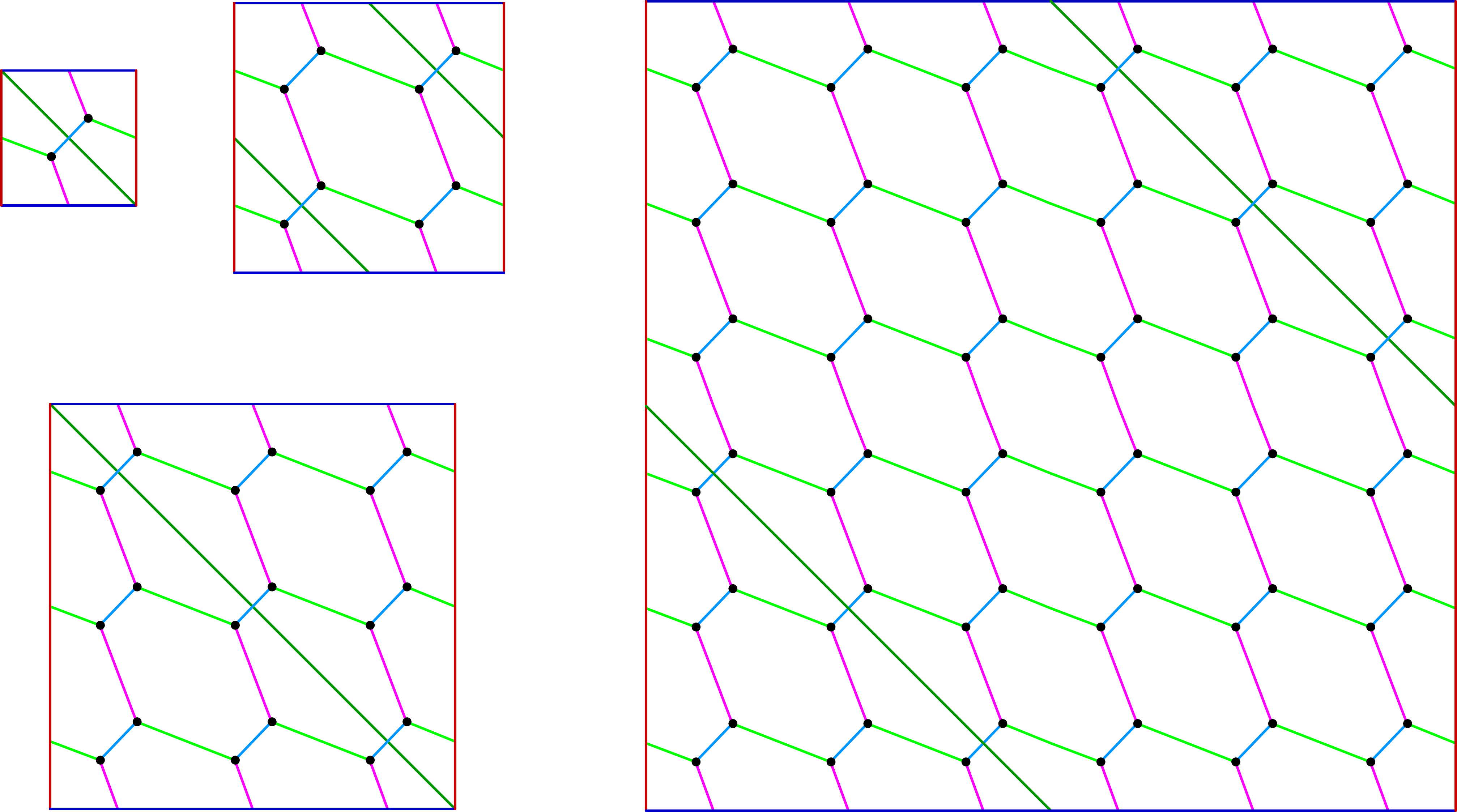}
\caption{Examples of the $(1,0;d^2,d)$--bridge trisections for the complex curves of degree $d$ in $\CP^2$. Shown are the cases $d=1,2,3,6$.}
\label{fig:Curves_Plane_Square}
\end{figure}

	The union of the red arcs and the blue arcs is a $d$--component unlink, braided around the core of the handlebody $H_\beta$.  Therefore the link $\cL_1$ bounds a collection of trivial disks $\cD_1$ in $\del X_1$ that each intersect the core of $H_\beta$ geometrically once.  These disks are Seifert disks for the trivial $d$--strand braid.  Next, each component of the unlink $\cL_2$ given as the union of the blue arcs and the green arcs is an unknot in $d$--bridge position.  Consequently, each disk of $\cD_2$ contributes $d-1$ bands to the surface.  (See Lemma~3.3 of~\cite{MZ-Bridge}.)  We can choose the cores of these bands to be the green arcs $\cC$ and the surface framing implies that these are positive bands.  As a result, we see that the union $\cL_3$ of the green arcs and the red arcs, considered as a link in $\del X_1$, is the torus link $T(d,d)$ bounding its canonical Seifert surface.  Viewing $\cL_3$ as a link in $\del X_3$, however, it is the $d$--component unlink as each component is isotopic to $\gamma$ in the central surface and therefore bounds a compressing disk in $H_\gamma$.  These disks comprise $\cD_3$.  It is now clear from the above discussion that we have constructed a bridge trisection of $\cC_d$.  Figure \ref{fig:Banded_Link_Cubic} illustrates this correspondence for the cubic curve $\cC_3$.
\end{proof}

\begin{remark}
	A third approach to construct a shadow diagram of $\cC_d$ is by explicit computation in homogeneous coordinates.  Define a finite, holomorphic map $\pi\colon \CP^2 \rightarrow \CP^2$ in homogeneous coordinates by setting
\[\pi_d([z_1:z_2:z_3]) = [z_1^d: z_2^d: z_3^d].\]
Clearly, $\cV_d = \pi^{-1}_d(\cV_1)$.  Away from the variety $W = \{z_1z_2z_3 = 0\}$, the map $\pi_d$ is an unbranched covering of degree $d^2$.   In particular, it restricts to a $d^2$--fold covering map $\pi\colon T^2 \rightarrow T^2$ along the central surface. Pulling back the shadow diagram for $\cC_1$ by $\pi_d$ is exactly the shadow diagram of $\cC_d$ depicted in Figure \ref{fig:Curves_Plane_Square}.
\end{remark}

\begin{theorem}\label{thrm:projective-curve-efficient}
	The pair $(\CP^2,\cC_d)$ admits an efficient $(1,b)$--bridge trisection with $b = (d-1)(d-2) + 1$.
\end{theorem}

\begin{proof}
	As shown above, $(\CP^2,\cC_d)$ admits a $(1,0;d^2,d)$--bridge trisection.  To prove the theorem, it suffices to show that these bridge trisections can be maximally destabilized; i.e., that we can destabilize along $d-1$ arcs of each color.  In fact, we will perform $d-1$ \emph{balanced} destabilization using hexagons of the form shown in Figure~\ref{fig:Destab_Hexagon}.
	
	Whenever we see a hexagon that is embedded in the complement of the arcs and curves of the shadow diagram, we can destabilize along three non-adjacent arcs of its boundary, which will necessarily be of distinct colors, if the following three conditions holds:  We require that arcs of $\cA_i$ in the boundary of the hexagon correspond to arcs of $\taus_i$ that are contained in distinct components of $\cL_i$.  If the hexagon is of this sort, then we can destabilize, as described by Figure~5 of~\cite{MZ-GBT} and the surrounding text.

\begin{figure}[h!]
\centering
\includegraphics[width=.75\textwidth]{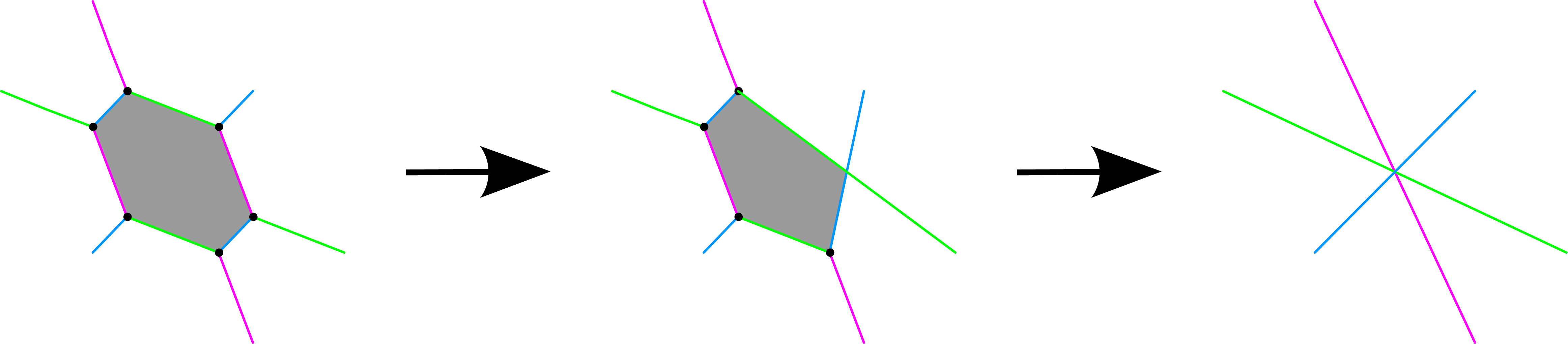}
\caption{When each edge of a hexagon of shadow arcs connects distict components of a one of the boundary links $L_i = \partial D_i$, the entire hexagon can be removed in a balanced type of deperturbation of the bridge trisection.}
\label{fig:Destab_Hexagon}
\end{figure}

	Returning the special case of shadow diagrams of the sort illustrated in Figure\ref{fig:Curves_Plane_Square}, it is easy to see that we can find $d-1$ such hexagons such that no two hexagons are in the same row, column, or diagonal.  Figure~\ref{fig:Destab_Sextic-Cubic_Square} shows two examples of such choices, which illustrate the general process.  Note that there is a slight distinction of cases imposed by the parity of $d$: When $d$ is odd, we choose the green curve to be the off-diagonal of the square, but when $d$ is even, we choose a slightly different representation for the green curve.  This is done to make the selection of hexagons easier.

\begin{figure}[h!]
\centering
\includegraphics[width=.9\textwidth]{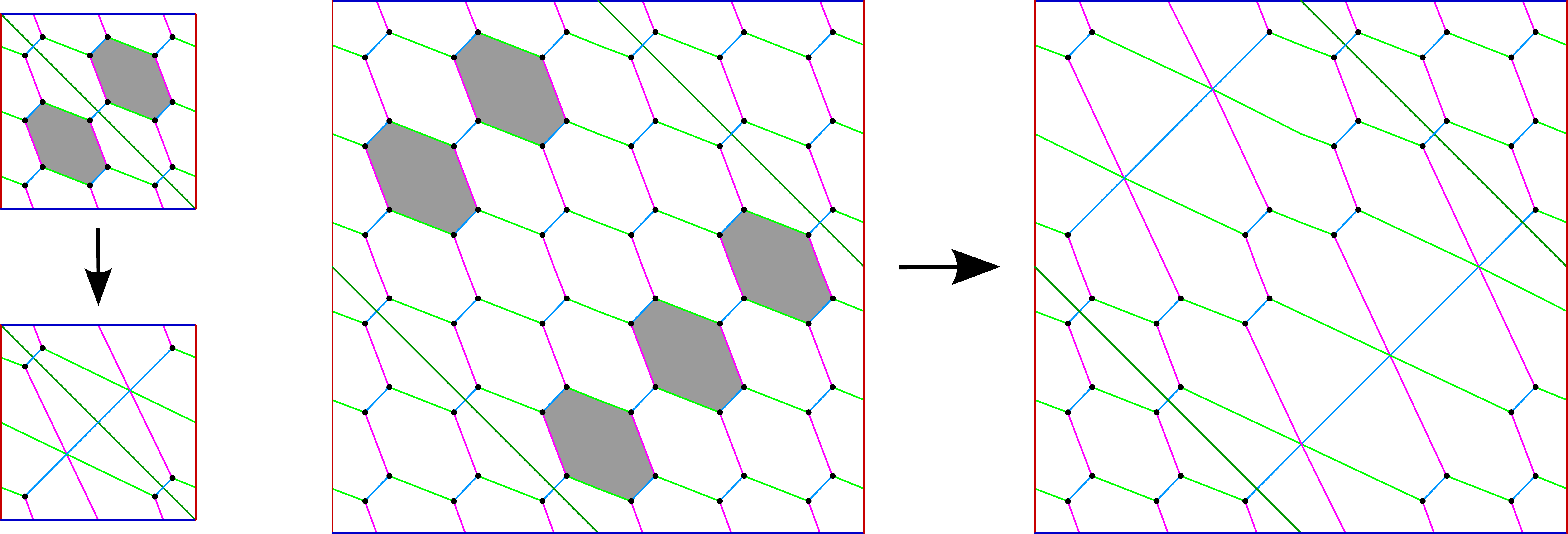}
\caption{Illustrative examples of the destabilization process used to turn the $(1,0;d^2,d)$--bridge trisections for the complex curves of degree $d$ in $\CP^2$ into efficient bridge trisections with bridge number $b = d^2-3d+3$. Shown are the instances of (Left) $d=3$ and (Right) $d=6$.}
\label{fig:Destab_Sextic-Cubic_Square}
\end{figure}

	Note that each destabilization decreases the bridge trisection value $b$ by one.  It follows that we end up with a $(1,0;d^2 - 3(d-1),1)$--bridge trisection.  Observing that $d^2-3d+3 = (d-1)(d-2)+1$ completes the proof.
\end{proof}

Up to this point, we have represented the toroidal trisection surface for $\CP^2$ a as a square with opposite edges identified.  This surface can also be represented as a hexagon with opposite edges identified. In Section~\ref{sec:menagerie}, Figure~\ref{fig:Destabs_Hex} shows how the destabilization process described above can be modified and applied under this rendering, and Figure~\ref{fig:Curves_Plane_Hex} shows efficient shadow diagrams for $\cC_1$, $\cC_2$, $\cC_3$, and $\cC_6$ under this rendering.

%%%%%%%%%%%%%%%%%%%%%%%%%%%%%%%%%%%%%%%%%%%%%%%%%%%%%%%%%%%%
%%%%%%%%%%%%%%%%%%%%%%%%%%%%%%%%%%%%%%%%%%%%%%%%%%%%%%%%%%%%
\section{Branched covers of $\CP^2$}\label{sec:cp2_branch}
%%%%%%%%%%%%%%%%%%%%%%%%%%%%%%%%%%%%%%%%%%%%%%%%%%%%%%%%%%%%
%%%%%%%%%%%%%%%%%%%%%%%%%%%%%%%%%%%%%%%%%%%%%%%%%%%%%%%%%%%%

In this section, we describe trisections of the four-manifolds obtained as cyclic branched covers over complex curves in $\CP^2$.

\subsection{Branched covers}

Let $Y$ be a complex surface, let $L$ be a holomorphic line bundle, and let $L_n = L^{\otimes n}$ be its $n^{\text{th}}$ tensor power.  Consider the holomorphic map of line bundles
$$ \zeta_n: L \rightarrow L_n$$
modeled as $z \mapsto z^n$ on each fiber.  The map $\zeta_n$ is a $n$--fold branched covering of $L$ over $L_n$ with branch set $Y$, viewed as the 0--locus of the line bundle.  Take a holomorphic section $s$ of $L_n$ with 0--set $B = s^{-1}(0)$.  Define $X = \zeta_n^{-1}(s(Y))$.  The restriction of $\zeta_n$ to $X$ determines a $n$--fold branched cover of $s(Y) \cong Y$ over the branch locus $B$.

The following is well-known (see \cite{GS}).

\begin{proposition}
\label{prop:branched-cover-data}
Let $\phi: X \rightarrow Y$ be a $n$--fold branched cover of complex surfaces along the complex curve $B$.  Then we have
\begin{align*}
c_1(X) &= \phi^*\left(c_1(Y) - \frac{n-1}{n} [B] \right)  & c_2(X) &= n \cdot c_2(Y) - (n-1) \cdot \chi(B)  \\
c_1(X)^2 &= n \cdot \left( c_1(Y) - \frac{n-1}{n} [B] \right)^2 & 
\sigma(X) &= n \cdot \sigma(Y) + \frac{n -1}{3} \left( \frac{n-1}{n} - 2 \right) [B]^2 %\\
%\chi_h(X) &= d \cdot \chi_h(Y) + \frac{d-1}{4} \left(  \left( \frac{d -1}{3d} - \frac{2}{3} \right) [B]^2 - \chi(B) \right)
\end{align*}
\end{proposition}

\subsection{Hypersurfaces in $\CP^3$.}

Define the projective variety
\[ S_d = \{w^d + x^d + y^d + z^d = 0 \}\]
in $\CP^3$. We note that all nonsingular hypersurfaces of degree $d$ in $\CP^3$ are linearly equivalent, hence diffeomorphic, to $S_d$.  Identify $\CP^3 \smallsetminus \{[1:0:0:0]\}$ with the total space of the holomorphic line bundle $\cL$ over the hyperplane $H \coloneqq \{w = 0\}$ with $c_1(\cL)$ Poincare dual to the projective line in $\CP^2$.  We can define a projection map
\[ \pi\colon [w:x:y:z] \mapsto [0:x:y:z] \]
for this bundle onto the 0--section $H$.  The restriction to $S_d$ is 1--to--1 along the curve $\{x^d + y^d + z^d = 0\}$ in $H$ and is $d$--to--1 over the complement of this curve.  Therefore, we get a $d$--fold branched cover over $\CP^2$ with branch locus $\cC_d$.

\begin{example}
\label{ex:branched-CP2}
The hypersurfaces $S_d$ are given as follows:
\begin{enumerate}
\item $S_2$ is diffeomorphic to $S^2 \times S^2$.
\item $S_3$ is diffeomorphic to $\CP^2 \# 6 \overline{\CP}^2$.
\item $S_4$ is diffeomorphic to $K3$.
\item For odd $d \geq 5$, the manifold $S_d$ is homeomorphic, but not diffeomorphic, to a connected sum of copies of $\CP^2$ and $\overline\CP^2$.
\item For even $d\geq 6$, the manifold $S_d$ is homeomorphic, but not diffeomorphic, to a connected sum of copies of $K3$ and $S^2\times S^2$.
\end{enumerate}
The last two results are highly non-trivial, relying on work of Freedman~\cite{Freedman}, Donaldson~\cite{Donaldson}, and Taubes~\cite{taubes}.  The particular nature of the connected sums is described in Proposition~\ref{prop:hypersurface-data}, below.
\end{example}

The topology of a complex hypersurface $S_d$ is well-understood.

\begin{proposition}[\cite{GS}]
\label{prop:hypersurface-data}
Let $S_d$ be the degree $d$ hypersurface in $\CP^3$.  
\begin{enumerate}
\item The surface $S_d$ is simply-connected with total Chern class
\[
c(S_d) = 1 + (4 - d) \zeta + (d^2 - 4d + 6) \zeta^2
\]
where $\zeta$ is the pullback of the generator $\alpha$ of $H^*(\CP^3;\ZZ)$ under the inclusion map $i: S_d \hookrightarrow \CP^3$.  Moreover, $\langle \zeta^2,[S_d] \rangle = d$.
\item The surface $S_d$ satisfies
\begin{align*}
\chi(S_d) &= d(d^2 - 4d + 6) & c_1^2(S_d) &= d(4-d)^2 \\
 \sigma(d) &= \frac{1}{3} \left( d(4-d)^2 - 2d(d^2 - 4d + 6) \right)  & b_2(S_d) &= d^3 - 4d^2 + 6d - 2 
\end{align*}
\item If $d$ is odd, then $Q_d \cong \lambda_d \langle 1 \rangle \oplus \mu_d \langle -1 \rangle$, 
where
\[ \lambda_d = \frac{1}{3} (d^3 - 6d^2 + 11d - 2) \qquad \mu_d = \frac{1}{3}(d-1)(2d^2 - 4d + 3) \]
Consequently, $S_d$ is homeomorphic to $(\lambda_d \CP^2) \# (\mu_d \overline{\CP}^2)$.
\item If $d$ is even, then $Q_d \cong l_d (-E_8) \oplus m_d H$, 
where
\[ l_d = \frac{1}{24} d(d^2 - 4) \qquad m_d = \frac{1}{3} (d^3 - 6d^2 + 11d - 2)\]
Consequently, $S_d$ is homeomorphic to $(\kappa_d K3) \# (\eta_d  \CP^1 \times \CP^1)$, 
where
\[ \kappa_d = \frac{1}{2} l_d \qquad \eta_d = m_d - \frac{3}{2} l_d\]
\end{enumerate}
\end{proposition}

Taking branched covers of bridge trisections of the complex curves $(\CP^2,\cC_d)$, we can construct trisections of the hypersurfaces $S_d$.

\begin{theorem}\label{thm:tri_hyper}
The hypersurface $S_d$ admits an efficient $(g,0)$--trisection, where $$g = d^3 - 4d^2 + 6d - 2.$$
\end{theorem}

\begin{proof}
	By Theorem \ref{thrm:projective-curve-efficient}, the pair $(\CP^2,\cC_d)$ admits an (efficient) $(1,0;b,1)$--bridge trisection with $b=d^2-3d+3$.  By Corollary \ref{coro:eff_branch}, the $d$--fold cover of $\CP^2$ branched along $\cC_D$ admits a $(g',0)$--trisection, with 
$$g'= d + (d-1)(d^2-3d+2) = (d-1)^2(d-2) + d = d^3 - 4d^2 + 6d - 2.$$
\end{proof}

\begin{figure}[h!]
\centering
\includegraphics[width=.65\textwidth]{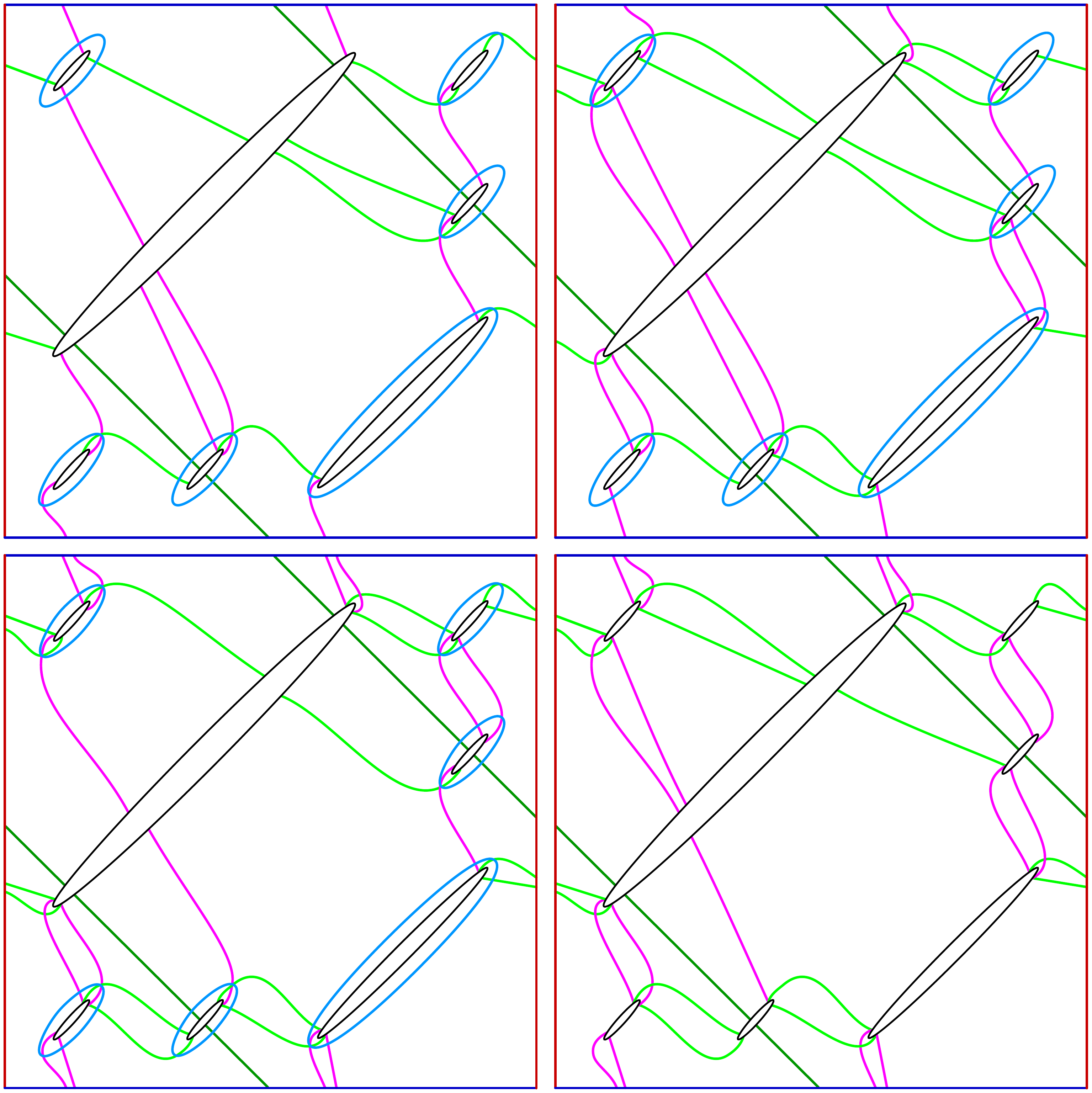}
\caption{A $(22,0)$--trisection of $K3$, thought of as the 4--fold cover $S_4$ of $\CP^2$ branched along the quartic $\cC_4$.  Each square corresponds to a torus once opposite edges are identified.  The northwestern edge of each ellipse is identified with the southeastern edge of the corresponding ellipse in the clockwise-adjacent square}
\label{fig:S4}
\end{figure}

For $d=2$ and $d=3$, we have $S_2\cong S^2\times S^2$ and $S_3\cong \CP^2 \# 6 \overline{\CP}^2$.  Each of these manifolds admits a trisection that is the connected sum of genus one  and genus two trisections.  (Recall that such trisections are called \emph{standard}.)  In the case of $S_2$, it is straight-forward to see that the trisection corresponding to the branched covering described above is, in fact, the standard one.  (Note that this is also implied by the fact that there is a unique irreducible genus two trisection~\cite{MZ-Genus-2}).  On the other hand, it is not clear whether or not the genus seven trisection of $S_3$ corresponding to the branched covering is standard, or even reducible.  This diagram is given in Figure~\ref{fig:cubic_branch_square}; the marked points and lighter-shaded arcs encode the lift of $\cC_3$, and can be ignored, presently.

\begin{question}
	Is the trisection of $S_3$ arising as the 3--fold branched cover of the efficient bridge trisection of $(\CP^2,\cC_3)$ standard?  Is it even reducible?
\end{question}

A trisection diagram for $S_4$ is given in Figure~\ref{fig:S4}.  In Section~\ref{sec:menagerie}, we give trisection diagrams for $S_5$ and $S_6$ in Figures~\ref{fig:S5} and~\ref{fig:S6}.

%%%%%%%%%%%%%%%%%%%%%%%%%%%%%%%%%%%%%%%%%%%%%%%%%%%%%%%%%%%%%%%%
\subsection{More branched covers of complex curves in $\CP^2$}
%%%%%%%%%%%%%%%%%%%%%%%%%%%%%%%%%%%%%%%%%%%%%%%%%%%%%%%%%%%%%%%%

Above, we considered the $d$--fold cover of $\CP^2$, branched along the curve $\cC_d$, which resulted in the hypersurface $S_d$.  More generally, we can take the $n$--fold cover of $\CP^2$, branched along the curve $\cC_d$, provided that $n$ divides $d$.  We denote the resulting complex surface by  $\mathcal{Q}_{d,n}$.

\begin{proposition}
Let $\phi: \mathcal{Q}_{d,n} \rightarrow \CP^2$ be a $n$--fold branched cover along $\cC_d$.  Then we have
\begin{align*}
c_1(\mathcal{Q}_{d,n}) &= (3 - m(n-1))\phi^*(\alpha)  & c_2(\mathcal{Q}_{d,n}) &= 3n - (n-1)(3d - d^2) \\
c_1(\mathcal{Q}_{d,n})^2 &= n(3 - m(n-1))^2 & 
\sigma(\mathcal{Q}_{d,n}) &= n  + \frac{n -1}{3} \left( \frac{n-1}{n} - 2 \right) d^2 %\\
%\chi_h(X) &= d \cdot \chi_h(Y) + \frac{d-1}{4} \left(  \left( \frac{d -1}{3d} - \frac{2}{3} \right) [B]^2 - \chi(B) \right)
\end{align*}
where $d = m \cdot n$ and $\alpha = c_1(\CP^2)$.

\end{proposition}

\begin{theorem}\label{thm:Sdn}
The complex surface $\cQ_{d,n}$ admits an efficient $(g,0)$--trisection where $$g = (n-1)(d-1)(d-2)+n.$$
\end{theorem}

\begin{proof}
	By Theorem \ref{thrm:projective-curve-efficient}, the pair $(\CP^2,\cC_d)$ admits an (efficient) $(1,0;b,1)$--bridge trisection with $b=d^2-3d+3$.  By Corollary \ref{coro:eff_branch}, the $n$--fold cover of $\CP^2$ branched along $\cC_D$ admits a $(g',0)$--trisection, with 
$$g'= n + (n-1)(d^2-3d+2) = (n-1)(d-1)(d-2)+n.$$
\end{proof}

\begin{example}
	For small values of $n$ and $d$ with $n$ dividing $d$, the manifolds $\cQ_{d,n}$ are standard:
\begin{enumerate}
\item $\mathcal{Q}_{4,2}$ is diffeomorphic to $\CP^2 \# 7 \overline{\CP}^2$.
\item $\mathcal{Q}_{6,2}$ is diffeomorphic to  $K3$.
\end{enumerate}
For larger values, we get exotic copies of standard manifolds.
\begin{enumerate}[resume]
\item $\mathcal{Q}_{6,3}$ is homeomorphic, but not diffeomorphic, to $7 \CP^2 \# 36 \overline{\CP}^2$.
\end{enumerate}
(See the comments of Example~\ref{ex:branched-CP2}.)
\end{example}

Trisection diagrams for $\cQ_{4,2}$, $\cQ_{6,2}$, and $\cQ_{6,3}$ are shown in Figures~\ref{fig:Q_4_2},~\ref{fig:Q_6_2_square},~\ref{fig:Q_6_3_square},~\ref{fig:Q_6_2_hex}. Again, we can ask whether the trisections produced in this way are standard.

\begin{figure}[h!]
\centering
\includegraphics[width=.5\textwidth]{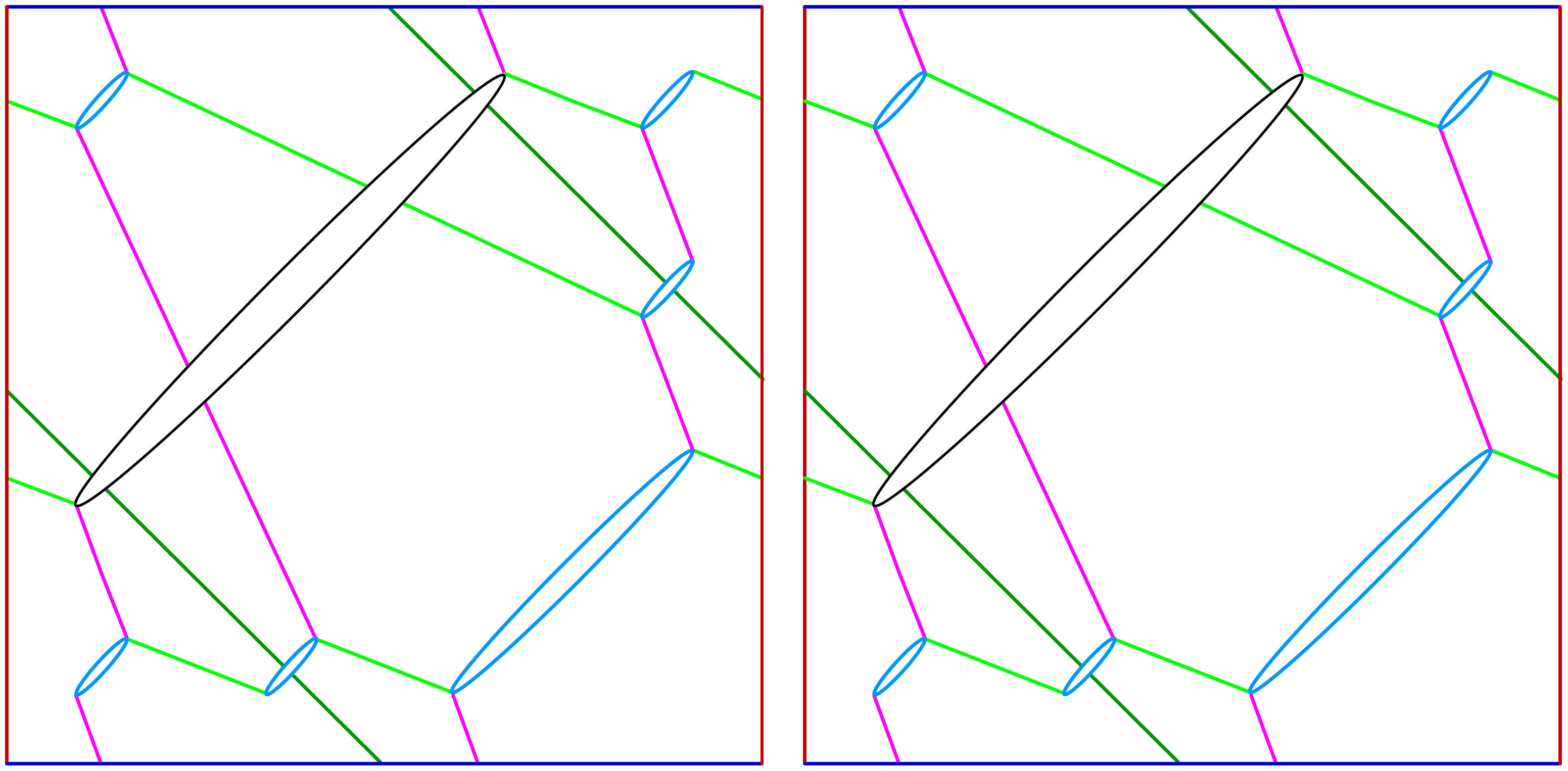}
\caption{An $(8,0)$--trisection of $\CP^2\#7\overline\CP^2$, thought of as the 2--fold cover $\cQ_{4,2}$ of $\CP^2$ branched along the quartic $\cC_4$.  Each square corresponds to a torus once opposite edges are identified.  Each ellipse in the left square is identified with the corresponding ellipse in the right square via a reflection across its major axis.}
\label{fig:Q_4_2}
\end{figure}

\begin{figure}[h!]
\centering
\includegraphics[width=.75\textwidth]{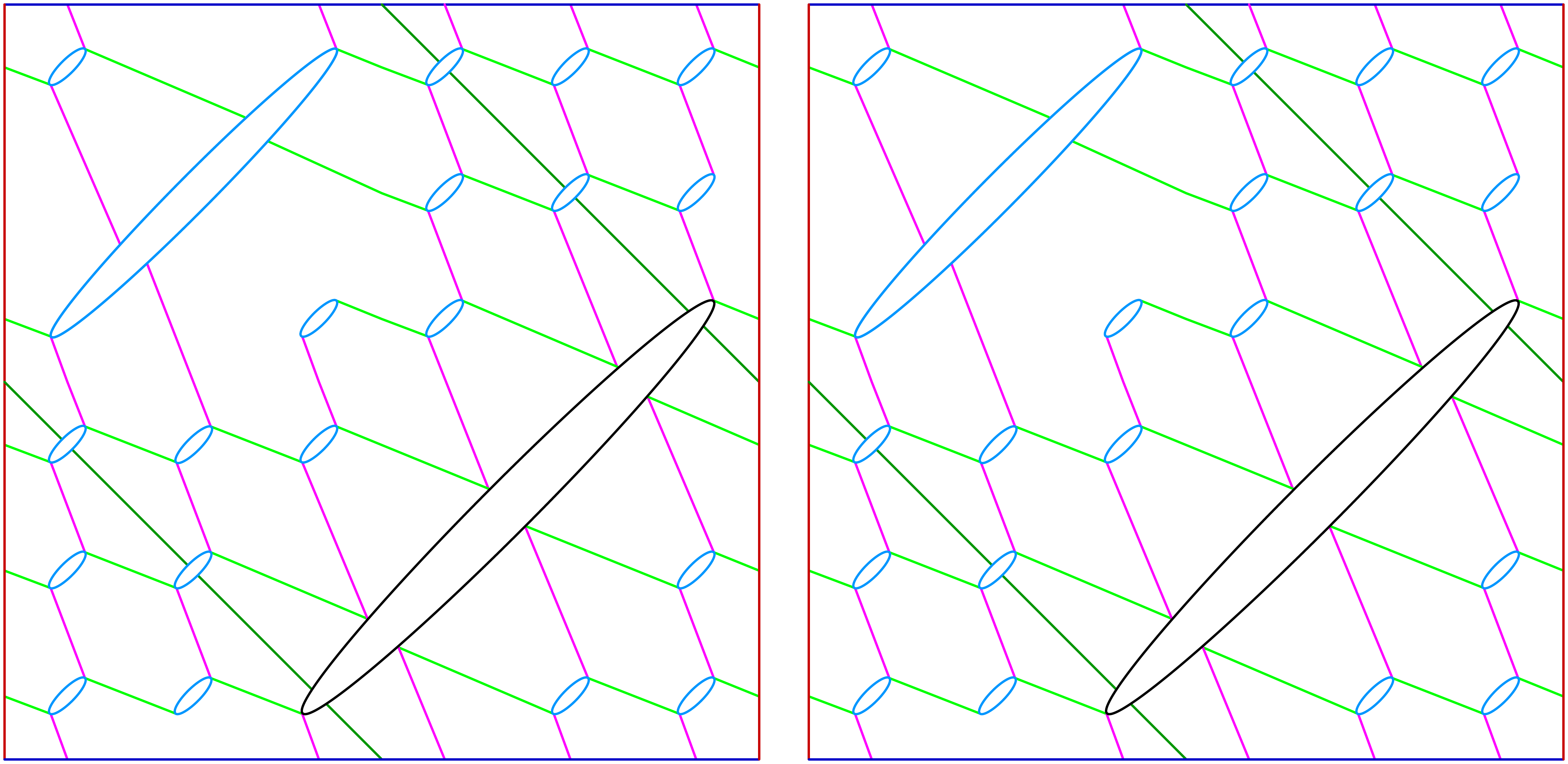}
\caption{A $(22,0)$--trisection of $K3$, thought of as the 2--fold cover $\cQ_{6,2}$ of $\CP^2$ branched along the sextic $\cC_6$.  Each square corresponds to a torus once opposite edges are identified.  Each ellipse in the left square is identified with the corresponding ellipse in the right square via a reflection across its major axis.}
\label{fig:Q_6_2_square}
\end{figure}

\begin{question}
	Is the trisection of $\mathcal{Q}_{4,2} \cong \CP^2 \# 7 \overline{\CP}^2$ standard?
\end{question}

\begin{question}
	Are the trisections of $K3$ arising from its descriptions as $S_4$ and as $\mathcal{Q}_{6,2}$ equivalent?
\end{question}

%%%%%%%%%%%%%%%%%%%%%%%%%%%%%%%%%%%%%%%%%%%%%%%%%%%%%%%%%%%%
%%%%%%%%%%%%%%%%%%%%%%%%%%%%%%%%%%%%%%%%%%%%%%%%%%%%%%%%%%%%
\section{Complex curves in $\CP^1 \times \CP^1$}\label{sec:S2xS2}
%%%%%%%%%%%%%%%%%%%%%%%%%%%%%%%%%%%%%%%%%%%%%%%%%%%%%%%%%%%%
%%%%%%%%%%%%%%%%%%%%%%%%%%%%%%%%%%%%%%%%%%%%%%%%%%%%%%%%%%%%

In this section, we construct efficient bridge trisections of complex curves of $\CP^1 \times \CP^1$ and prove Theorem \ref{thmx:eff_bridge}(2).

\subsection{$S^2 \times S^2$ as a branched cover of $(\CP^2,\cC_2)$.}

As mentioned in Example \ref{ex:branched-CP2}, the complex surface $\CP^1 \times \CP^1 \cong S^2 \times S^2$ is the double branched cover of the conic curve $\cC_2$ in $\CP^2$.  Following the method of Subsection \ref{subsec:diags_branch}, we can obtain an explicit $(2,0)$--trisection diagram of $S^2 \times S^2$ as follows.  

Starting from the $(4,2)$--bridge trisection of $\cC_2$, we can destabilize using the hexagon in the center to get a (minimal) $(1,1)$--bridge trisection of the quadratic.  We choose a branch cut along the $\cB$--arc and then glue two copies of the torus together.  See Figure \ref{fig:conic_cover}.  Note that we obtain a $(1,1)$--bridge trisection of the branch locus in $S^2 \times S^2$, which is a sphere of bidegree $(1,1)$.  We use the canonical orientation on the branch locus coming from the complex structure to orient the class in $H_2(S^2 \times S^2;\ZZ)$.

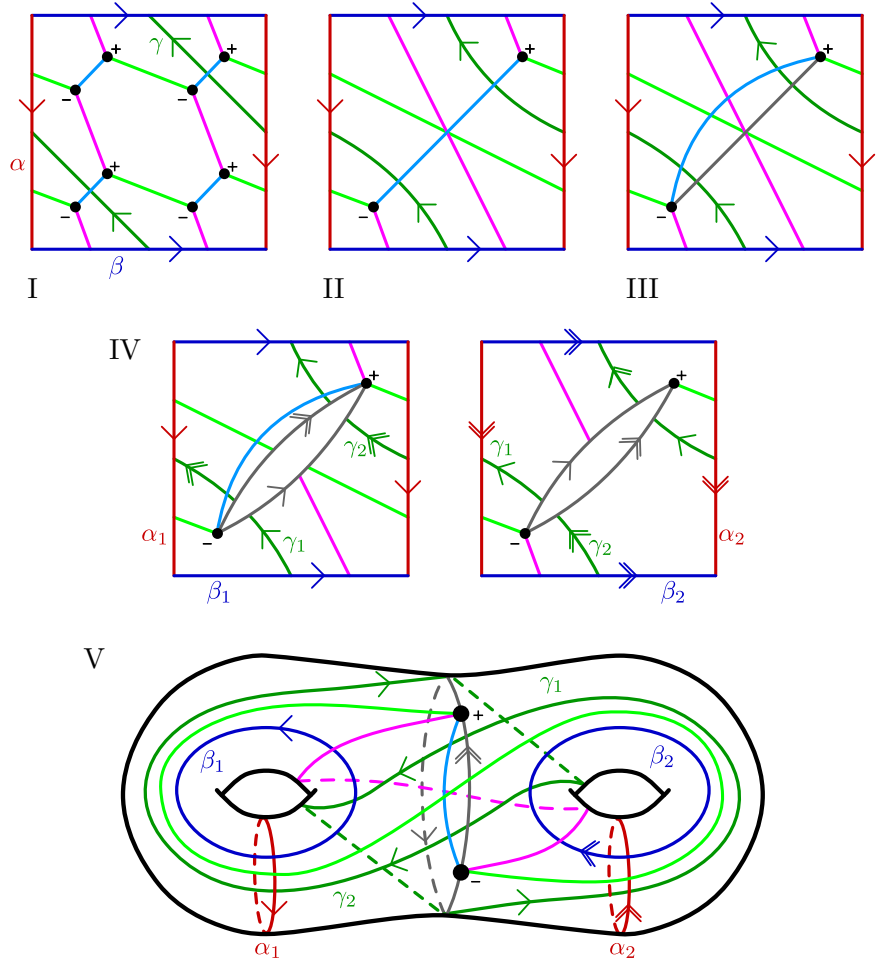
\begin{figure}[h!]
\centering
\labellist
	\large\hair 2pt
	\pinlabel I at 20 440
	\pinlabel II at 215 440
	\pinlabel III at 415 440
	\pinlabel IV at 80 400
	\pinlabel V at 60 200
\endlabellist
\includegraphics[width=.75\textwidth]{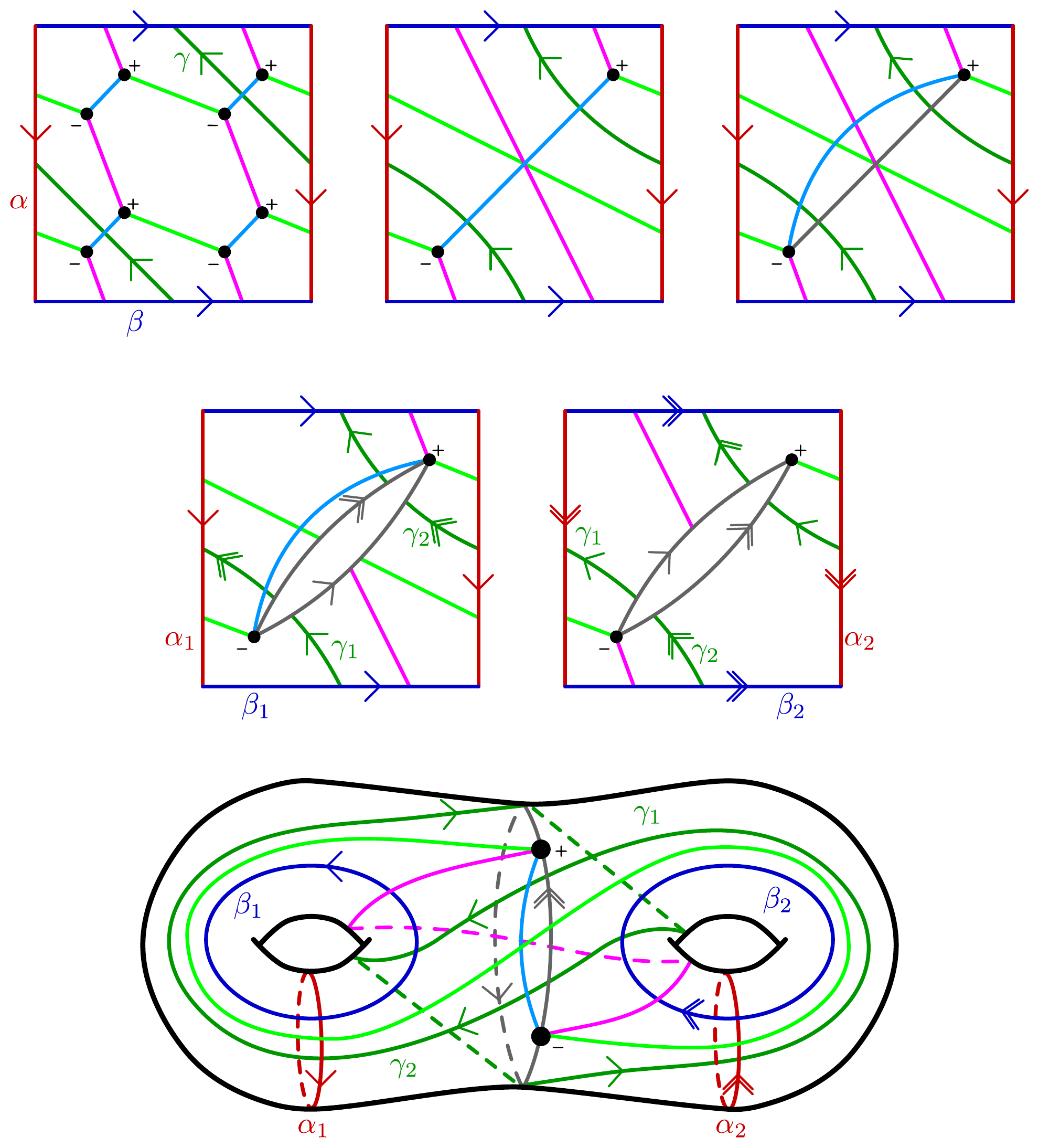}
\caption{A trisection of $S^2\times S^2$, viewed as the 2-fold branched cover of the quadric.  Diagram I is a shadow diagram of the $(1,0;4,2)$--bridge trisection and Diagram II depicts the efficient bridge trisection obtained by contracting a balanced hexagon.  The branch cut is identified in Diagram III and Diagram IV is a trisection diagram of the double branched cover.  Diagram V is equivalent to Diagram IV.}
\label{fig:conic_cover}
\end{figure}

%%%%%%%%%%%%%%%%%%%%%%%%%%%%%%%%%%%%%%%%%%%%%%%%%%%%%%
\subsection{Homology class computations}
%%%%%%%%%%%%%%%%%%%%%%%%%%%%%%%%%%%%%%%%%%%%%%%%%%%%%%

In this subsection, we describe how to compute the homology class of a surface $(X,\cK)$ in bridge position from a shadow diagram.  For simplicity, we assume that $H_2(X;\ZZ)$ is torsion-free.

Let $(\Sigma,\alphas,\betas,\gammas\}$ be a trisection diagram for $X$.  The cut systems $\{\alphas,\betas,\gammas\}$ span subspaces $L_{\alpha},L_{\beta},L_{\gamma} \subset H_1(\Sigma;\ZZ)$ that are Lagrangian with respect to the homology intersection pairing $\langle - , - \rangle_{\Sigma}$ on $H_1(\Sigma;\ZZ)$.  According to \cite{FKSZ}, the homology of $X$ can be computed via the chain complex
\[\xymatrix{
0 \ar[r] & \ZZ \ar[r]^-{0} & (L_{\alpha} \cap L_{\gamma}) \oplus (L_{\beta} \cap L_{\gamma}) \ar[r]^-{\del_3} & L_{\gamma} \ar[r]^-{\del_2} & \text{Hom}(L_{\alpha} \cap L_{\beta}, \ZZ) \ar[r]^-{0} & \ZZ \ar[r] & 0 }\]
where
\begin{align*}
\del_3(x,y) &= x+ y \\
\del_2(x) &= \langle -, x \rangle_{\Sigma}
\end{align*}
In particular,
\begin{align*}
\text{ker}(\del_2) &= (L_{\alpha} \cap L_{\beta})^{\perp} = L_{\alpha} + L_{\beta} \\
H_2(X) &= \frac{L_{\gamma} \cap (L_{\alpha} + L_{\beta})}{(L_{\gamma} \cap L_{\alpha}) + (L_{\gamma} \cap L_{\beta})} 
\end{align*}
For any class $A \in H_2(X)$, we can find coefficients $\{c_i\}$ such that $\sum c_i [\gamma_i]$ representing $A$ is the above chain complex.  Furthermore we can also find coefficients such that
\[\sum c_i [\gamma_i] = \sum  (a_j [\alpha_j] +  b_j [\beta_j]).\]
As a result, there is a 2-chain $\phi$ in $\Sigma$ that is locally constant on $\Sigma \smallsetminus (\alphas \cup \betas \cup \gammas)$ and satisfying
\[\del \phi = \sum (a_j \alpha_j + b_j \beta_j) - \sum c_i \gamma_i.\]
Given a point $p \in \Sigma \smallsetminus (\alpha \cup \beta \cup \gamma)$, let $n_p(\phi)$ denote the local multiplicity of $\phi$ at $p$.  

\begin{proposition}
Let $\frak D = ((\alphas,\cA),(\betas,\cB),(\gammas,\cC))$ be a shadow diagram for $(X,\cK)$.  Let $A$ be a class in $H_2(X,\ZZ)$ and $\phi_A$ a corresponding 2--chain.  Then
\[ [A] \cdot [\cK] = \sum_{p \in \bold x_+} n_{p}(\phi_A) - \sum_{q \in \bold x_-} n_{q}(\phi_A)\]
\end{proposition}

\begin{proof}
We can represent the homology class $A$ by a 2--chain consisting of the union of $a_j$ copies of the compressing disk $D_{\alpha,j}$, $b_j$ copies of the compressing disk $D_{\beta,j}$, $c_i$ copies of the compressing disk $D_{\gamma,i}$ and $\phi_A$.  The 2--chain representing $A$ lives completely in the spine of the trisection.  The intersection of $\cK$ with the spine consists of the tangles $\tau_{\alpha},\tau_{\beta},\tau_{\gamma}$, which are disjoint from the compressing disks bounded by the cut systems $\alphas,\betas,\gammas$.  Thus, the only intersection points between $A$ and $\cK$ occur at the bridge points and the algebraic intersection number at each bridge point is exactly the local multiplicity of $\phi$ at that bridge point, counted with sign.
\end{proof}

\begin{proposition}
The six $(1,1)$--bridge presentations in Figure \ref{fig:six_spheres} all depict the holomorphic sphere in bidegree $(1,0)$.
\end{proposition}

\begin{proof}
These are six equivalent bridge presentations of the bidegree $(1,0)$ sphere, related by handlesliding the bridge arcs over the compressing disks and isotoping the bridge points.  Specifically, (a) Diagram II is obtained from Diagram I by handlesliding $\cA$ across $\alpha_1$; (b) Diagram IV is obtained from Diagram II by handlesliding $\cC$ across $\gamma_1$; (c) Diagram III is obtained from Diagram IV by handlesliding $\cB$ across $\beta_2$; (d) Diagram V is obtained from diagram III by handlesliding $\cA$ across $\alpha_1$; (e) Diagram VI is obtained from Diagram V by handlesliding $\cC$ across $\gamma_1$.
\end{proof}

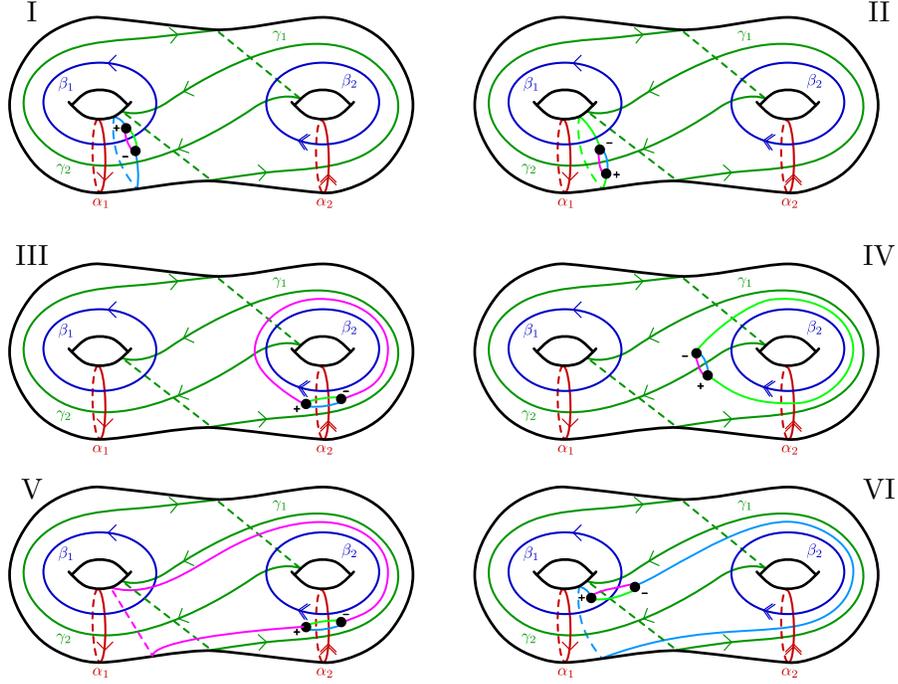
\begin{figure}
\centering
\labellist
	\large\hair 2pt
	\pinlabel I at 30 690
	\pinlabel II at 900 690
	\pinlabel III at 30 440
	\pinlabel IV at 900 440
	\pinlabel V at 30 200
	\pinlabel VI at 900 200
\endlabellist
\includegraphics[width=.75\textwidth]{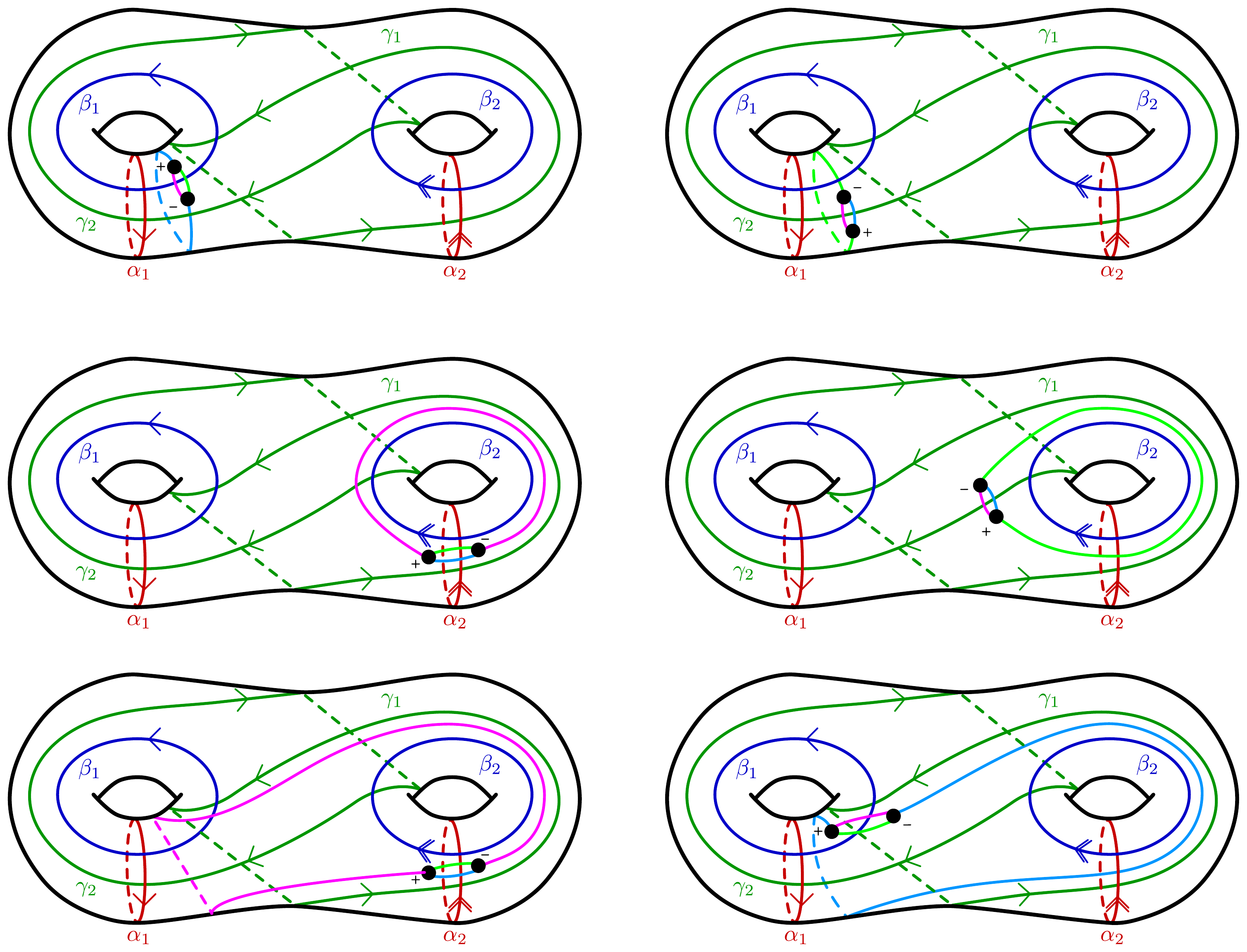}
\caption{Six bridge presentations of the bidegree $(1,0)$ holomorphic sphere in $\CP^1 \times \CP^1$.  }
\label{fig:six_spheres}
\end{figure}

\begin{remark}
\label{rem:attaching-wrong-color-F0}
It is useful to think of the shadow diagrams in Figure \ref{fig:six_spheres} as `attaching circles of the wrong color.'  In particular, in Diagram I the shadow diagram is a `blue $\alpha_1$'; in Diagram II it is a `green $\alpha_1$'; etc....
\end{remark}

\subsection{Curves of bidegree $(p,q)$}

\begin{theorem}
\label{thrm:s2xs2_eff_branch}
The complex curve $(\CP^1 \times \CP^1, \cC_{p,q})$ has an efficient $b$--bridge trisection, where $$b = 1 + 2(p-1)(q-1).$$
\end{theorem}

\begin{proof}
Up to isotopy, the surface $\cC_{p,q}$ can be obtained by taking $p$ copies of the bidegree $(1,0)$ sphere and $q$ copies of the bidegree $(0,1)$ sphere and resolving the $pq$ nodes to obtain a surface of genus $(p-1)(q-1)$.  We will follow this strategy, starting with a shadow diagram for $p$ copies of the $(1,0)$ sphere and $q$ copies of the $(0,1)$ sphere, then resolving crossings.  We will then show how to destabilize to get an efficient bridge trisection.  For concreteness, Figure \ref{fig:s2xs2_proof} depicts the steps for obtaining the curve $\cC_{6,4}$.

First, take $p$ copies of a `blue $\alpha_1$' and $q$ copies of a `red $\beta_1$'.  This gives a bridge trisection of the immersed collection of spheres.

Second, resolve the $pq$ intersections of $\cA$ and $\cB$.  The resolution is determined by orienting the arcs as follows: orient the $\cA$ arcs from the $(-)$--bridge points to the $(+)$--bridge points and orient the $\cB$ arcs from the $(+)$--bridge points to the $(-)$--bridge points.  The resolution is then the standard oriented resolution.

Third, handleslide all but one of the $\cA$ arcs in the bottom row across $\alpha_1$.  This creates $p(q-1)$ new intersection points between $\cA$ and $\cB$.  Now resolve $(p-1)(q-1)$ of these, leaving out the final collection on the right.  Destabilize along the starred edges in Diagram IV to get Diagram V and then isotope the bridge points to obtain Diagram VI.

{\bf Claim:} The result is a $(2,0;pq+(p-1)(q-1);1,p,q)$--bridge trisection for $\cC_{p,q}$.

It is clear that $\cA \cup -\cB$ is a diagram for the unknot in $Y_1 = H_{\alpha} \cup \overline H_{\beta}$.  Furthermore, $\cB \cup - \cC$ is isotopic to $p$ surface-framed pushoffs of $\beta_1$.  These all bound compressing disks in $H_{\beta}$ and therefore this is a diagram for the unlink in $Y_2$.  Finally, we can handleslide the bottom row of $\cA$ arcs in Diagram VI (not depicted) to see that $\cC \cup - \cA$ is also isotopic to $q$ surface-framed pushoffs.  Note that these do {\it not} bound compressing disks in $Y_3 = H_{\gamma} \cup \overline H_{\alpha}$.  However, the surface framing of $\beta_1$, considered as a knot in $Y_3$, is 0.  To see this, note that the trisection diagram for $S^2 \times S^2$ is triple-symmetric.  It is clear from Figure \ref{fig:six_spheres}, say, that the surface framing of the curves $\gamma_1,\gamma_2$, considered as knots in $Y_1 = H_{\alpha} \cup H_{\beta}$, is 0.  By symmetry, this also holds for $\beta_1$ in $Y_3$.

Now, to obtain an efficient trisection, we can destabilize along the starred edges in Diagram VI to obtain Diagram VII.  
\end{proof}

\begin{figure}
\centering
\labellist
	\large\hair 2pt
	\pinlabel I at 10 1685
	\pinlabel II at 750 1685
	\pinlabel III at 10 1250
	\pinlabel IV at 740 1250
	\pinlabel V at 10 810
	\pinlabel VI at 730 810
	\pinlabel VII at 380 385
\endlabellist
\includegraphics[width=.95\textwidth]{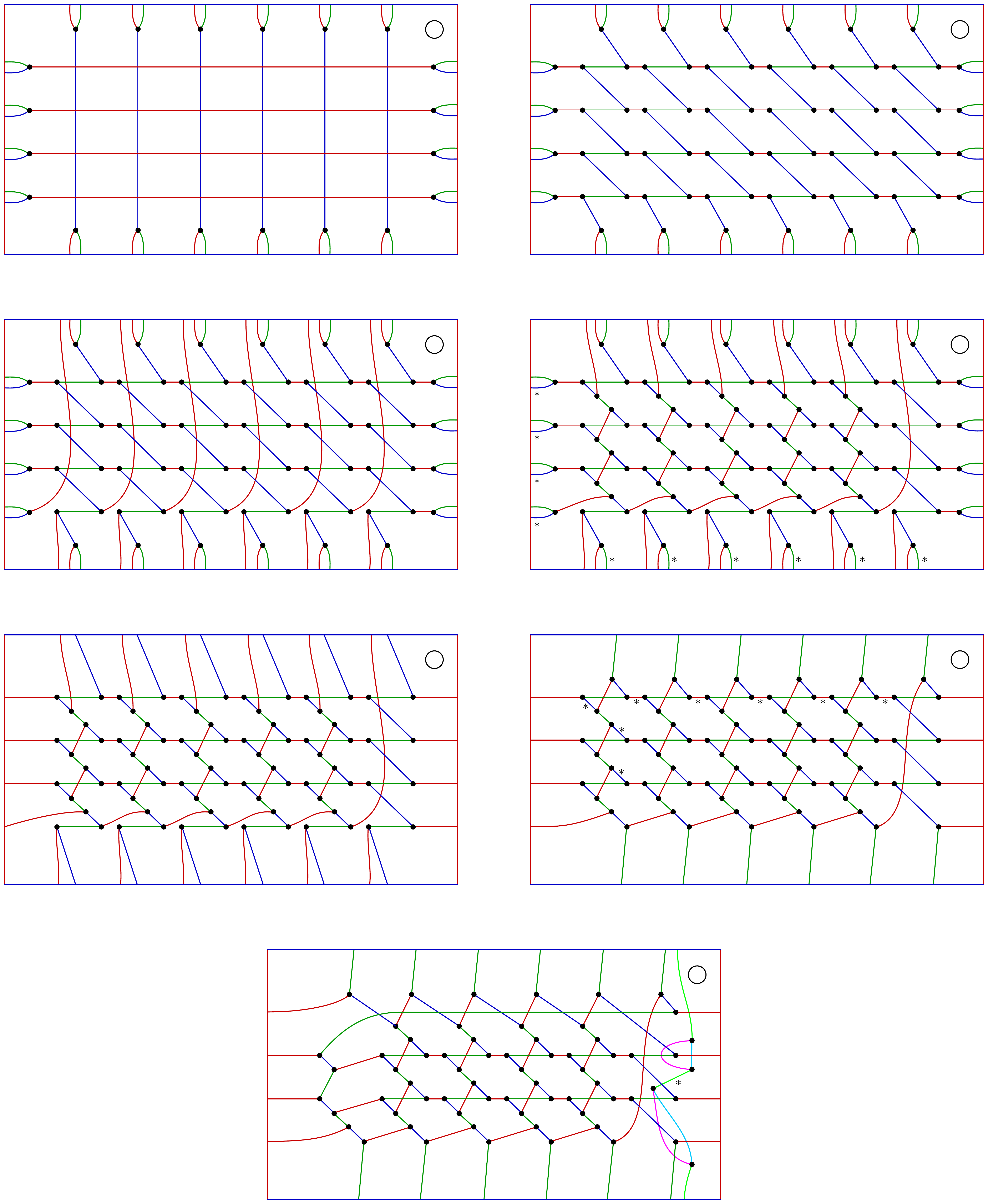}
\caption{A sequence of shadow diagrams of surfaces in bidegree $(6,4)$ in $\CP^2 \times \CP^2$.  Diagram I depicts a collection of immersed spheres, obtained by taking the standard bridge trisection of the component spheres.  Diagram VII is shadow diagram for an efficient trisection of $\cC_{6,4}$.  In addition, a shadow diagram for a stabilized bidegree $(1,0)$ sphere is depicted in Diagram VII; this lifts to a torus fiber in $E(2)$.}
\label{fig:s2xs2_proof}
\end{figure}
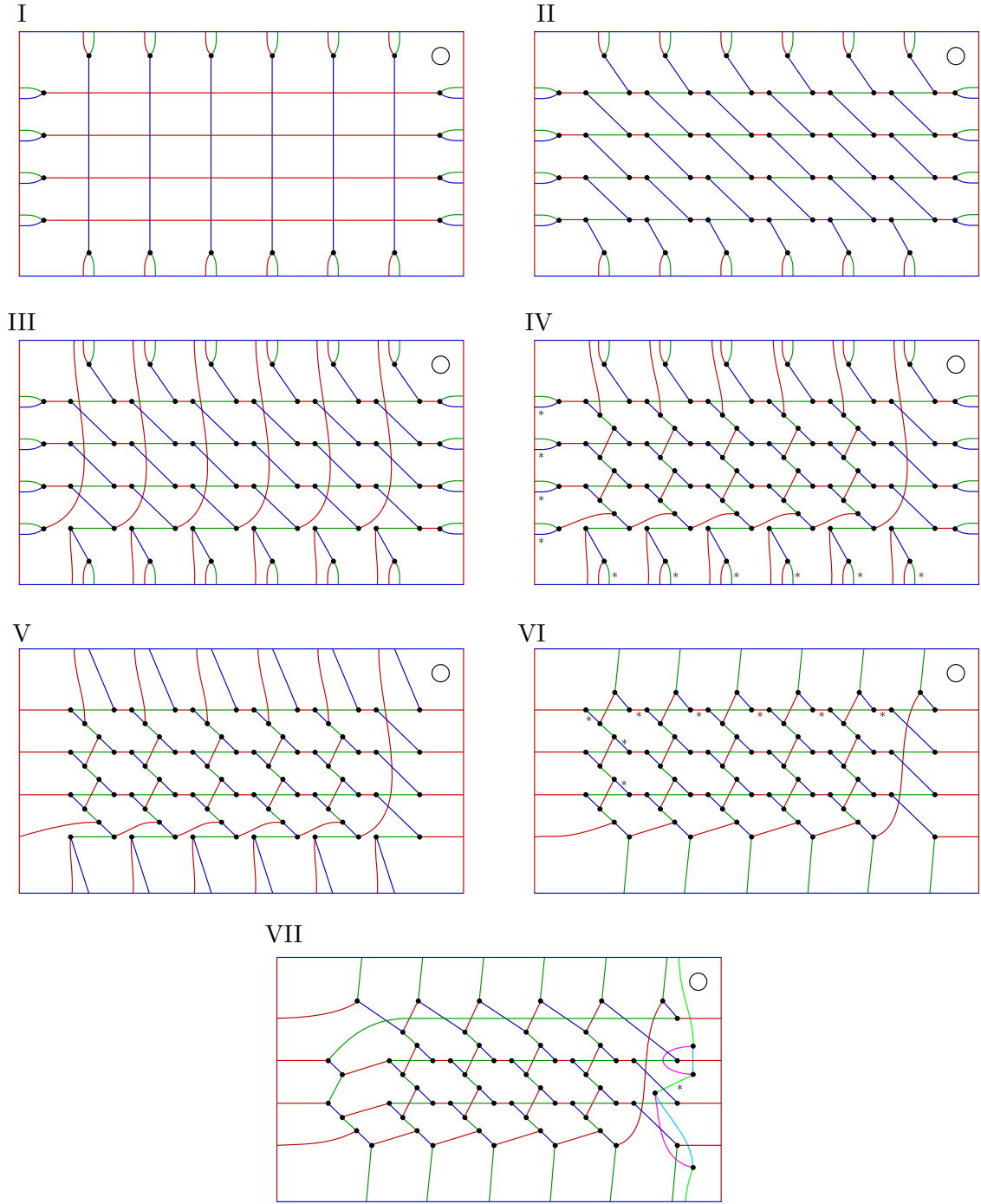

%%%%%%%%%%%%%%%%%%%%%%%%%%%%%%%%%%%%%%%%%%%%%%%%%%%%%%%%%%%%
%%%%%%%%%%%%%%%%%%%%%%%%%%%%%%%%%%%%%%%%%%%%%%%%%%%%%%%%%%%%
\section{Branched covers of $S^2\times S^2$}\label{sec:s2xs2_branch}
%%%%%%%%%%%%%%%%%%%%%%%%%%%%%%%%%%%%%%%%%%%%%%%%%%%%%%%%%%%%
%%%%%%%%%%%%%%%%%%%%%%%%%%%%%%%%%%%%%%%%%%%%%%%%%%%%%%%%%%%%

%%%%%%%%%%%%%%%%%%%%%%%%%%%%%%%%%%%%%%%%%%%%%%%%%%%%%%%
In this section, we recall some standard families of complex surfaces that can be constructed as branched covers of complex curves in $\CP^1 \times \CP^1$.  Combining the results of the previous section with Corollary \ref{coro:eff_branch}, we obtain efficient trisections of all of these complex surfaces.  In particular, this proves Theorem \ref{thmx:eff_tri}(3) and \ref{thmx:eff_tri}(4).

Throughout, we will let $\mathcal{X}_{p,q,n}$ denote the $n$--fold branched cover of the curve $\cC_{p,q}$ in $\CP^1 \times \CP^1$.  We will also use the shorthand $\cX_{p,q}$ to denote the 2-fold branched cover $\cX_{p,q,2}$.  The topology of these surfaces is described by the following proposition.

\begin{proposition}
Let $\phi: \cX_{p,q,n} \rightarrow \CP^1 \times \CP^1$ be an $n$--fold branched cover along $\cC_{p,q}$.  Then we have
\begin{align*}
c_2(\cX_{p,q,n}) &= 4n - 6 + 6(p-1)(q-1)\\
\sigma(\cX_{p,q,n}) &= -\frac{2n^2 - 2}{3n}pq\\
\end{align*}
In particular, if $n = 2$, then
\begin{align*}
c_2(\cX_{p,q}) &= -2 + 6(p-1)(q-1) & \sigma(\cX_{p,q}) &= -pq
\end{align*}
\end{proposition}

The main result of this section, obtained by combining Theorem \ref{thrm:s2xs2_eff_branch} with Corollary \ref{coro:eff_branch}, is the following theorem.

\begin{theorem}
The complex surface $\cX_{p,q,n}$ admits an efficient $(g,0)$--trisection, where
\[g = 4n - 8 + 6(p-1)(q-1)\]
In particular, if $n = 2$ then $\cX_{p,q}$ admits an efficient $(g,0)$-trisection where
\[g = 6(p-1)(q-1)\]
\end{theorem}

The remainer of this section is devoted to expanding on this theorem for some well-known families of complex surfaces.  

For the sake of exposition, we state the following well-known result, that double branched covers over complex curves can be interpreted as Lefschetz fibrations.

\begin{theorem}
\label{thrm:lefschetz-2fold}
The composition
\[ \xymatrix{
\cX_{2p+2,2q} \ar[r]^{\phi} & \CP^1 \times \CP^1 \ar[r]^{\pi} & \CP^1 }\]
where $\phi$ is the double branched cover over $\cC_{2p+2,2q}$ and $\pi$ denotes projection onto the second factor, is a genus $p$ Lefschetz fibration.
\end{theorem}

\subsection{Rational surfaces}  The double branched covers $\cX_{2,2q}$ of the pairs $(S^2 \times S^2,\cC_{2,2q})$ for $q > 0$ are rational surfaces.

\begin{proposition}
\label{prop:X2}
There is a diffeomorphism
$$ \cX_{2,2q} \cong S^2 \times S^2 \# (4q) \overline{\CP}^2 \cong \CP^2 \# (4q + 1) \overline{\CP}^2. $$
\end{proposition}

\begin{proof}
According to Theorem \ref{thrm:lefschetz-2fold}, the surfaces $\cX_{2,2q}$ are Lefschetz fibrations of genus 0.  Consequently, the vanishing cycles are contractible in the fiber.  Each singular fiber consists of a pair of spheres, each of self-intersection $-1$, intersecting positively transversely in a single point.  Equivalently, each singular fiber is the total transform of a generic fiber after blowing up at some point in that fiber.  As a result, the surface $\cX_{2,2q}$ is obtained by blowing up some number of points in the trivial genus $0$ fibration $S^2 \times S^2$.  The exact number of blowups can be calculated from the Euler characteristic.  The branch curve $\cC_{2,2q}$ has  genus $g(\cC_{2,2q}) = 2q-1$, Euler characteristic $\chi(\cC_{2,2q}) = 4 - 4q$, and self-intersection $\cC_{2,2q} \cdot \cC_{2,2q} = 8q$.  Therefore, according to Proposition \ref{prop:branched-cover-data}, the surface $\cX_{2,2q}$ has
$$ \chi(\cX_{2,2q}) = 4 + 4q \qquad \text{and} \qquad \sigma(\cX_{2,2q}) = -4q. $$
It is now clear that $\cX_{2,2q}$ is obtained by blowing up $S^2 \times S^2$ exactly $4q$ times.
\end{proof}

Applying Corollary \ref{coro:eff_branch}, we obtain efficient $(4q+2,0)$--trisections of the surfaces
$$\cX_{2,2q} \cong \CP^2 \# (4q + 1) \overline{\CP}^2.$$
However, these surfaces admit standard, efficient trisections that are obtained as the connected sum of the standard genus 1 trisections of $\CP^2$ and $\overline{\CP}^2$.  It is unclear whether these efficient trisections are equivalent.

\begin{question}
	Are the trisections of $\cX_{2,2q} \cong \CP^2 \# (4q + 1) \overline{\CP}^2$ obtained via the branched cover construction standard?
\end{question}

\subsection{Elliptic surfaces}

The double branched covers $\cX_{4,2q}$ of the pairs $(S^2 \times S^2,\cC_{4,2q})$ for $q > 0$ are elliptic surfaces.
\begin{proposition}
There is a diffeomorphism
$$ \cX_{4,2q} \cong E(q). $$
\end{proposition}
\begin{proof}
By the obvious symmetry and Proposition \ref{prop:X2}, there is a diffeomorphism
$$ \cX_{4,2} \cong \cX_{2,4} \cong \CP^2 \# 9 \overline{\CP}^2 \cong E(1). $$
We can choose the curve $\cC_{4,4}$ to be the resolution of $4$ horizontal lines and $4$ vertical lines.  We can choose the vertical lines such that $2$ are fibers over the northern hemisphere of $\CP^1$ and the other are fibers over the southern hemisphere.  The double branched cover $\cX_{4,4}$ is therefore the fiber sum of two copies of $\cX_{4,2}$ along a torus fiber $F$ of the Lefschetz fibration.  In other words
$$ \cX_{4,4} \cong \cX_{4,2} \cup_F \cX_{4,2} \cong E(1) \cup_F E(1) \cong E(2). $$
Repeating this inductively, we can see that $\cX_{4,2q}$ is the fiber sum of $\cX_{4,2q-2}$ and $\cX_{4,2}$ along a torus fiber $F$ and therefore
$$ \cX_{4,2q} \cong \cX_{4,2q-2} \cup_F \cX_{4,2} \cong E(q-1) \cup_F E(1) \cong E(q). $$
\end{proof}

Applying Corollary \ref{coro:eff_branch}, we obtain efficient trisections of elliptic surfaces.

\begin{theorem}
The elliptic surface $E(q)$ admits an efficient $(g,0)$--trisection with
$$g = 12q - 2.$$
\end{theorem}

In addition, we can also locate the torus fiber of the elliptic surface $E(q)$.

\begin{theorem}
	The torus fiber of the elliptic surface $E(q)$ admits an efficient $(12q-2;3)$--bridge trisection.
\end{theorem}

\begin{proof}
Consider the secondary shadow diagram of a bidegree $(1,0)$ sphere in Figure \ref{fig:s2xs2_proof} (Diagram VII).  It is a stabilization of the standard `green' $\alpha_1$ shadow diagram and can be destabilized by contracting the marked green arc.  Each unknot component of this sphere links the branch locus algebraically once and we isotope each disk component of the sphere to intersects the branch locus geometrically once.  

Lifting to the double branched cover, the bridge index doubles but each disk component lifts to a single disk.  The resulting surface has Euler characteristic 0 and therefore is a torus.  Furthermore, the destabilization arc lifts to two potential destabilization arcs.  We can always contract one and obtain an efficient bridge trisection.
\end{proof}

\subsection{Horikawa surfaces}

The double branched covers $\cX_{6,2q}$ of the pairs $(S^2 \times S^2,\cC_{6,2q})$ for $q \geq 2$ are known as {\it Horikawa surfaces}.  We will let $H(q) = \cX_{6,2q+2}$ denote the $q^{\text{th}}$ Horikawa surface.  

\begin{proposition}
\label{prop:Horikawa-char-numbers}
The Horikawa surface $H(q)$ has characteristic numbers
\begin{align*}
c_1^2(H(q)) &= 4q - 4 & c_2(H(q)) &=  20q + 16\\
\sigma(H(q)) &= -12q - 12 & \chi_h(H(q)) &= 2q + 1
\end{align*}
The intersection form of $H(q)$ is odd for all $q \geq 1$.
\end{proposition}

In particular, the Horikawa surfaces lie along the Noether line $c_1^2 = 2 \chi_h - 6$.  Applying Corollary~\ref{coro:eff_branch}, we obtain efficient trisections of this family.

\begin{theorem}
For all $q \geq 1$, the Horikawa surface $H(q)$ admits an efficient $(g,0)$--trisection with $$g = 20q - 6.$$
\end{theorem}

\subsection{Ruled surfaces}

A geometrically ruled surface is a $\CP^1$--bundle over a Riemann surface $\Sigma_h$ of genus $h$.  Up to diffeomorphism, there are two surfaces for each fixed $h$.  

\begin{theorem}
Let $X$ be a $S^2$--bundle over a surface $\Sigma_h$ of genus $h$.  Then $X$ admits a $(5 + 2h, 2h+1)$--trisection.
\end{theorem}

\begin{proof}
We can view $X$ as the 2--fold branched cover of $2h + 2$ disjoint copies of the $S^2$--fiber in $S^2 \times S^2$ or $S^2 \widetilde{\times} S^2 \cong \CP^2 \# \overline{\CP}^2$.  This branch set admits a bridge trisection with bridge index $2h + 2$.  Thus taking the branched cover, we obtain a trisection of $X$ whose central surface has genus $5 + 2h$.
\end{proof}

\begin{remark}
The trisection obtained in this way is {\it not} efficient.  However, Marla Williams has shown that the sphere-bundles over $\Sigma_h$ do, indeed, admit efficient $(2h+2,2h)$--trisections.
\end{remark}

%%%%%%%%%%%%%%%%%%%%%%%%%%%%%%%%%%%%%%%%%%%%%%%%%%%%%%%

%%%%%%%%%%%%%%%%%%%%%%%%%%%%%%%%%%%%%%%%%%%%%%%%%%%%%%%
%%%%%%%%%%%%%%%%%%%%%%%%%%%%%%%%%%%%%%%%%%%
%%%%%%%%%%%%%%%%%%%%%%%%%%%%%%%%%%%%%%%%%%%
\section{Complete intersections}
\label{sec:complete}
%%%%%%%%%%%%%%%%%%%%%%%%%%%%%%%%%%%%%%%%%%%
%%%%%%%%%%%%%%%%%%%%%%%%%%%%%%%%%%%%%%%%%%%

In this section, we review the construction of complete intersections in $\CP^N$ as iterated branched covers of complex curves.  Applying the results of previous sections to this construction, we prove Theorem \ref{thmx:eff_tri}(5) and Theorem \ref{thmx:eff_bridge}(3).

Fix a multi-index $\d = (d_1,\dots,d_n) \in \NN^{n}$.  Recall that a projective surface $Y$ is a {\it complete intersection} $S_{\d}$ in $\CP^{n+2}$ if it is cut out as the transverse intersection of $n$ hypersurfaces $\{X_{d_1},\dots,X_{d_n}\}$, where $X_{d_i}$ is a nonsingular hypersurface of degree $d_i$.  A {\it hyperplane section} $\cH$ of projective surface $Y$ in $\CP^N$ is a complex curve in $Y$ of the form $\cH = V \cap Y$ where $V$ is a hyperplane.  We will let $\cH_{d_0}$ denote a complex curve in the class $d_0[\cH]$ in $Y$; it is a complete intersection of $n+1$ hypersurfaces in $\CP^{n+2}$.

\begin{proposition}
Fix a multi-index $\d = (d_1,\dots,d_n) \in \NN^{n}$.  Let $\cH$ denote a hyperplane section of $S_{\d}$.
\begin{enumerate}
\item The diffeomorphism type of a complete intersection $S_{\d}$ depends only on the multi-index $\d$.
\item The smooth isotopy class of $\cH$ in $S_{\d}$ is well-defined.
\item The smooth isotopy class of $\cH_d$ in $S_{\d}$ is well-defined.
\end{enumerate}
\end{proposition}

\begin{proof}
The proofs of all three are analogous to the proof of Proposition \ref{prop:curve-isotopy}.  We can always find a 1--parameter family interpolating between two of these surfaces or curves.  Moreover, the set of singular surfaces or curves has complex codimension 1.  Thus its complement is connected and we can assume the 1--parameter family is actually an isotopy.
\end{proof}

\begin{proposition}
\label{prop:complete-branched}
Fix multi-indices $\d = (d_1\dots,d_n) \in \NN^{n}$ and $\d' = (d_0,d_1,\dots,d_n) \in \NN^{n+1}$.  The complete intersection $S_{\d'}$ is the $d_{0}$--fold cyclic branched cover of the pair $(S_{\d},\cH_{d_{0}})$.  In addition, the preimage of $\cH_{d_{0}}$ is a hyperplane section in $S_{\d'}$.
\end{proposition}

\begin{proof}
Let $F = \{f_1,\dots,f_k\}$ be a collection of polynomials such that $f_i$ is homogeneous of degree $d_i$ in $\CC[x_0,\dots,x_{n+2}]$.  For each $i = 1,\dots,n$, define the variety
$$V_i \coloneqq \{\x \in \CP^{n+2} : f_i(\x) = 0\}. $$
For a generic choice of $F$, the intersection $V_1 \cap \cdots \cap V_n$ is a complete intersection $S_{\d}$.  Furthermore, set $g \coloneqq x_0^{d_{0}} + \dots + x_{n+2}^{d_0}$ and $W \coloneqq \{\x \in \CP^{n+2} : g(\x) = 0\}$.  For a generic choice of $F$, the intersection $W \cap V_1 \cap \cdots \cap V_n$ is a curve $\cH_{d_0}$.

We can view each $f_i$ as a homogeneous polynomial in the ring $\CC[x_0,\dots,x_{n+2},x_{n+3}]$.  Define $\widetilde{V}_i$ to be the variety cut out by $f_i$ in $\CP^{n+3}$.  It is the cone on $V_i$, viewed as a variety in the hyperplane $\{x_{n+3} = 0\}$, with a singular point at $[0:\cdots:0:1]$.  Define $\widetilde{g} \coloneqq g + x_{n+3}^{d_{0}} $ and let $\widetilde{W}$ be the hypersurface cut out by $\widetilde{g}$ in $\CP^{n+3}$.  The surface $S_{\d'} \coloneqq \widetilde{W} \cap \widetilde{V}_1 \cap \cdots \cap \widetilde{V}_n$ is a complete intersection. 

Consider the projection map from the point $[0:\dots:0:1]$ onto the hyperplane $\{x_{n+3} = 0\}$.  The variety $S\widetilde{W}$ is disjoint from the point $[0:\dots:0:1]$ and consequently so it the surface $S_{\d'}$.  Thus, the projection map restricts to a map on $S_{\d'}$.  Moreover, since each $\widetilde{V}_i$ is a cone on $V_i$, the image of $S_{\d'}$ under the projection map is contained in $S_{\d}$.  The projection map is $1:1$ along $W$ and $d_0:1$ over the complement of $W$.  Thus, this defined a $d_0$-fold cyclic branched covering of $S_{\d'}$ over $S_{\d}$, ramified over the curve $\cH_{d_0}$ in $S_{\d}$.

Moreover, the curve $\cH_{d_0}$ is clearly a hyperplane section of $S_{\d'}$ as it is the intersection of $S_{\d'}$ with the hyperplane $\{x_{n+3} = 0\}$.  
\end{proof}

Using the branched covering construction of complete intersections in Proposition \ref{prop:complete-branched}, we can inductively apply Proposition \ref{prop:branched-cover-data} to obtain the following topological data (see also \cite{GK}).

\begin{proposition}
Fix $\d = (d_1,\dots,d_n) \in \NN^{n}$, let $S_{\d}$ be a complete intersection in $\CP^{n+2}$ of multidegree $\d$ and let $\cH$ a hyperplane section of $S_{\d}$.
\begin{enumerate}
\item The surface $S_{\d}$ is a simply-connected surface of degree $d = \prod d_i$ with characteristic numbers
\begin{align*}
\chi(S_{\d}) &= \left( \frac{(n+2)(n+3)}{2} - (n+3) \cdot \sum_{i = 1}^n d_i + \sum_{i = 1}^n d_i^2  + \sum_{1 \leq i < j \leq n} d_i d_j \right) \cdot \prod_{i = 1}^n d_i \\
\sigma(S_{\d}) &= \frac{1}{3} \left( n+ 3 - \sum_{i = 1}^n d_i^2 \right) \cdot \prod_{i = 1}^n d_i \\
c_1(S_{\d}) &= \left( n + 3 - \sum_{i = 1}^n d_i \right) \zeta \\
c_1^2(S_{\d}) &= \left( n + 3 - \sum_{i = 1}^n d_i \right)^2 \prod_{i = 1}^n d_i 
\end{align*}
where $\zeta = i^*(\alpha)$ is the pullback of the generator $\alpha$ of $H^*(\CP^{n+2};\ZZ) \cong \ZZ[\alpha]/\langle \alpha^{n+3} \rangle$ under the inclusion $i: S_{\d} \hookrightarrow \CP^{n+2}$.
\item The hyperplane section $\cH \subset S_{\d}$ is a connected curve Poincare dual to $\zeta = i^*(\alpha)$, satisfying 
$$ g(\cH) = \frac{1}{2} \left( \sum_{i = 1}^n d_i - n \right) \cdot \prod_{i = 1}^n d_i \qquad \text{and} \qquad \cH \cdot \cH = \prod_{i = 1}^n d_i.$$
\end{enumerate}
\end{proposition}

Starting with Theorem \ref{thrm:projective-curve-efficient}, we can apply Corollary \ref{coro:eff_branch} and Theorem \ref{thm:push-off_efficient} inductively using the branched cover construction to obtain the following result.

\begin{theorem}
Let $S_{\d}$ be a complete intersection in $\CP^{n+2}$ and $\cH$ a hyperplane section.
\begin{enumerate}
\item The surface $S_{\d}$ admits an efficient $(g,0)$--trisection, with
\[g = -2 + \left( \frac{(n+2)(n+3)}{2} - (n+3) \cdot \sum_{i = 1}^n d_i + \sum_{i = 1}^n d_i^2  + \sum_{1 \leq i < j \leq n} d_i d_j \right) \cdot \prod_{i = 1}^n d_i \]
\item The pair $(S_{\d},\cH)$ admits an efficient $(g;b)$--bridge trisection, with
\[b = 1 +\left( \sum_{i = 1}^n d_i - n \right) \cdot \prod_{i = 1}^n d_i \]
\end{enumerate}
\end{theorem}

%%%%%%%%%%%%%%%%%%%%%%%%%%%%%%%%%%%%%%%%%%%%%%%%%%%%%%%

%%%%%%%%%%%%%%%%%%%%%%%%%%%%%%%%%%%%%%%%%%%%%%%%%%%%%%%
%%%%%%%%%%%%%%%%%%%%%%%%%%%%%%%%%%%%%%%%%%%%%%%%%%%%%%%%%%%%%%%%%%%%%
%%%%%%%%%%%%%%%%%%%%%%%%%%%%%%%%%%%%%%%%%%%%%%%%%%%%%%%%%%%%%%%%%%%%%
\section{Proper transform and variously trisecting $K3$}
\label{sec:K3}
%%%%%%%%%%%%%%%%%%%%%%%%%%%%%%%%%%%%%%%%%%%%%%%%%%%%%%%%%%%%%%%%%%%%%
%%%%%%%%%%%%%%%%%%%%%%%%%%%%%%%%%%%%%%%%%%%%%%%%%%%%%%%%%%%%%%%%%%%%%

In the section, we describe several well-known constructions of K3 as a branched cover over a rational complex surface.  By taking branched covers of a bridge trisection of the branch locus, we obtain several trisections of K3.    All but one of these constructions give an efficient $(22,0)$--trisection of K3.  This proves Theorem \ref{thmx:K3}.  In addition, we discuss blowing up and proper transform of a surface from the perspective of trisections.

%%%%%%%%%%%%%%%%%%%%%%%%%%%%%%%%%%%%%%%%%%%%%%%%%%%%%%%%%%%%%%%%%%%%%
\subsection{Blowing up}
\label{subsec:blowup}
%%%%%%%%%%%%%%%%%%%%%%%%%%%%%%%%%%%%%%%%%%%%%%%%%%%%%%%%%%%%%%%%%%%%%
In this subsection, we describe how to blow up a bridge trisection to obtain a bridge trisection of the proper transform. First, we recall the definition of the oriented connected sum of pairs, which we extend to the setting of bridge trisections.

\begin{definition}
	Let $(X,\cK)$ and $(Y,\cJ)$ be oriented knotted surfaces.  Choose points $x \in \cK$ and $y \in \cL$ and tubular neighborhoods $\nu_X(x),\nu_Y(y)$.  We also obtain tubular neighborhoods $\nu_{\cK}(x) = \nu_X(x) \cap \cK$ and $\nu_{\cL}(y) = \nu_Y(y) \cap \cL$.  Choose an orientation-reversing diffeomorphism $\Phi\colon \del \nu_{X}(x) \rightarrow \del \nu_{Y}(y)$ that restricts to an orientation-reversing diffeomorphism $\phi\colon \del \nu_{\cK}(x) \rightarrow \del \nu_{\cL}(y)$.
	
	The {\it oriented connected sum} of $(X,\cK)$ and $(Y,\cJ)$ is the pair
	$$(X,\cK) \# (Y,\cJ) = (X \# Y, \cK \# \cL) = (X\cup_\Phi Y,\cK\cup_\phi\cJ).$$
	
	Suppose $(X,\cK)$, and $(Y,\cJ)$ are equipped with bridge trisections $\cT_\cK$ and $\cT_\cJ$, respectively. We can assume that $x \in \Sigma_{\cT_\cK}\cap\cK$ with $\sigma(x)=1$ and $y \in \Sigma_{\cT_\cJ}\cap\cJ)$ with $\sigma(y)=-1$. Then we naturally obtain a bridge trisection $\cT_{\cK\#\cJ}$ for $(X\#Y,\cK\#\cJ)$, called the {\it oriented connect sum}, whose central surface is $\Sigma = \Sigma_X \underset{x = y}\# \Sigma_Y$.
\end{definition}

A special case of the oriented connected sum operation is the proper transform.  The \emph{proper transform} of a knotted surface $(X,\cK)$, is the knotted surface $(X,\cK)\#(\overline\CP^2,\overline\cC_1)$.  

\begin{proposition}
\label{prop:transform}
	Suppose that the knotted surface $(X,\cK)$ has a $(g,\bold k;b,\bold c)$--bridge trisection $\cT_\cK$.  We can construct a $(g+1,\bold k;b,\bold c)$--bridge trisection for the proper transform $(X \# \overline{\CP}^2, \cK\#\overline\cC_1)$, by performing the local modification shown in Figure~\ref{fig:transform} at a point $x\in\Sigma_X\cap\cK$ with $\epsilon(x)=1$.
\end{proposition}

\begin{proof}
	We form the connected sum of the trisection for $(X,\cK)$ with that of $(\overline\CP^2,\overline\cC_1)$ by puncturing each at a point where the knotted surface meets the trisection surface.  In the former bridge trisection, we choose a point $x$ with $\epsilon(x)=1$, as in the first two frames of Figure~\ref{fig:transform}.  In the latter bridge trisection, we choose the unique point $x$ with $\epsilon(x)=-1$, as in Figure~\ref{fig:punc_line}.  The bridge trisection for the proper transform is the union of these two punctured bridge trisections along the punctures, as shown in the third frame of Figure~\ref{fig:transform}.
\end{proof}

\begin{figure}[h!]
\centering
\includegraphics[width=.5\textwidth]{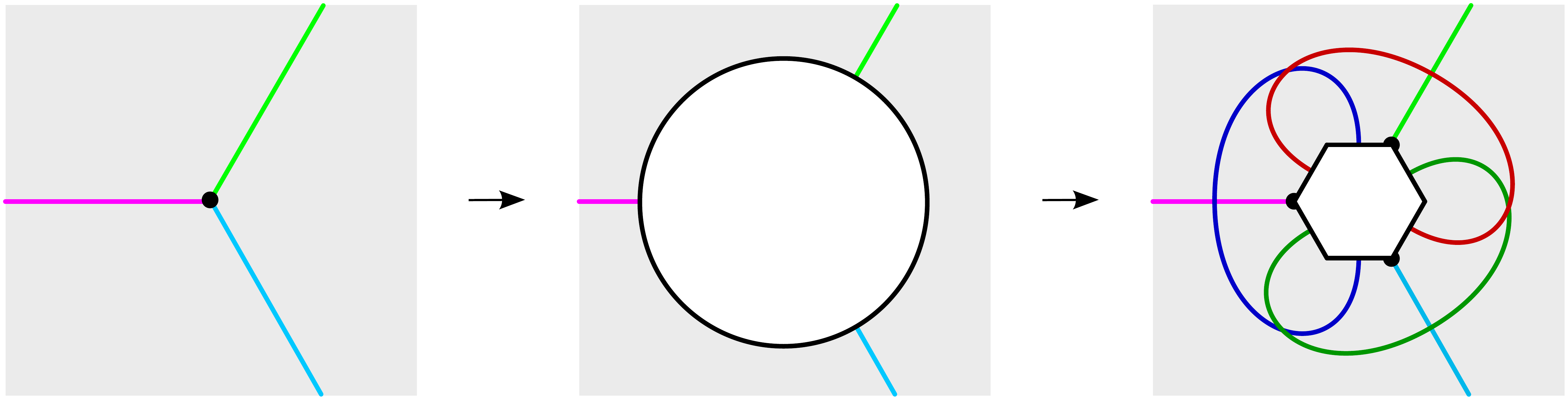}
\caption{The local modification at a point $x\in\Sigma_X\cap\cK$ with $\epsilon(x)=1$ representing the proper transform of $(X,\cK)$ at $x$.}
\label{fig:transform}
\end{figure}

\begin{figure}[h!]
\centering
\includegraphics[width=.5\textwidth]{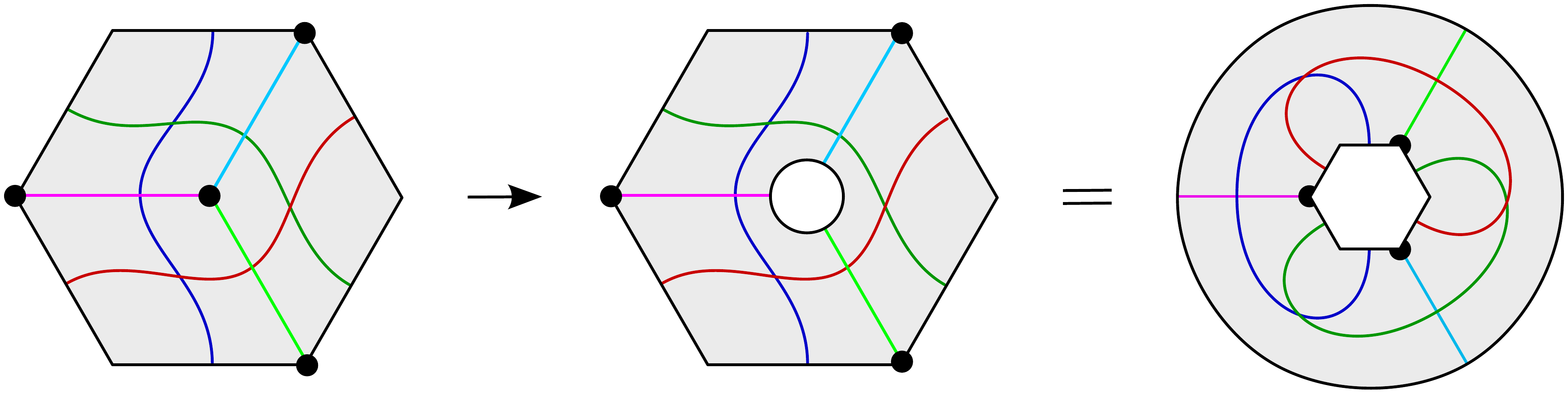}
\caption{The process of puncturing the shadow diagram for $(\overline\CP^2,\overline\cC_1)$ at the point $y\in\Sigma\cap\overline\cC_1$ with $\epsilon(y)=-1$.  The third frame}
\label{fig:punc_line}
\end{figure}

%%%%%%%%%%%%%%%%%%%%%%%%%%%%%%%%%%%%%%%%%%%%%%%%%%%%%%%%%%%%%%%%%%%%%
\subsection{K3 as branched cover of a rational surface}
%%%%%%%%%%%%%%%%%%%%%%%%%%%%%%%%%%%%%%%%%%%%%%%%%%%%%%%%%%%%%%%%%%%%%

Let $0\leq n\leq 9$, and let $\widetilde{\cE} \subset \CP^2 \# n \overline{\CP}^2$ be a nonsingular complex curve dual to the anticanonical class $K^*$.  Up to isotopy, we can choose it to be the proper transform of a degree 3 elliptic curve $\cE = \cC_3$ in $\CP^2$. Note that by Proposition~\ref{prop:transform}, the knotted surface $(\CP^2\#n\overline\CP^2)$ admits an efficient $(n+1;3)$--bridge trisection and has self-intersection $9-n$.

Let $\widetilde{\cE}_2$ denote a nonsingular complex curve dual to $2K^*$.  This curve is obtained by resolving the transverse intersections of the union of two copies of $\widetilde\cE$.  For example, in $\CP^2$, $\widetilde{\cE}_2$ is the sextic.

\begin{lemma}
	Let $\widetilde{\cE}_2$ denote a nonsingular complex curve dual to $2K^*$ in $\CP^2 \# n \overline{\CP}^2$.
	\begin{enumerate}
		\item For $0 \leq n \leq 8$, the curve $\widetilde{\cE}_2$ is connected with genus $g = 10-n$.
		\item For $n = 9$, the curve $\widetilde{\cE}_2$ consists of two disjoint copies of $\widetilde{\cE}$.
	\end{enumerate}
\end{lemma}

\begin{proof}
	The self-intersection number of $\widetilde{\cE}$ is
	\[ \widetilde{\cE} \cdot \widetilde{\cE} = c_1(\CP^2 \# n \overline{\CP}^2)^2 = 9 - n \]
	Thus, to obtain a curve $\widetilde{\cE}_2$, we can take two copies of $\widetilde{\cE}$ intersecting transversely in $9 - n$ points.  Resolving these intersection points we obtain a complex curve of the required genus.  When $n = 9$, the curves do not intersect.
\end{proof}

\begin{theorem}
\label{thm:fiber_eps}
	The knotted surface $(\CP^2 \# n \overline{\CP}^2, \widetilde{\cE}_2)$ admits an efficient $(n+1;21-2n)$--bridge trisection.
\end{theorem}

\begin{proof}
	Since $(\CP^2 \# n \overline{\CP}^2, \widetilde{\cE})$ admits an efficient $(n+1;3)$--bridge trisection, the result follows from Theorem~\ref{thm:push-off_efficient} and the fact that $\widetilde\cE$ has self-intersection $9-n$.
\end{proof}

\begin{proposition}
\label{prop:K3-rational-double}
	The double branched cover of the pair $(\CP^2 \# n \overline{\CP}^2, \widetilde{\cE}_2)$ is K3.
\end{proposition}

\begin{proof}
	Let $X$ be the double branched cover of the pair $(\CP^2 \# n \overline{\CP}^2, \widetilde{\cE}_2)$.  According to Proposition \ref{prop:branched-cover-data}, the anticanonical class is the pullback of the class
	\[c_1(\CP^2 \# n \overline{\CP}^2) - \frac{1}{2}PD(\widetilde{\cE}_2) = 0\]
	by the branched covering map.  Therefore, $c_1(X) = 0$.  It is well-known that K3 is the only simply-connected complex surface with trivial canonical class.  Therefore $X$ is diffeomorphic to K3.
\end{proof}

Each of these nine branched coverings give efficient trisections of K3, which may or may not be isotopic.  Moreover, the lift of the branch locus is in efficient bridge trisected position in each of the nine examples.

%%%%%%%%%%%%%%%%%%%%%%%%%%%%%%%%%%%%%%%%%%%%%%%%%%%%%%%%%%%%%%%%%%%%%
\subsection{Elliptic surfaces}
%%%%%%%%%%%%%%%%%%%%%%%%%%%%%%%%%%%%%%%%%%%%%%%%%%%%%%%%%%%%%%%%%%%%%

Take a pair of nonsingular elliptic curves $C_1,C_2$ in $\CP^2$.  The resulting Lefschetz pencil of cubics has nine basepoints, consisting of points of intersection between $C_1$ and $C_2$.  Blowing up at these 9 points yields an elliptic surface $E(1) \coloneqq \CP^2 \# 9 \overline{\CP}^2$ with a genus 1 Lefschetz fibration $ \pi: E(1) \rightarrow \CP $. 

\begin{theorem}
Let $\cE$ denote a generic fiber of the elliptic fibration $E(1)$ and let $\cE_k$ denote $k$ disjoint, parallel copies of $\cE$.
\begin{enumerate}
\item The pair $(E(1),\cE)$ admits an efficient $(10;3)$--bridge trisection.
\item The pair $(E(1),\cE_k)$ admits a $(10,0;3k,k)$--bridge trisection.
\end{enumerate}
\end{theorem}

\begin{proof}
	The process of blowing up the cubic $\cC_3$ in $\CP^2$ at nine points is shown in Figure~\ref{fig:cubic_E1}.  By Theorem~\ref{thm:fiber_eps}, we have part (1).  Since $\cE$ has self-intersection zero, the $k$ copies of $\cE$ in $\cE_k$ are disjoint, and automatically bridge trisected as a $k$--component surface-link.
\end{proof}

\begin{figure}[h!]
\centering
\includegraphics[width=.5\textwidth]{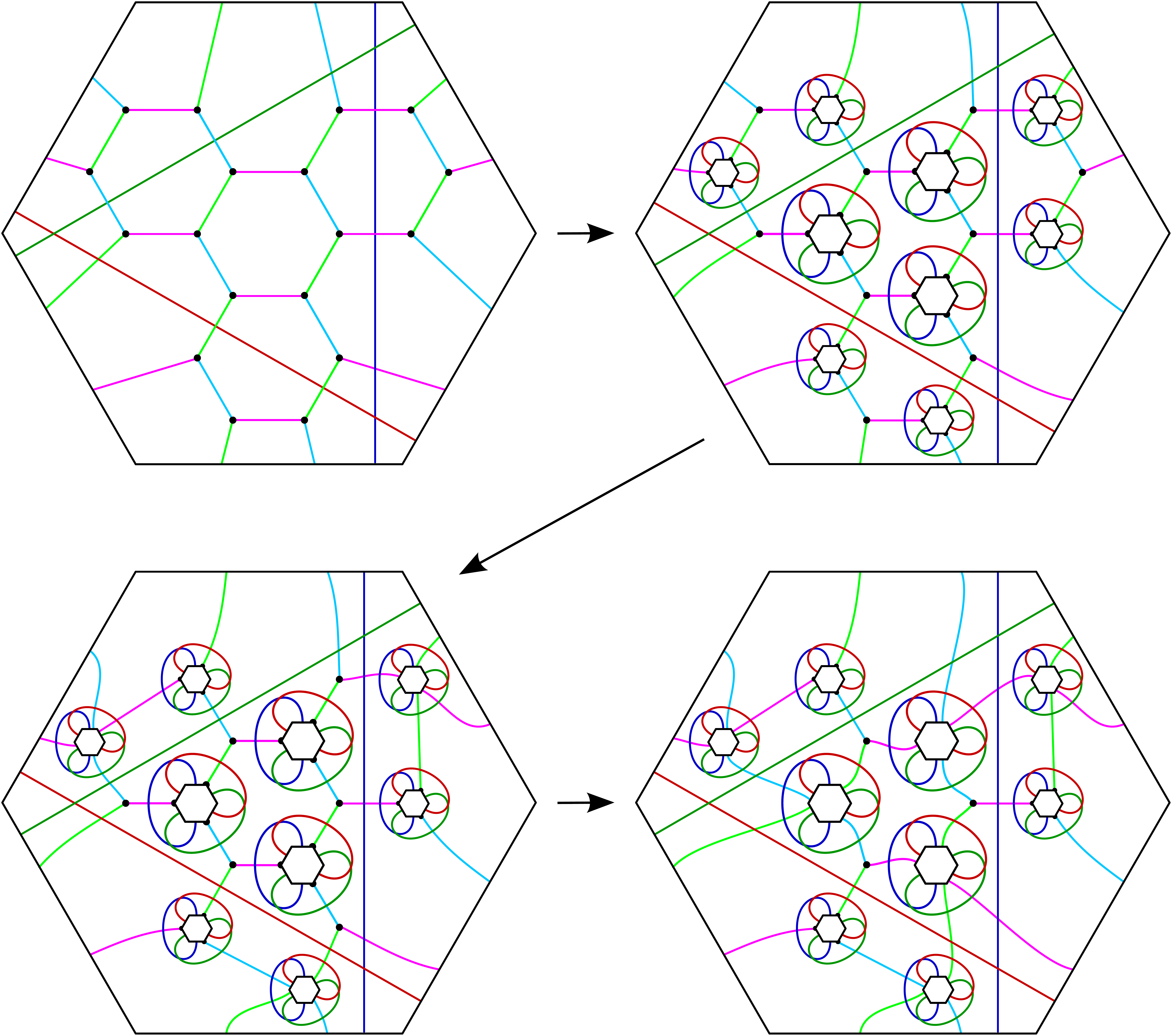}
\caption{The process of blowing up $\CP^2$ along nine points on the cubic curve to obtain $E(1)$.}
\label{fig:cubic_E1}
\end{figure}

We can now construct the elliptic surface $E(n)$ by pulling back the fibration $\pi$ by the map $z^n: \CP^1 \rightarrow \CP^1$ of the base:
\[
\xymatrix{
E(n)  \ar[r] \ar[d]^{\pi_n} & E(1) \ar[d]^{\pi} \\
\CP^1 \ar[r]^{z^n} & \CP^1
} \]
A generic fiber is a nonsingular elliptic curve with self-intersection number 0.  By construction, the fibers of $\pi$ over $[1:0]$ and $[0:1]$ are the proper transforms of $C_1$ and $C_2$.  The following proposition is clear from the construction.

\begin{proposition}
The elliptic surface $E(n)$ is the $n$--fold cyclic branched cover of $E(1) \cong \CP^2 \# 9 \overline{\CP}^2$ over a disjoint pair of generic torus fibers of the fibration $\pi: E(1) \rightarrow \CP^1$.
\end{proposition}

For $n\geq 2$, this allows us to obtain new trisections of $E(n)$, which are inefficient since the branch locus is disconnected.

%%%%%%%%%%%%%%%%%%%%%%%%%%%%%%%%%%%%%%%%%%%%%%%%%%%%%%%

\newpage
%%%%%%%%%%%%%%%%%%%%%%%%%%%%%%%%%%%%%%%%%%%%%%%%%%%%%%%
%%%%%%%%%%%%%%%%%%%%%%%%%%%%%%%%%%%%%%%%%%%%%%%%%%%%%%%
\section{Menagerie of diagrams}
\label{sec:menagerie}
%%%%%%%%%%%%%%%%%%%%%%%%%%%%%%%%%%%%%%%%%%%%%%%%%%%%%%%
%%%%%%%%%%%%%%%%%%%%%%%%%%%%%%%%%%%%%%%%%%%%%%%%%%%%%%%

We have gathered a handful of figures referenced in the text into this final section for expositional clarity.

\begin{figure}[h!]
\centering
\includegraphics[width=\textwidth]{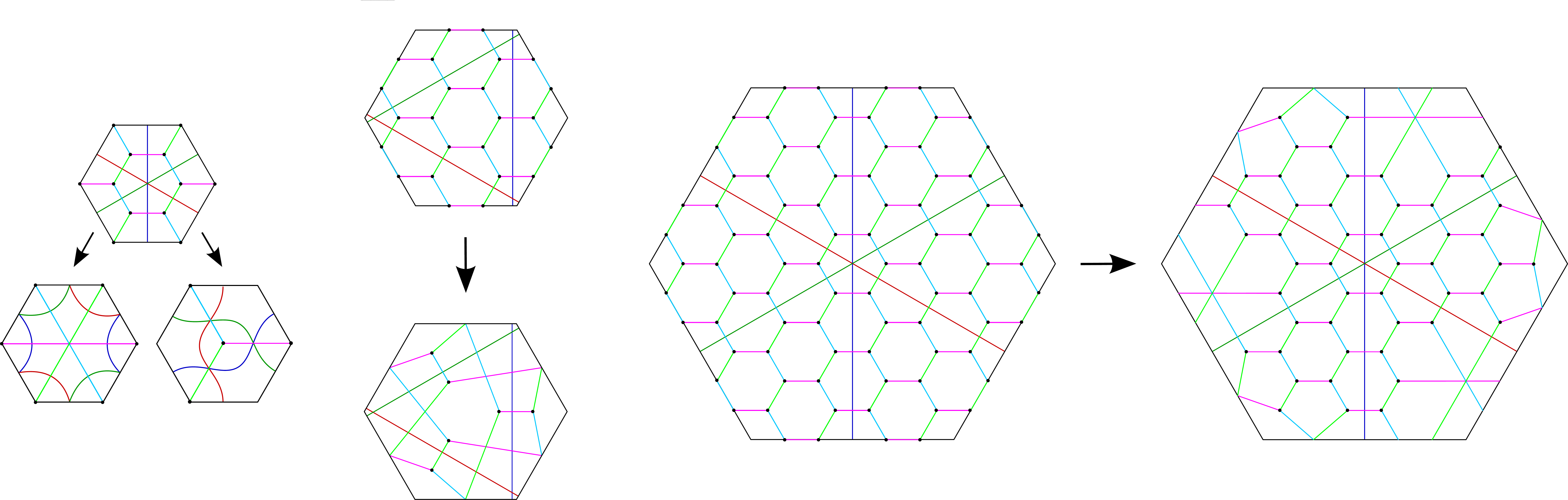}
\caption{Examples of destabilizations of the $(1,0;d^2,d)$--bridge trisections for the complex curves of degree $d$ in $\CP^2$ into efficient bridge trisections with bridge number $b = d^2-3d+3$ with the convention that the central trisection surface (a torus) is represented by a hexagon with opposite sides identified. Shown are the instances of (Left) $d=2$, (Center) $d=3$, and (Right) $d=6$.}
\label{fig:Destabs_Hex}
\end{figure}

\vfill

\begin{figure}[h!]
\centering
\includegraphics[width=\textwidth]{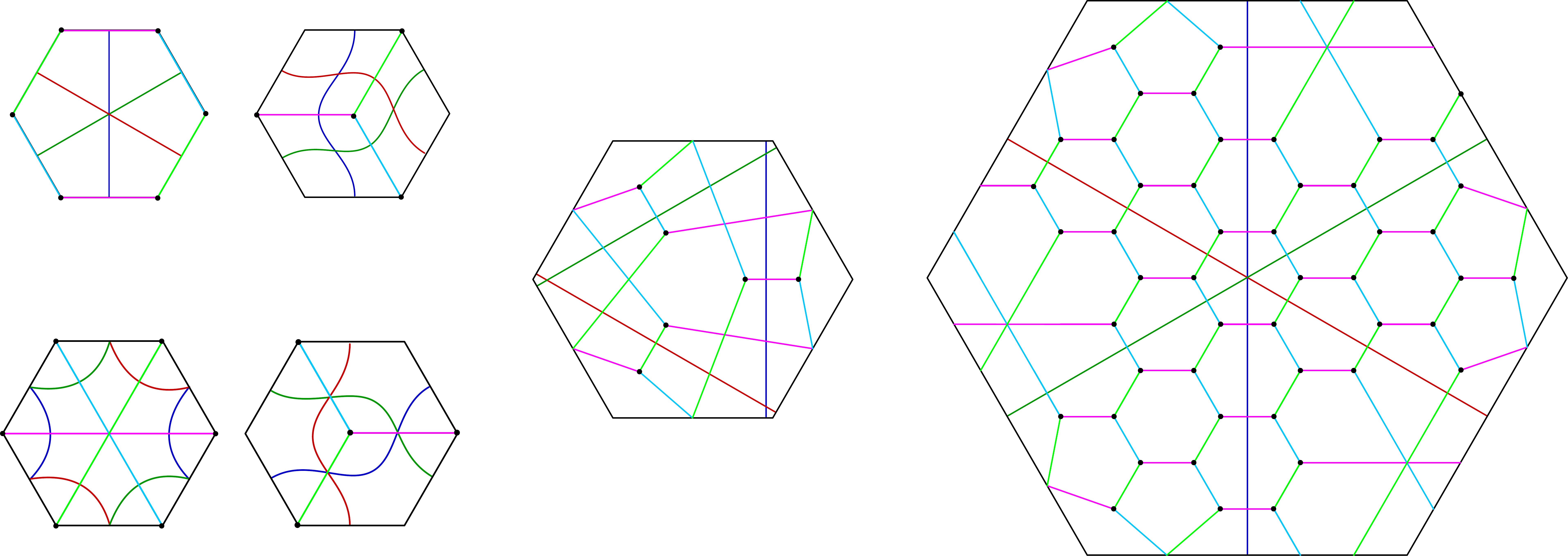}
\caption{Examples of efficient bridge trisections for complex curves in $\CP^2$ with the central trisection surfaces (a torus) represented as a hexagon with opposite edges identified. Shown are the instances of (Top-Left) $d=1$, (Bottom-Left) $d=2$, (Center) $d=3$, and (Right) $d=6$.}
\label{fig:Curves_Plane_Hex}
\end{figure}

\clearpage 

\begin{figure}[h!]
\centering
\includegraphics[width=.7\textwidth]{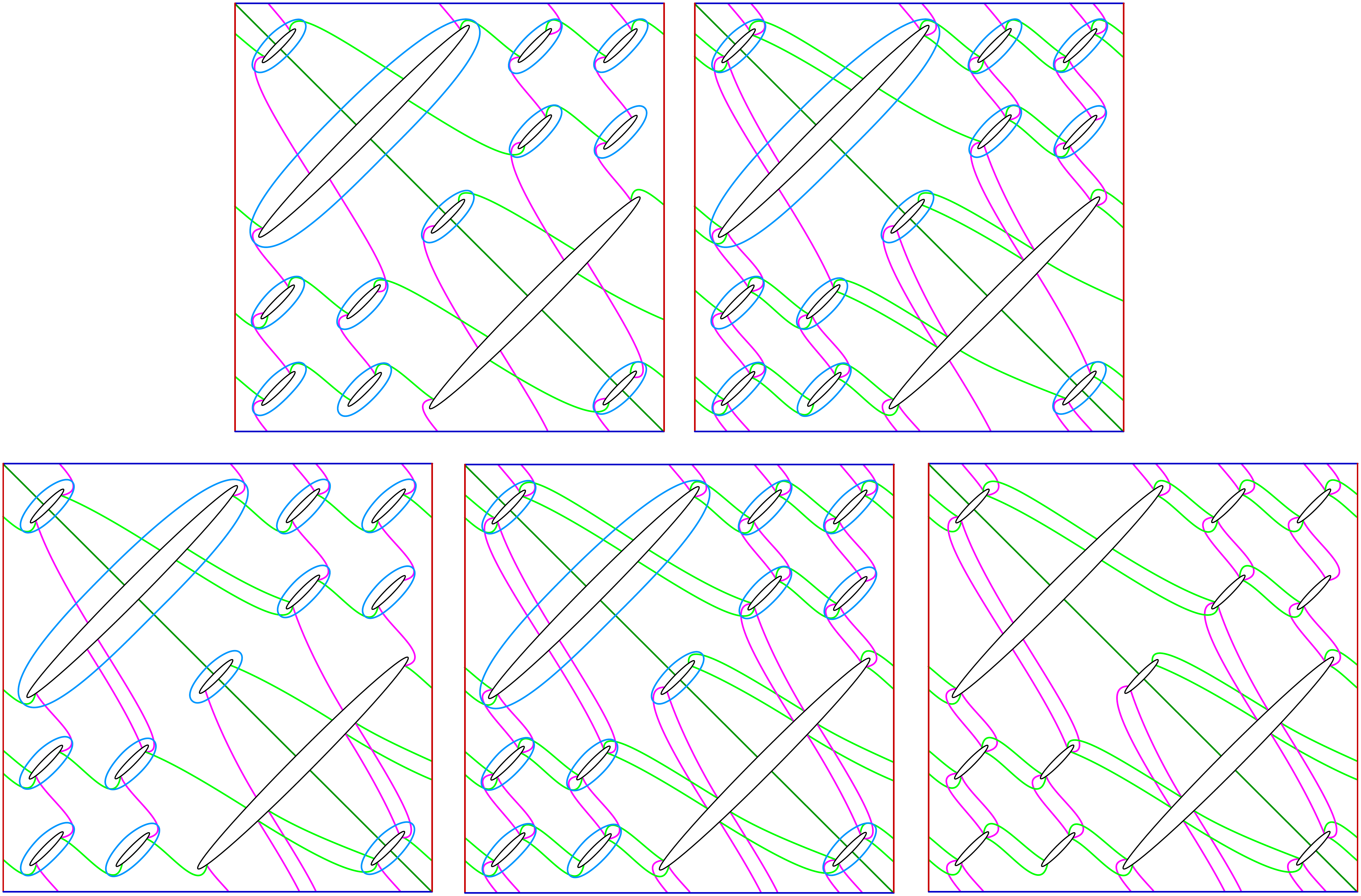}
\caption{A $(53,0)$--trisection of $S_5$, thought of as the 5--fold cover of $\CP^2$ branched along the quintic $\cC_5$.  Each square corresponds to a torus once opposite edges are identified.  The northwestern edge of each ellipse is identified with the southeastern edge of the corresponding ellipse in the clockwise-adjacent square.}
\label{fig:S5}
\end{figure}

\vfill

\begin{figure}[h!]
\centering
\includegraphics[width=.7\textwidth]{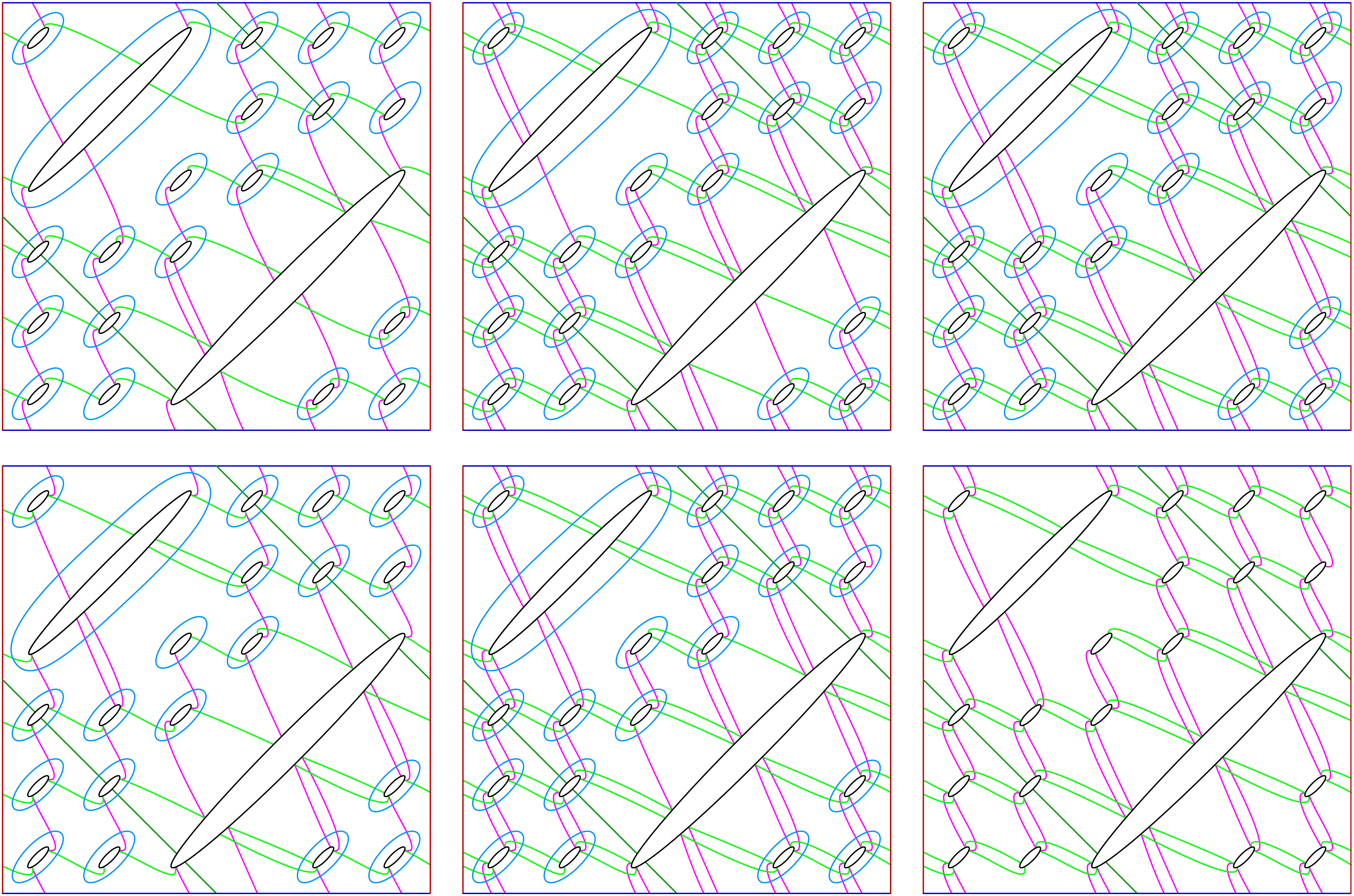}
\caption{A $(106,0)$--trisection of $S_6$, thought of as the 6--fold cover of $\CP^2$ branched along the sextic $\cC_6$.  Each square corresponds to a torus once opposite edges are identified.  The northwestern edge of each ellipse is identified with the southeastern edge of the corresponding ellipse in the clockwise-adjacent square.}
\label{fig:S6}
\end{figure}

\clearpage

\begin{figure}[h!]
\centering
\includegraphics[width=.6\textwidth]{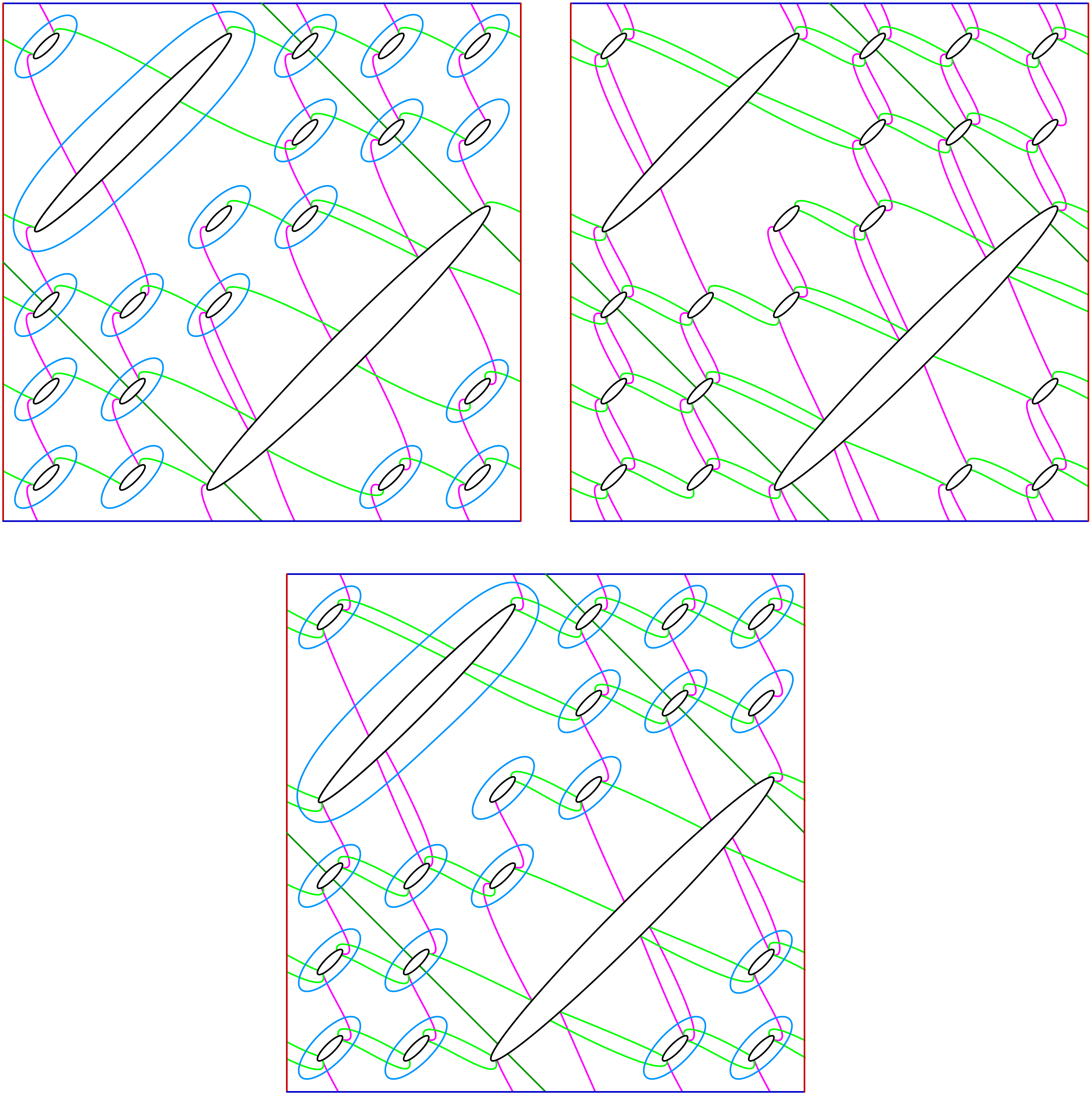}
\caption{A $(43,0)$--trisection of the 3--fold cover $\cQ_{6,3}$ of $\CP^2$ branched along the sextic $\cC_6$.  Each square corresponds to a torus once opposite edges are identified. The northwestern edge of each ellipse is identified with the southeastern edge of the corresponding ellipse in the clockwise-adjacent square.}
\label{fig:Q_6_3_square}
\end{figure}

\vfill

\begin{figure}[h!]
\centering
\includegraphics[width=.75\textwidth]{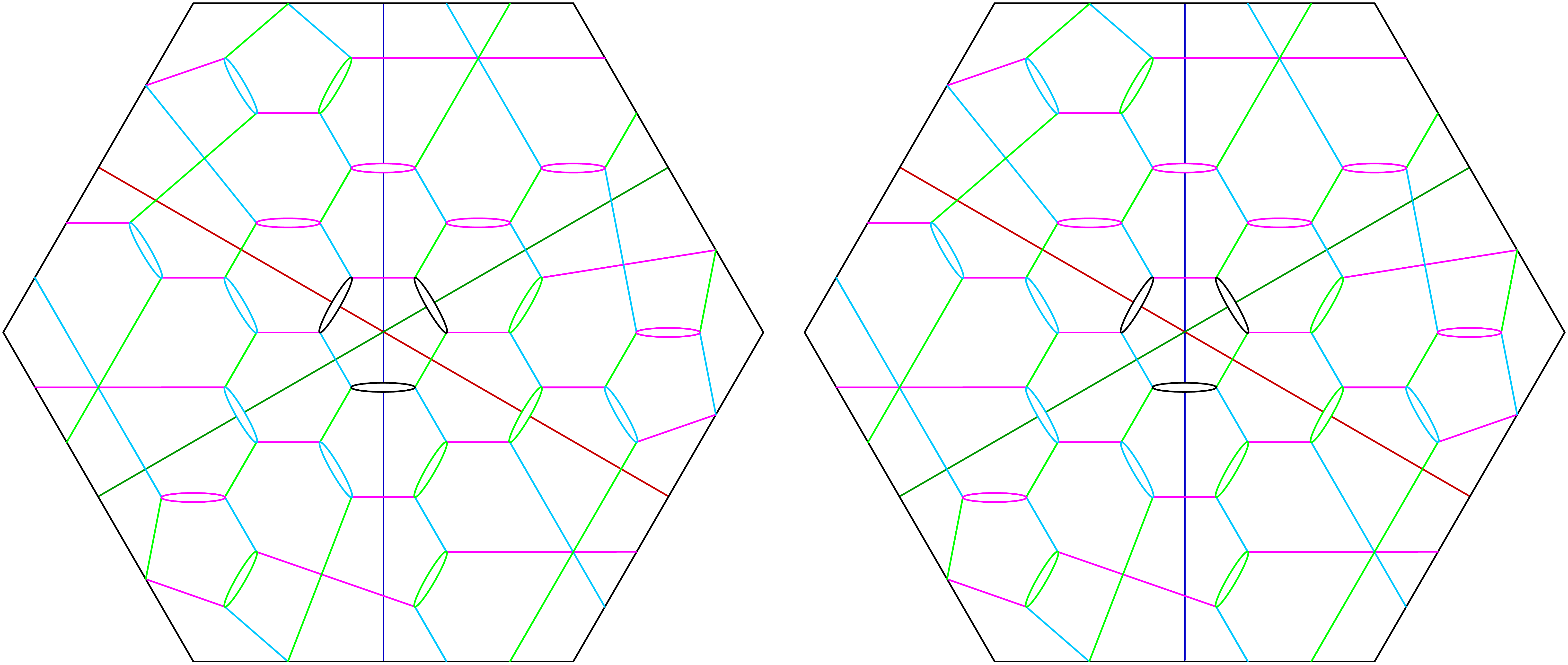}
\caption{A $(22,0)$--trisection of $K3$, thought of as the 2--fold cover $\cQ_{6,2}$ of $\CP^2$ branched along the sextic $\cC_6$.  Each hexagon corresponds to a torus once opposite edges are identified.  Each ellipse in the left hexagon is identified with the corresponding ellipse in the right hexagon via a reflection across its major axis.}
\label{fig:Q_6_2_hex}
\end{figure}

\clearpage

%%%%%%%%%%%%%%%%%%%%%%%%%%%%%%%%%%%%%%%%%%%%%%%%%%%%%%%
\bibliographystyle{amsalpha}
\bibliography{References}

%%%%%%%%%%%%%%%%%%%%%%%%%%%%%%%%%%%%%%%%%%%%%%%%%%%%%%%

\end{document}